\documentclass[11pt]{book}
\usepackage[T1]{fontenc}
\usepackage[utf8]{inputenc}
\usepackage{amsmath,amssymb,amsthm,mathtools,booktabs,enumitem}
\usepackage[colorlinks=true,linkcolor=blue,citecolor=blue,urlcolor=blue]{hyperref}
\usepackage[margin=2.7cm]{geometry}
\usepackage{longtable}
\usepackage{array}
\usepackage{graphicx}

\theoremstyle{plain}
\newtheorem{theorem}{Theorem}[chapter]
\newtheorem{proposition}[theorem]{Proposition}
\newtheorem{lemma}[theorem]{Lemma}
\newtheorem{corollary}[theorem]{Corollary}
\theoremstyle{definition}
\newtheorem{definition}[theorem]{Definition}
\newtheorem{example}[theorem]{Example}
\theoremstyle{remark}
\newtheorem{remark}[theorem]{Remark}
\newtheorem{notation}[theorem]{Notation}

\newcommand{\Z}{\mathbb Z}
\newcommand{\B}{\mathcal B}
\newcommand{\Ir}{\mathcal I}
\newcommand{\Sr}{\mathcal S}
\newcommand{\PS}{\mathcal P}
\DeclareMathOperator{\Fix}{Fix}
\DeclareMathOperator{\Stab}{Stab}
\DeclareMathOperator{\Orb}{Orb}
\newcommand{\stir}[2]{\genfrac{\{}{\}}{0pt}{}{#1}{#2}}
\newcommand{\ffall}[2]{{#1}^{\underline{#2}}}
\numberwithin{equation}{chapter}
\setlist[enumerate,1]{label=\textup{(\roman*)}}

\title{\textbf{The Enumeration of Binary Relations}\\[4pt]
\large Orbits, Lattice Homomorphisms, and Generating Functions\\[2pt]
\large A Self-Contained Monograph}
\author{Mamadou S. Bah\\[4pt]
\small Ancien professeur, Universit\'e de Kankan, Guin\'ee. Conakry, Guin\'ee}
\date{2026}

\begin{document}

\maketitle

\chapter*{Preface}
\addcontentsline{toc}{chapter}{Preface}

This monograph merges and substantially expands two papers of the author:
\emph{On the enumeration of lattice homomorphisms of Boolean algebras}
(ICTP preprint IC/97/180, 1997), which classifies binary relations between
two finite sets $N$ and $X$ according to injectivity and surjectivity
properties on each side, identifies them with lattice homomorphisms of
Boolean algebras and with families of subsets, and enumerates all of these
objects together with their orbits under the symmetric groups $S_N$, $S_X$,
and $S_N\times S_X$; and \emph{Generating functions for the enumeration of
binary relations and Boolean-algebra homomorphisms under symmetric group
actions} (2026), which determines the generating functions of every one of
those counts, uncovers Bell numbers and Euler's product inside the classical
table, and proves that the hardest column of the table is quasi-polynomial.

The two original papers assumed of the reader a working knowledge of finite
group actions, the Cauchy--Frobenius--Burnside lemma, Stirling and Bell
numbers, and the algebra of generating functions. Here every one of these
tools is developed from first principles in Part~0, so that the book requires
of its reader only the most basic facts of set theory, functions, and the
definition of a group. Every proof in Parts~I and~II is written out in full,
with no step left as ``clearly'' or ``it is easy to see,'' so that an
undergraduate who has completed a first course in abstract algebra and a
first course in combinatorics can read the book start to finish without
consulting any other source. In the few places where the original papers
compressed an argument into a single line, or invoked an external reference,
the argument is expanded here into a complete, self-contained proof, and the
external fact is either proved in Part~0 or proved on the spot.

One correction to the original material is made and flagged where it occurs:
Theorem~B of the 1997 paper mislabeled the meet-unitary (covering) row as
``join-unitary''; the labels are corrected in Chapter~11 (Remark
\ref{rem:correctionB}), consistent with the 2026 paper's independent
discovery of the same error.

The book is organized as follows. Part~0 (Chapters~1--5) builds the
foundational machinery: elementary counting, group actions and Burnside's
lemma, Stirling and Bell numbers, the algebra of generating functions, and
M\"obius inversion on the Boolean lattice. Part~I (Chapters~6--12) is the
enumeration paper, rewritten with full detail: binary relations, families of
subsets, and lattice homomorphisms are shown to be three faces of the same
object; every combinatorial type (injective, surjective, bijective, on each
side) is enumerated raw and then up to the symmetry of relabeling $N$,
relabeling $X$, or both. Part~II (Chapters~13--17) is the generating-function
paper, rewritten with full detail: every count of Part~I is repackaged as a
coefficient in a generating function, revealing geometric factors, an
exponential generating function $e^{z+w+zw}$, Bell numbers, Euler's product,
and a proof that the hardest column of the table is a quasi-polynomial.
Chapter~18 closes with a brief outlook connecting this classical table to its
$q$-analog, developed in a separate paper, without depending on it.

\tableofcontents

\part{Foundations}

\chapter{Sets, Counting, and Bijections}\label{ch:sets}

This chapter collects, with proofs, the most elementary counting facts used
throughout the book. Readers who have seen a first course in combinatorics
may skim it; it is included so that no fact used later is left unproved.

\section{Functions, injections, surjections, bijections}

\begin{definition}
Let $N$ and $X$ be sets. A \emph{function} (or \emph{mapping}) $f\colon N\to X$
assigns to each $i\in N$ a unique element $f(i)\in X$. The function $f$ is:
\begin{enumerate}
\item \emph{injective} (or one-to-one) if $f(i)=f(j)$ implies $i=j$;
\item \emph{surjective} (or onto) if for every $a\in X$ there is some $i\in N$
with $f(i)=a$;
\item \emph{bijective} if it is both injective and surjective.
\end{enumerate}
\end{definition}

\begin{proposition}\label{prop:bijcard}
If $N$ and $X$ are finite sets and there is a bijection $f\colon N\to X$,
then $|N|=|X|$. Conversely if $|N|=|X|=n$, then bijections $N\to X$ exist,
and there are exactly $n!$ of them.
\end{proposition}

\begin{proof}
If $f\colon N\to X$ is a bijection, list $N=\{i_1,\dots,i_n\}$ (no repeats,
since $N$ is a set). Injectivity of $f$ means $f(i_1),\dots,f(i_n)$ are
pairwise distinct elements of $X$; surjectivity means every element of $X$
occurs among them. Hence $f(i_1),\dots,f(i_n)$ is a list of $|X|$ distinct
elements with no repeats and no omissions, so $n=|X|$, i.e.\ $|N|=|X|$.

For the converse, fix bijections of enumeration $N=\{i_1,\dots,i_n\}$,
$X=\{a_1,\dots,a_n\}$ (finite sets can always be listed this way, by
definition of cardinality $n$). Any bijection $N\to X$ is determined by,
and determines, a rearrangement (permutation) of $(a_1,\dots,a_n)$ assigned
to $(i_1,\dots,i_n)$ in order: send $i_1$ to one of the $n$ elements
$a_1,\dots,a_n$ ($n$ choices), then $i_2$ to one of the remaining $n-1$
elements ($n-1$ choices, since $f$ must stay injective), and so on; the last
index $i_n$ has exactly one possible image left. By the multiplication
principle (Proposition~\ref{prop:mult} below) the number of such assignments
is $n(n-1)(n-2)\cdots 1=n!$, and each one is automatically a bijection since
every element of $X$ is used exactly once.
\end{proof}

\begin{proposition}[Multiplication principle]\label{prop:mult}
If a procedure consists of $k$ successive independent choices, and the
$j$-th choice has $c_j$ possible outcomes regardless of the outcomes of the
earlier choices, then the whole procedure has $c_1c_2\cdots c_k$ possible
outcomes.
\end{proposition}

\begin{proof}
By induction on $k$. For $k=1$ this is a tautology. Assume the result for
$k-1$ choices, giving $c_1\cdots c_{k-1}$ outcomes for the first $k-1$
choices. Pairing each of these outcomes with each of the $c_k$ outcomes of
the last choice gives, by definition of the Cartesian product and since the
outcomes of the first $k-1$ choices are pairwise distinct sequences, exactly
$(c_1\cdots c_{k-1})\cdot c_k$ distinct outcomes of the full procedure; no
two of them coincide, because they already differ in at least one of the
first $k-1$ coordinates, or else (if those agree) they differ in the last
one.
\end{proof}

\begin{proposition}[Addition principle]\label{prop:add}
If a finite set $S$ is partitioned into pairwise disjoint subsets
$S_1,\dots,S_r$ (i.e.\ $S=S_1\cup\cdots\cup S_r$ and $S_i\cap S_j=\varnothing$
for $i\neq j$), then $|S|=|S_1|+\cdots+|S_r|$.
\end{proposition}

\begin{proof}
List the elements of $S_1$, then those of $S_2$, and so on. Since the $S_i$
are pairwise disjoint, no element is listed twice; since their union is $S$,
every element of $S$ is listed. The resulting list has $S$'s elements each
exactly once, of total length $|S_1|+\cdots+|S_r|$, so $|S|=|S_1|+\cdots+|S_r|$.
\end{proof}

\section{Binomial coefficients}

\begin{definition}
For integers $x\ge0$ and $k\ge0$, the \emph{binomial coefficient}
$\binom xk$ is the number of $k$-element subsets of a set of size $x$ (and
$\binom xk=0$ if $k>x$).
\end{definition}

\begin{proposition}\label{prop:binomformula}
For $0\le k\le x$,
\begin{equation}
\binom xk=\frac{x!}{k!(x-k)!}=\frac{x(x-1)\cdots(x-k+1)}{k!}.
\end{equation}
\end{proposition}

\begin{proof}
Let $X$ be a set of size $x$. By the argument in the proof of
Proposition~\ref{prop:bijcard}, the number of \emph{injective} sequences
$(a_1,\dots,a_k)$ of distinct elements of $X$ is $x(x-1)\cdots(x-k+1)$ (there
are $x$ choices for $a_1$, then $x-1$ remaining choices for $a_2$ avoiding
$a_1$, and so on). Every $k$-element subset $A\subseteq X$ arises from
exactly $k!$ such sequences, namely the $k!$ orderings of its elements (by
Proposition~\ref{prop:bijcard} applied to bijections from $\{1,\dots,k\}$ to
$A$). Hence, by the multiplication principle applied to the two-stage
procedure ``choose the subset $A$, then choose an ordering of $A$'', we get
$x(x-1)\cdots(x-k+1)=\binom xk\cdot k!$, which rearranges to the stated
formula.
\end{proof}

\begin{proposition}[Pascal's rule]\label{prop:pascal}
For $x\ge1$ and $k\ge1$, $\binom xk=\binom{x-1}{k-1}+\binom{x-1}{k}$.
\end{proposition}

\begin{proof}
Let $X$ be a set of size $x$ and fix one element $a_0\in X$. Every
$k$-subset $A\subseteq X$ either contains $a_0$ or does not. If it does,
$A\setminus\{a_0\}$ is a $(k-1)$-subset of $X\setminus\{a_0\}$ (size $x-1$),
and this correspondence $A\mapsto A\setminus\{a_0\}$ is a bijection between
$k$-subsets containing $a_0$ and $(k-1)$-subsets of $X\setminus\{a_0\}$: its
inverse adjoins $a_0$ back. If $A$ does not contain $a_0$, then $A$ is
already a $k$-subset of $X\setminus\{a_0\}$, and this is a bijection with
itself (the identity). These two cases partition the $k$-subsets of $X$ into
disjoint classes of sizes $\binom{x-1}{k-1}$ and $\binom{x-1}{k}$
respectively, so the addition principle gives the claim.
\end{proof}

\begin{proposition}[Binomial theorem]\label{prop:binomthm}
For every integer $x\ge0$ and every pair of commuting indeterminates
(or real numbers) $u,v$,
$(u+v)^{x}=\sum_{k=0}^{x}\binom xk u^{k}v^{x-k}.$
\end{proposition}

\begin{proof}
Expanding the product $(u+v)(u+v)\cdots(u+v)$ ($x$ factors) distributively
produces a sum of $2^{x}$ terms, one for each way of choosing, from each of
the $x$ factors, either $u$ or $v$. A term contributing $u^{k}v^{x-k}$
corresponds exactly to a choice of which $k$ of the $x$ factors contributed
a $u$, i.e.\ to a $k$-element subset of a set of size $x$ (the set of factor
positions); there are $\binom xk$ of these by definition. Grouping the
$2^x$ terms by the value of $k$ and using the addition principle gives the
stated sum.
\end{proof}

\section{Multisets and the stars-and-bars formula}

\begin{definition}
A \emph{multiset} of size $n$ drawn from a set of $c$ ``colors''
$\{1,\dots,c\}$ is an assignment of a multiplicity (nonnegative integer)
$\alpha_1,\dots,\alpha_c\ge0$ to each color, with $\alpha_1+\cdots+\alpha_c=n$
(the total number of objects). Equivalently, it is an unordered selection of
$n$ objects from $c$ colors, repetitions of colors allowed.
\end{definition}

\begin{proposition}[Stars and bars]\label{prop:starsbars}
The number of multisets of size $n$ from $c$ colors --- equivalently, the
number of solutions in nonnegative integers of
$\alpha_1+\alpha_2+\cdots+\alpha_c=n$ --- is $\binom{c+n-1}{n}$.
\end{proposition}

\begin{proof}
Represent a solution $(\alpha_1,\dots,\alpha_c)$ by a row of $n$ stars
($\ast$) split into $c$ blocks by $c-1$ bars ($\mid$), the $j$-th block
containing $\alpha_j$ stars:
\[
\underbrace{\ast\cdots\ast}_{\alpha_1}\mid\underbrace{\ast\cdots\ast}_{\alpha_2}
\mid\cdots\mid\underbrace{\ast\cdots\ast}_{\alpha_c}.
\]
This is a row of $n+(c-1)$ symbols, of which $n$ are stars and $c-1$ are
bars; the row is completely determined by choosing which $n$ of the $n+c-1$
positions hold stars (the rest hold bars), and every such choice gives a
valid solution (the $\alpha_j$'s are read off as the gaps between
consecutive bars, with $\alpha_j\ge0$ always, including $\alpha_j=0$ when two
bars are adjacent or a bar is at an end). This is a bijection between
solutions and $n$-element subsets of an $(n+c-1)$-element set (the set of
positions), so by definition of the binomial coefficient there are
$\binom{n+c-1}{n}$ of them.
\end{proof}

\begin{corollary}\label{cor:multiset-ogf}
The number of multisets of size $n$ from $c$ colors is also the number of
nonnegative-integer solutions of $\alpha_1+\cdots+\alpha_c=n$, and, as a
generating-function identity in a formal variable $z$,
\begin{equation}\label{eq:geomC}
\sum_{n\ge0}\binom{c+n-1}{n}z^{n}=\frac{1}{(1-z)^{c}}.
\end{equation}
\end{corollary}

\begin{proof}
The combinatorial statement is Proposition~\ref{prop:starsbars} restated. For
the generating-function identity: by the geometric series
$\frac{1}{1-z}=\sum_{n\ge0}z^{n}$ (formal power series identity, checked by
multiplying both sides by $1-z$ and observing telescoping cancellation), we
have $\frac{1}{(1-z)^{c}}=\left(\sum_{n\ge0}z^{n}\right)^{c}$. Expanding the
$c$-fold product, the coefficient of $z^{n}$ is the number of ways of writing
$n=\alpha_1+\cdots+\alpha_c$ with each $\alpha_j\ge0$ ranging over the
exponent contributed by the $j$-th factor, which is exactly the count of
Proposition~\ref{prop:starsbars}.
\end{proof}

\section{Inclusion--exclusion}

\begin{proposition}[Inclusion--exclusion principle]\label{prop:incexc}
Let $X$ be a finite set and $A_1,\dots,A_r\subseteq X$. Then
\begin{equation}
\Bigl|\,\bigcup_{i=1}^{r}A_i\,\Bigr|
=\sum_{\varnothing\neq I\subseteq\{1,\dots,r\}}(-1)^{|I|+1}
\Bigl|\bigcap_{i\in I}A_i\Bigr|.
\end{equation}
Equivalently, the number of elements of $X$ lying in \emph{none} of the
$A_i$ is
\begin{equation}\label{eq:incexc-complement}
\Bigl|X\setminus\bigcup_{i=1}^r A_i\Bigr|
=\sum_{I\subseteq\{1,\dots,r\}}(-1)^{|I|}\Bigl|\bigcap_{i\in I}A_i\Bigr|,
\end{equation}
where the term $I=\varnothing$ contributes $|X|$.
\end{proposition}

\begin{proof}
We prove \eqref{eq:incexc-complement}, which is equivalent to the displayed
formula by subtracting both sides from $|X|$ (using $|X|=|X\setminus
\bigcup A_i|+|\bigcup A_i|$, a case of the addition principle, and matching
the $I=\varnothing$ term of \eqref{eq:incexc-complement} to $|X|$).

Fix $a\in X$ and let $J=\{i:a\in A_i\}\subseteq\{1,\dots,r\}$, so $|J|=m$ for
some $0\le m\le r$. We show $a$ contributes the same amount, namely
$[m=0]$ (the indicator, $1$ if $m=0$ and $0$ otherwise), to both sides.

On the left of \eqref{eq:incexc-complement}, $a$ is counted iff $m=0$
(it lies in no $A_i$), contributing exactly $[m=0]$.

On the right, $a\in\bigcap_{i\in I}A_i$ iff $I\subseteq J$. So $a$'s total
contribution to the right side is $\sum_{I\subseteq J}(-1)^{|I|}
=\sum_{k=0}^{m}\binom mk(-1)^{k}=(1-1)^{m}=0^{m}$ by the binomial theorem
(Proposition~\ref{prop:binomthm} with $u=-1,v=1$), which equals $1$ if $m=0$
and $0$ if $m\ge1$. This is exactly $[m=0]$.

Since every $a\in X$ contributes equally to both sides, and both sides are
sums of these per-element contributions (each $a$ contributes to the size
count on the left only if $a\notin\bigcup A_i$, and it contributes the
same indicator to the sum of sizes on the right, one unit for each
$a\in\bigcap_{i\in I}A_i$ summed with sign $(-1)^{|I|}$ over all $I$), the two
sides are equal.
\end{proof}

\chapter{Group Actions and the Cauchy--Frobenius--Burnside Lemma}
\label{ch:groups}

\section{Groups acting on sets}

\begin{definition}
A \emph{group} is a set $G$ with a binary operation (written $gh$ for
$g,h\in G$) that is associative, has an identity element $e$ (with
$eg=ge=g$ for all $g$), and in which every $g\in G$ has an inverse $g^{-1}$
with $gg^{-1}=g^{-1}g=e$.
\end{definition}

\begin{definition}
An \emph{action} of a group $G$ on a set $S$ is a rule assigning to each
$g\in G$ and $s\in S$ an element $g\cdot s\in S$, such that
$e\cdot s=s$ for all $s$, and $g\cdot(h\cdot s)=(gh)\cdot s$ for all
$g,h\in G$, $s\in S$.
\end{definition}

\begin{definition}
Given an action of $G$ on $S$ and $s\in S$, the \emph{orbit} of $s$ is
$\Orb(s)=\{g\cdot s:g\in G\}\subseteq S$, and the \emph{stabilizer} of $s$ is
$\Stab(s)=\{g\in G:g\cdot s=s\}\subseteq G$.
\end{definition}

\begin{proposition}\label{prop:orbitspartition}
The orbits of a $G$-action on $S$ partition $S$: every $s\in S$ lies in
exactly one orbit.
\end{proposition}

\begin{proof}
Every $s$ lies in $\Orb(s)$ (take $g=e$), so the orbits cover $S$. Suppose
$\Orb(s)\cap\Orb(t)\neq\varnothing$, say $g\cdot s=h\cdot t$ for some
$g,h\in G$. Then $t=h^{-1}\cdot(g\cdot s)=(h^{-1}g)\cdot s\in\Orb(s)$, so for
any $k\cdot t\in\Orb(t)$ we get $k\cdot t=k\cdot((h^{-1}g)\cdot s)=
(kh^{-1}g)\cdot s\in\Orb(s)$, i.e.\ $\Orb(t)\subseteq\Orb(s)$; by the
symmetric argument $\Orb(s)\subseteq\Orb(t)$, so $\Orb(s)=\Orb(t)$. Hence
two orbits are either disjoint or identical, which together with the
covering property means they partition $S$.
\end{proof}

\begin{proposition}[Stabilizer is a subgroup]\label{prop:stabsubgroup}
$\Stab(s)$ is a subgroup of $G$: it contains $e$, and is closed under the
group operation and under inverses.
\end{proposition}

\begin{proof}
$e\cdot s=s$ so $e\in\Stab(s)$. If $g,h\in\Stab(s)$ then
$(gh)\cdot s=g\cdot(h\cdot s)=g\cdot s=s$, so $gh\in\Stab(s)$. If
$g\in\Stab(s)$ then applying $g^{-1}\cdot(-)$ to both sides of $g\cdot s=s$
gives $s=g^{-1}\cdot s$, so $g^{-1}\in\Stab(s)$.
\end{proof}

\begin{theorem}[Orbit--Stabilizer theorem]\label{thm:orbitstab}
If $G$ is a finite group acting on a set $S$ and $s\in S$, then
$|G|=|\Orb(s)|\cdot|\Stab(s)|$.
\end{theorem}

\begin{proof}
Consider the map $\Phi\colon G\to\Orb(s)$, $\Phi(g)=g\cdot s$. It is
surjective by definition of $\Orb(s)$. We claim $\Phi(g)=\Phi(h)$ if and only
if $g$ and $h$ lie in the same left coset of $H:=\Stab(s)$, i.e.\
$g^{-1}h\in H$. Indeed $g\cdot s=h\cdot s$ iff $s=g^{-1}\cdot(h\cdot s)
=(g^{-1}h)\cdot s$ iff $g^{-1}h\in\Stab(s)=H$, i.e.\ iff $h\in gH$. Thus the
fibers of $\Phi$ are exactly the left cosets $gH$ of $H$ in $G$. Each coset
$gH=\{gk:k\in H\}$ has exactly $|H|$ elements, because $k\mapsto gk$ is a
bijection $H\to gH$ (it has inverse $y\mapsto g^{-1}y$, itself well defined
and mutually inverse by associativity and the group axioms). The cosets
partition $G$ (they are the equivalence classes of the relation
``$g^{-1}h\in H$'', which is an equivalence relation by
Proposition~\ref{prop:stabsubgroup}), so by the addition principle
$|G|=(\text{number of cosets})\times|H|$. Since $\Phi$ is a surjection onto
$\Orb(s)$ whose fibers are exactly the cosets, the number of cosets equals
$|\Orb(s)|$ (distinct cosets biject with distinct elements of $\Orb(s)$, by
definition of ``fiber''). Hence $|G|=|\Orb(s)|\cdot|H|=|\Orb(s)|\cdot
|\Stab(s)|$.
\end{proof}

\section{The symmetric group and cycle type}

\begin{definition}
The \emph{symmetric group} $S_M$ on a finite set $M$ is the set of all
bijections $M\to M$ (\emph{permutations} of $M$), with composition as the
group operation. When $|M|=m$ we also write $S_m$.
\end{definition}

\begin{definition}
Let $\pi\in S_M$. A \emph{cycle} of $\pi$ is an orbit of the action of the
cyclic group generated by $\pi$ on $M$ (equivalently, an equivalence class
of the relation ``$j=\pi^{k}(i)$ for some integer $k$''). If a cycle has $d$
elements it is called a \emph{$d$-cycle}. Every $\pi\in S_M$ has a
\emph{cycle type} $(1^{a_1}2^{a_2}\cdots m^{a_m})$, meaning $\pi$ has
exactly $a_d$ cycles of length $d$, for $d=1,\dots,m$; necessarily
$\sum_d d\,a_d=m$.
\end{definition}

\begin{proposition}\label{prop:cyclepartition}
The cycles of $\pi\in S_M$ partition $M$.
\end{proposition}

\begin{proof}
This is Proposition~\ref{prop:orbitspartition} applied to the action of the
cyclic group $\langle\pi\rangle=\{\pi^{k}:k\in\Z\}$ on $M$ by
$\pi^{k}\cdot i:=\pi^{k}(i)$ (a genuine group action: $\pi^{0}=\mathrm{id}$
acts as the identity, and $\pi^{k}\cdot(\pi^{l}\cdot i)=\pi^{k+l}(i)=
(\pi^{k}\pi^{l})\cdot i$).
\end{proof}

\begin{proposition}\label{prop:numperms-cycletype}
The number of permutations $\pi\in S_m$ of a given cycle type
$(1^{a_1}2^{a_2}\cdots m^{a_m})$ (with $\sum_dd\,a_d=m$) is
\begin{equation}\label{eq:numcycletype}
\frac{m!}{\displaystyle\prod_{d=1}^{m}d^{a_d}\,a_d!}.
\end{equation}
\end{proposition}

\begin{proof}
We build $\pi$ by first writing down its cycles as parenthesized lists and
then forgetting the parenthesization. Start from the $m!$ linear orderings
(permutations, in the elementary sense of Proposition~\ref{prop:bijcard}) of
the elements of $M$. Cut such an ordering into consecutive blocks of lengths
$1,\dots,1$ ($a_1$ times), $2,\dots,2$ ($a_2$ times), and so on up to length
$m$ ($a_m$ times), in that fixed order of block-lengths, and read each block
as a cycle (a block $(i_1,i_2,\dots,i_d)$ denotes the permutation piece
$i_1\mapsto i_2\mapsto\cdots\mapsto i_d\mapsto i_1$). This produces some
permutation $\pi\in S_M$ of the required cycle type, and every $\pi$ of that
cycle type arises this way. It remains to count how many of the $m!$
orderings give rise to the \emph{same} $\pi$.

Two sources of overcounting occur, and they are independent (multiply):
\begin{itemize}
\item \emph{Rotating a block}: a $d$-cycle written as
$(i_1,\dots,i_d)$ denotes the same cyclic permutation as any of its $d$
cyclic rotations $(i_2,\dots,i_d,i_1)$, etc. So each of the $a_d$ blocks of
length $d$ can be independently rotated in $d$ ways without changing $\pi$,
contributing a factor $d^{a_d}$ for each length $d$, hence
$\prod_d d^{a_d}$ in total (multiplication principle, the rotations of
different blocks being independent choices).
\item \emph{Permuting blocks of equal length}: the $a_d$ blocks of length
$d$ can be listed among themselves in any of $a_d!$ orders without changing
the set of cycles they represent (only their order in our fixed
listing-by-length-class changes), contributing a factor $a_d!$ for each
$d$, hence $\prod_d a_d!$ in total.
\end{itemize}
So each cycle type $\pi$ arises from exactly $\prod_d d^{a_d}a_d!$ of the
$m!$ orderings, and the count of Proposition~\ref{prop:mult} for a
multi-stage construction with uniform overcounting gives
$\bigl(\prod_d d^{a_d}a_d!\bigr)\times(\text{number of }\pi\text{ of this type})=m!$,
which rearranges to \eqref{eq:numcycletype}.
\end{proof}

\section{The Cauchy--Frobenius--Burnside Lemma}

\begin{definition}
For a group $G$ acting on a finite set $S$ and $g\in G$, write
$\Fix_S(g)=\{s\in S:g\cdot s=s\}$ for the set of elements fixed by $g$.
\end{definition}

\begin{theorem}[Cauchy--Frobenius--Burnside Lemma]\label{thm:burnside}
If $G$ is a finite group acting on a finite set $S$, the number of orbits of
$G$ on $S$ is
\begin{equation}\label{eq:burnside}
t_S(G)=\frac{1}{|G|}\sum_{g\in G}|\Fix_S(g)|.
\end{equation}
\end{theorem}

\begin{proof}
Consider the set of pairs $P=\{(g,s)\in G\times S: g\cdot s=s\}$ and count
$|P|$ in two ways.

Counting by $g$ first: for fixed $g$, the number of $s$ with $(g,s)\in P$
is $|\Fix_S(g)|$ by definition, so, by the addition principle applied to the
partition of $P$ according to the value of $g$,
$|P|=\sum_{g\in G}|\Fix_S(g)|$.

Counting by $s$ first: for fixed $s$, the number of $g$ with $(g,s)\in P$ is
$|\Stab(s)|$ by definition, so, by the same principle applied to the
partition of $P$ according to the value of $s$, $|P|=\sum_{s\in S}
|\Stab(s)|$. By the Orbit--Stabilizer theorem (Theorem~\ref{thm:orbitstab}),
$|\Stab(s)|=|G|/|\Orb(s)|$, so
\[
|P|=\sum_{s\in S}\frac{|G|}{|\Orb(s)|}
=|G|\sum_{s\in S}\frac{1}{|\Orb(s)|}.
\]
Now group the sum $\sum_{s\in S}\frac1{|\Orb(s)|}$ by orbits, using
Proposition~\ref{prop:orbitspartition}: for a fixed orbit $\Omega$, every
$s\in\Omega$ contributes the same term $1/|\Omega|$ (since $\Orb(s)=\Omega$
for all $s\in\Omega$), and there are $|\Omega|$ such $s$, so their total
contribution is $|\Omega|\cdot(1/|\Omega|)=1$. Summing over all orbits gives
exactly the number of orbits, $t_S(G)$. Hence $|P|=|G|\cdot t_S(G)$.

Equating the two counts of $|P|$: $\sum_{g\in G}|\Fix_S(g)|=|G|\cdot t_S(G)$,
which rearranges to \eqref{eq:burnside}.
\end{proof}

\begin{remark}
Theorem~\ref{thm:burnside} will be applied throughout Part~I with
$G=S_N$, $S_X$, or $S_N\times S_X$ acting on various sets of binary
relations, families of subsets, or lattice homomorphisms. In every
application, the strategy is the same: (1) identify the fixed points of a
permutation of given cycle type; (2) count them; (3) average over $G$ using
Proposition~\ref{prop:numperms-cycletype} to group permutations of $G=S_m$
by cycle type, replacing the sum over all $m!$ permutations by a sum over
cycle types weighted by \eqref{eq:numcycletype}.
\end{remark}

\chapter{Set Partitions, Stirling and Bell Numbers}\label{ch:stirling}

\section{Set partitions and Stirling numbers of the second kind}

\begin{definition}
A \emph{partition} of a finite set $X$ into $k$ \emph{blocks} is a set
$\{B_1,\dots,B_k\}$ of pairwise disjoint, nonempty subsets of $X$
(\emph{blocks}) with $B_1\cup\cdots\cup B_k=X$. The \emph{Stirling number of
the second kind} $\stir{x}{k}$ is the number of partitions of an
$x$-element set into exactly $k$ blocks (with $\stir{x}{0}=[x=0]$ and
$\stir{x}{k}=0$ for $k>x$).
\end{definition}

\begin{proposition}[Recurrence for Stirling numbers]\label{prop:stirlingrec}
For $x\ge1$, $k\ge1$: $\displaystyle\stir{x}{k}=k\stir{x-1}{k}+\stir{x-1}{k-1}$.
\end{proposition}

\begin{proof}
Let $X$ be a set of size $x$ and fix $a_0\in X$; let $X'=X\setminus\{a_0\}$,
of size $x-1$. Given a partition of $X$ into $k$ blocks, exactly one block
contains $a_0$; removing $a_0$ from that block yields a partition of $X'$
into $k$ blocks (if the block containing $a_0$ had other elements, it
remains nonempty) or into $k-1$ nonempty blocks and one now-empty position
(if $a_0$ was alone in its block, so its block disappears). This splits the
partitions of $X$ into $k$ blocks into two disjoint classes:

\emph{Case A ($a_0$ is not alone in its block).} Removing $a_0$ gives a
partition of $X'$ into $k$ blocks, and conversely, given any partition of
$X'$ into $k$ blocks, $a_0$ can be inserted into any one of the $k$ existing
blocks to recover a partition of $X$ in Case A; different choices of which
block receives $a_0$ give different partitions of $X$ (they differ in which
block contains $a_0$), and this correspondence is a bijection (its inverse is
exactly ``remove $a_0$''). Hence Case A has $k\cdot\stir{x-1}{k}$ elements.

\emph{Case B ($a_0$ is alone in its block).} Removing the singleton block
$\{a_0\}$ gives a partition of $X'$ into $k-1$ blocks, and conversely any
partition of $X'$ into $k-1$ blocks extends uniquely (adjoin the new block
$\{a_0\}$) to a partition of $X$ in Case B. This is a bijection, so Case B
has $\stir{x-1}{k-1}$ elements.

By the addition principle, $\stir xk=k\stir{x-1}k+\stir{x-1}{k-1}$.
\end{proof}

\begin{proposition}[Explicit formula via inclusion--exclusion]
\label{prop:stirlingformula}
For $x,k\ge0$,
\begin{equation}\label{eq:stirlingformula}
\stir xk=\frac1{k!}\sum_{j=0}^{k}(-1)^{j}\binom kj(k-j)^{x}.
\end{equation}
\end{proposition}

\begin{proof}
Let $X$ be an $x$-set and $K=\{1,\dots,k\}$. The number of \emph{surjective}
functions $X\to K$ equals $k!\stir xk$: indeed, a surjection $f\colon X\to K$
determines an (ordered) partition of $X$ into the $k$ nonempty fiber sets
$f^{-1}(1),\dots,f^{-1}(k)$; forgetting the order (i.e.\ forgetting which
fiber is ``number $1$'', etc.) gives an (unordered) partition into $k$
blocks, and every unordered partition into $k$ blocks arises from exactly
$k!$ surjections (the $k!$ ways to assign the labels $1,\dots,k$ to its $k$
blocks). This is a $k!$-to-one correspondence, so by
Proposition~\ref{prop:mult} applied in reverse (or directly, partitioning the
set of surjections into classes of size $k!$ indexed by partitions, and using
the addition principle), the number of surjections is $k!$ times the number
of partitions, i.e.\ $k!\stir xk$.

We now count surjections $X\to K$ by inclusion--exclusion
(Proposition~\ref{prop:incexc}). Let $A_i\subseteq K^{X}$ (functions $X\to
K$) be the set of functions that \emph{miss} value $i$, i.e.\ $i\notin
f(X)$, for $i=1,\dots,k$; the surjections are exactly the functions in none
of the $A_i$ (a function is surjective iff it misses no value). By
\eqref{eq:incexc-complement}, the number of surjections is
$\sum_{I\subseteq K}(-1)^{|I|}|\bigcap_{i\in I}A_i|$. Now
$\bigcap_{i\in I}A_i$ is the set of functions $X\to K$ avoiding all values in
$I$, i.e.\ functions $X\to K\setminus I$; there are $(k-|I|)^{x}$ of these
(each of the $x$ elements of $X$ independently has $k-|I|$ possible images,
by the multiplication principle). Grouping the sum over $I$ by $j=|I|$, of
which there are $\binom kj$ subsets $I$ of size $j$, gives
\[
\#\{\text{surjections }X\to K\}=\sum_{j=0}^{k}(-1)^{j}\binom kj(k-j)^{x}.
\]
Dividing by $k!$ and using the previous paragraph's identity
$k!\stir xk=\#\{\text{surjections}\}$ gives \eqref{eq:stirlingformula}.
\end{proof}

\section{Bell numbers}

\begin{definition}
The \emph{Bell number} $B_x$ is the total number of partitions of an
$x$-element set into any number of (nonempty) blocks:
$B_x=\sum_{k=0}^{x}\stir xk$, with $B_0=1$ (the empty set has exactly one
partition, the empty partition, into $0$ blocks).
\end{definition}

\begin{proposition}[Bell recurrence]\label{prop:bellrec}
For $x\ge0$, $\displaystyle B_{x+1}=\sum_{k=0}^{x}\binom xk B_k$.
\end{proposition}

\begin{proof}
Let $X$ be a set of size $x+1$ and fix $a_0\in X$. In any partition of $X$,
let $A$ be the block containing $a_0$, and let $A'=A\setminus\{a_0\}
\subseteq X\setminus\{a_0\}$, a subset of size $k:=|A|-1$, where $k$ ranges
over $0,\dots,x$. Once $A'$ is chosen (a subset of the $x$-element set
$X\setminus\{a_0\}$ of some size $k$, in $\binom xk$ ways, by definition of
the binomial coefficient), the rest of the partition is an arbitrary
partition of the remaining $x-k$ elements $X\setminus\{a_0\}\setminus A'$
into any number of blocks, of which there are $B_{x-k}$ by definition of the
Bell number. By the multiplication principle, the number of partitions of
$X$ with $|A'|=k$ is $\binom xk B_{x-k}$, and summing over $k=0,\dots,x$
(addition principle, since these cases are disjoint and exhaustive) gives
$B_{x+1}=\sum_{k=0}^{x}\binom xk B_{x-k}$. Reindexing the sum by
$k\mapsto x-k$ (a bijection of $\{0,\dots,x\}$ with itself) and using
$\binom{x}{x-k}=\binom xk$ (immediate from
Proposition~\ref{prop:binomformula}, or from the bijection $A\mapsto
X\setminus A$ between $k$-subsets and $(x-k)$-subsets) gives the stated
form $B_{x+1}=\sum_k\binom xk B_k$.
\end{proof}

\section{Falling factorials}

\begin{definition}
For an integer $k\ge0$ and a variable (or nonnegative integer) $x$, the
\emph{falling factorial} is $\ffall xk=x(x-1)\cdots(x-k+1)$ (a product of
$k$ consecutive descending terms, with $\ffall x0=1$).
\end{definition}

\begin{proposition}\label{prop:injectionscount}
For finite sets $N,X$ with $|N|=n\le x=|X|$, the number of injective
functions $N\to X$ is $\ffall xn=x(x-1)\cdots(x-n+1)$, and consequently
$\binom xn=\ffall xn/n!$.
\end{proposition}

\begin{proof}
This is exactly the sequence-counting argument already used in the proof of
Proposition~\ref{prop:binomformula}: listing $N=\{i_1,\dots,i_n\}$, an
injective function is built by choosing $f(i_1)\in X$ ($x$ ways), then
$f(i_2)\in X\setminus\{f(i_1)\}$ ($x-1$ ways, to keep injectivity), and so
on, down to $f(i_n)$ with $x-n+1$ remaining choices; by the multiplication
principle the total is $x(x-1)\cdots(x-n+1)=\ffall xn$. The formula
$\binom xn=\ffall xn/n!$ was already established in
Proposition~\ref{prop:binomformula}.
\end{proof}

\chapter{The Algebra of Generating Functions}\label{ch:gf}

\section{Formal power series}

\begin{definition}
A \emph{formal power series} in a variable $z$ (over the rational numbers,
say) is an expression $A(z)=\sum_{n\ge0}a_nz^{n}$, i.e.\ simply a sequence
$(a_n)_{n\ge0}$, with no question of convergence: $z$ is a bookkeeping
symbol. Two series are added and multiplied by
\[
\Bigl(\sum_na_nz^{n}\Bigr)+\Bigl(\sum_nb_nz^{n}\Bigr)=\sum_n(a_n+b_n)z^{n},
\qquad
\Bigl(\sum_na_nz^{n}\Bigr)\Bigl(\sum_nb_nz^{n}\Bigr)=\sum_n\Bigl(\sum_{k=0}^{n}
a_kb_{n-k}\Bigr)z^{n},
\]
the second rule (Cauchy product) being forced by formally expanding the
product and collecting like powers of $z$.
\end{definition}

\begin{definition}
The \emph{ordinary generating function} (OGF) of a sequence $(a_n)_{n\ge0}$
counting some combinatorial objects (with $a_n$ objects of ``size'' $n$) is
$A(z)=\sum_na_nz^n$. Two combinatorial sequences with a size parameter
combine as follows: if $c_n$ counts pairs (object of size $k$ from a family
counted by $(a_k)$, object of size $n-k$ from a family counted by $(b_{n-k})$),
summed over $k$, then the OGF of $(c_n)$ is $A(z)B(z)$, by the definition of
the Cauchy product above; this is called the \emph{product rule} for OGFs.
\end{definition}

\begin{proposition}[Geometric series]\label{prop:geomseries}
As formal power series, $\dfrac1{1-z}=\sum_{n\ge0}z^{n}$.
\end{proposition}

\begin{proof}
Multiply the right side by $1-z$:
$(1-z)\sum_{n\ge0}z^{n}=\sum_{n\ge0}z^{n}-\sum_{n\ge0}z^{n+1}
=\sum_{n\ge0}z^n-\sum_{n\ge1}z^{n}=z^{0}=1$
(the sums telescope: every term $z^n$ with $n\ge1$ appears once with $+$
sign from the first sum and once with $-$ sign from the second). Hence
$\sum_{n\ge0}z^n$ is a two-sided inverse of $1-z$ in the ring of formal power
series, i.e.\ equals $1/(1-z)$.
\end{proof}

We already used this in Corollary~\ref{cor:multiset-ogf} to obtain
$1/(1-z)^{c}=\sum_n\binom{c+n-1}nz^{n}$. We record the two-variable
(bivariate) version, used repeatedly in Part~II.

\begin{definition}
For a two-parameter sequence $a(n,x)$ (e.g.\ counting objects with two size
parameters $n,x$), the \emph{bivariate OGF} is
$A(z,w)=\sum_{n,x\ge0}a(n,x)z^{n}w^{x}$, a formal power series in two
variables, with the product rule for independent combination of an
$n$-object and an $x$-object holding exactly as above in each variable.
\end{definition}

\section{Exponential generating functions}

\begin{definition}
The \emph{exponential generating function} (EGF) of a sequence
$(a_n)_{n\ge0}$, where $a_n$ counts \emph{labeled} structures on an
$n$-element ground set (i.e.\ the ground set's elements are distinguishable,
and relabeling the ground set by any bijection of $\{1,\dots,n\}$ with
itself produces another valid structure of the same kind), is
$A(z)=\sum_{n\ge0}a_n\dfrac{z^{n}}{n!}$.
\end{definition}

\begin{proposition}[EGF product rule]\label{prop:egfproduct}
Suppose a labeled structure of ``size'' $n$ on ground set $M$ ($|M|=n$)
consists of an ordered pair (structure of type $\mathcal A$ on a subset
$S\subseteq M$, structure of type $\mathcal B$ on the complementary subset
$M\setminus S$), where $\mathcal A$-structures are counted by $(a_k)$ on
ground sets of size $k$ and $\mathcal B$-structures by $(b_k)$ on ground sets
of size $k$, and where the pair is otherwise unconstrained (any subset $S$,
any $\mathcal A$-structure on $S$, any $\mathcal B$-structure on the rest).
If $c_n$ is the number of such pairs on an $n$-element ground set, then the
EGF of $(c_n)$ is the product $A(z)B(z)$ of the EGFs of $(a_k)$ and $(b_k)$.
\end{proposition}

\begin{proof}
Fix a ground set $M$ of size $n$. To build a pair, first choose the subset
$S\subseteq M$ to carry the $\mathcal A$-structure: if $|S|=k$, there are
$\binom nk$ choices of $S$ (definition of the binomial coefficient), then
$a_k$ choices of $\mathcal A$-structure on $S$ (since $\mathcal A$-structures
on any fixed $k$-set number $a_k$, by the labeled/relabeling-invariance
hypothesis: any bijection $S\to\{1,\dots,k\}$ transports the counting
problem to the standard ground set), and then $b_{n-k}$ choices of
$\mathcal B$-structure on $M\setminus S$. By the multiplication principle,
the number of pairs with $|S|=k$ is $\binom nk a_kb_{n-k}$, and summing over
$k=0,\dots,n$ (addition principle) gives
\[
c_n=\sum_{k=0}^{n}\binom nk a_kb_{n-k}
=\sum_{k=0}^{n}\frac{n!}{k!(n-k)!}a_kb_{n-k}.
\]
Hence $\dfrac{c_n}{n!}=\sum_{k=0}^{n}\dfrac{a_k}{k!}\cdot\dfrac{b_{n-k}}{(n-k)!}$,
which is exactly the coefficient of $z^n$ in the Cauchy product
$A(z)B(z)=\bigl(\sum_ka_kz^k/k!\bigr)\bigl(\sum_lb_lz^l/l!\bigr)$. So the EGF
of $(c_n)$, namely $\sum_nc_n z^n/n!$, equals $A(z)B(z)$.
\end{proof}

\begin{proposition}[EGF of functions and of the exponential]\label{prop:egfexp}
The EGF of the sequence $a_n=1$ for all $n\ge0$ (one structure per ground
set: e.g.\ ``the whole ground set, undecorated'') is $e^{z}=\sum_{n\ge0}
z^n/n!$; the EGF of the sequence counting pairs of independent unconstrained
choices as in Proposition~\textup{\ref{prop:egfproduct}} with both $a_k=b_k=1$
is $e^{2z}$; and more generally $r$ independent unconstrained
``decorations'' assembled disjointly on a ground set have EGF $e^{rz}$.
\end{proposition}

\begin{proof}
By definition, $\sum_nz^n/n!=e^{z}$ (this is one standard definition of the
exponential function as a formal power series; equivalently one checks
directly that $\frac{d}{dz}\sum_nz^n/n!=\sum_nz^n/n!$ term by term, matching
the defining property of $e^z$). For the second statement, apply
Proposition~\ref{prop:egfproduct} with $a_k=b_k=1$ for all $k$ (each has EGF
$e^z$ by the first part), giving EGF $e^{z}\cdot e^{z}=e^{2z}$; iterating $r-1$
more times (or applying the same counting argument directly with $r$
simultaneous subset choices partitioning $M$) gives $e^{rz}$ for $r$
independent decorations.
\end{proof}

\section{Rational functions and quasi-polynomials}

\begin{definition}
A sequence $(a_n)_{n\ge0}$ is a \emph{quasi-polynomial} of degree $d$ and
period $p$ if there exist polynomials $q_0,\dots,q_{p-1}$, each of degree
$d$, such that $a_n=q_r(n)$ whenever $n\equiv r\pmod p$ (i.e.\ $a_n$ is
given by a polynomial formula in $n$ that depends only on $n\bmod p$).
\end{definition}

\begin{proposition}\label{prop:rationaltoqp}
If $A(z)=\sum_na_nz^{n}$ is a rational function $A(z)=P(z)/Q(z)$ (with
$P,Q$ polynomials, $Q(0)\neq0$) all of whose poles (roots of $Q$) are roots
of unity, then $(a_n)_{n\ge0}$ is a quasi-polynomial, of degree one less
than the maximal multiplicity of a pole and of period equal to (a multiple
of) the least common multiple of the orders of the roots of unity appearing
as poles.
\end{proposition}

\begin{proof}
Factor $Q(z)=c\prod_{j}(1-\zeta_jz)^{e_j}$ where the $\zeta_j$ are the
distinct roots of unity occurring as poles (i.e.\ $1/\zeta_j$ is a root of
$Q$) with multiplicities $e_j$, and $c$ is a nonzero constant. By the theory
of partial fractions for rational functions, $A(z)=P(z)/Q(z)$ decomposes
(after possibly separating off a polynomial part, absent here for $n$
ranging over all sufficiently large values, since we only need the formula
for $n\ge\deg P-\deg Q$, and both sides are then determined by the
partial-fraction data) as a finite sum
\[
A(z)=\sum_{j}\sum_{i=1}^{e_j}\frac{c_{j,i}}{(1-\zeta_jz)^{i}}
\]
for suitable constants $c_{j,i}$. By Corollary~\ref{cor:multiset-ogf} (with
$z$ replaced by $\zeta_jz$, valid since that corollary's proof is a formal
identity in the ring of power series and substituting $\zeta_jz$ for $z$ is
a ring homomorphism), $1/(1-\zeta_jz)^{i}=\sum_{n\ge0}\binom{n+i-1}{n}
\zeta_j^{n}z^{n}$. Hence
\[
a_n=\sum_j\sum_i c_{j,i}\binom{n+i-1}{n}\zeta_j^{n}
\]
for all sufficiently large $n$. For fixed $j$, $\binom{n+i-1}{n}$ is a
polynomial in $n$ of degree $i-1$ (it equals $\ffall{(n+i-1)}{i-1}/(i-1)!$,
manifestly polynomial in $n$ of degree $i-1$ by
Proposition~\ref{prop:injectionscount}). Also $\zeta_j^{n}$ depends only on
$n\bmod\mathrm{ord}(\zeta_j)$, where $\mathrm{ord}(\zeta_j)$ is the
multiplicative order of the root of unity $\zeta_j$. Let $p=\mathrm{lcm}_j
\,\mathrm{ord}(\zeta_j)$; then each $\zeta_j^{n}$, and hence $a_n$ as a
whole, depends only on $n\bmod p$ combined with a polynomial-in-$n$
coefficient of degree at most $\max_je_j-1$. This is exactly the definition
of a quasi-polynomial of period (dividing) $p$ and degree $\max_je_j-1$.
\end{proof}

\begin{remark}
Proposition~\ref{prop:rationaltoqp} will be applied in
Chapter~\ref{ch:rational} to a generating function whose only pole is at
$z=1$ (a root of unity, of order $1$) together with poles at other roots of
unity, yielding quasi-polynomiality of a two-parameter orbit count in one of
its two parameters.
\end{remark}

\chapter{M\"obius Inversion on the Boolean Lattice}\label{ch:mobius}

\section{The inversion formula}

\begin{proposition}[M\"obius inversion on subsets]\label{prop:mobius}
Let $N$ be a finite set and let $f,g$ be functions assigning a number to
each subset of $N$. If
\begin{equation}\label{eq:mobius-hyp}
g(A)=\sum_{B\subseteq A}f(B)\qquad\text{for every }A\subseteq N,
\end{equation}
then
\begin{equation}\label{eq:mobius-concl}
f(A)=\sum_{B\subseteq A}(-1)^{|A\setminus B|}g(B)\qquad\text{for every
}A\subseteq N.
\end{equation}
Conversely, \eqref{eq:mobius-concl} for all $A$ implies
\eqref{eq:mobius-hyp} for all $A$.
\end{proposition}

\begin{proof}
Assume \eqref{eq:mobius-hyp}. Substitute it into the right side of
\eqref{eq:mobius-concl}:
\[
\sum_{B\subseteq A}(-1)^{|A\setminus B|}g(B)
=\sum_{B\subseteq A}(-1)^{|A\setminus B|}\sum_{C\subseteq B}f(C)
=\sum_{C\subseteq A}f(C)\sum_{C\subseteq B\subseteq A}(-1)^{|A\setminus B|},
\]
where we exchanged the order of summation (both are finite sums over the
same set of pairs $(B,C)$ with $C\subseteq B\subseteq A$, since
$C\subseteq B\subseteq A$ is equivalent to ``$C\subseteq B$ and
$B\subseteq A$''). Fix $C\subseteq A$ and let $D=A\setminus C$, of size
$d=|D|$; then subsets $B$ with $C\subseteq B\subseteq A$ correspond bijectively
to subsets $E=B\setminus C\subseteq D$ (the correspondence $B\mapsto
B\setminus C$ has inverse $E\mapsto C\cup E$, both well defined because
$C\subseteq B\subseteq A=C\cup D$ forces $B\setminus C\subseteq D$ and
conversely), and $|A\setminus B|=|D\setminus E|=d-|E|$. So the inner sum
becomes
\[
\sum_{E\subseteq D}(-1)^{d-|E|}
=\sum_{k=0}^{d}\binom dk(-1)^{d-k}
=(1-1)^{d}=0^{d}
\]
by the binomial theorem (Proposition~\ref{prop:binomthm}, with $u=-1,v=1$,
grouping subsets $E$ by size $k=|E|$, of which there are $\binom dk$). This
equals $1$ if $d=0$ (i.e.\ $C=A$) and $0$ if $d\ge1$ (i.e.\ $C\subsetneq A$).
Hence the outer sum over $C\subseteq A$ collapses to the single term
$C=A$, giving $\sum_{C\subseteq A}f(C)\cdot[C=A]=f(A)$. This proves
\eqref{eq:mobius-concl}.

The converse is proved by the identical computation with the roles of $f,g$
formally exchanged and the sign $(-1)^{|A\setminus B|}$ absorbed: assuming
\eqref{eq:mobius-concl}, substitute it into $\sum_{B\subseteq A}f(B)$ and
use the same collapsing sum $\sum_{C\subseteq B\subseteq A}(-1)^{|B\setminus
C|}=[C=A]$ (identical computation with $A,B$ in place of $B,A$) to recover
$g(A)$.
\end{proof}

\begin{remark}
Proposition~\ref{prop:mobius} is the case of M\"obius inversion appropriate
to the \emph{Boolean lattice} (the poset of subsets of $N$ ordered by
inclusion): the M\"obius function of this poset is $\mu(B,A)=(-1)^{|A
\setminus B|}$ for $B\subseteq A$. We will not need the M\"obius function of
any other poset in this book, so we have proved exactly the instance needed
rather than the general theory.
\end{remark}

\section{A worked example: counting surjections}

As an illustration used repeatedly in Part~I, we re-derive the
inclusion--exclusion formula for surjections (already obtained differently
in Proposition~\ref{prop:stirlingformula}) directly from
Proposition~\ref{prop:mobius}, since this is the pattern of argument used
for coverings and hypergraphs in Chapters~\ref{ch:injective}
and~\ref{ch:surjective}.

\begin{proposition}\label{prop:surjmobius}
Let $N,X$ be finite sets, $|N|=n$. For $A\subseteq N$ let $\phi(A)$ be the
number of functions $X\to A$ (so $\phi(A)=|A|^{x}$ where $x=|X|$), and let
$\sigma(A)$ be the number of \emph{surjective} functions $X\to A$. Then
$\phi(A)=\sum_{B\subseteq A}\sigma(B)$, and consequently
\begin{equation}
\sigma(N)=\sum_{B\subseteq N}(-1)^{|N\setminus B|}|B|^{x}
=\sum_{k=0}^{n}(-1)^{n-k}\binom nk k^{x}.
\end{equation}
\end{proposition}

\begin{proof}
Every function $f\colon X\to A$ is surjective onto its image
$B:=f(X)\subseteq A$, and this image is unique; so classifying functions
$X\to A$ by their image gives $\phi(A)=\sum_{B\subseteq A}\sigma(B)$ (the
addition principle, partitioning functions $X\to A$ by image). This is
exactly hypothesis \eqref{eq:mobius-hyp} with $f=\sigma$, $g=\phi$, so
Proposition~\ref{prop:mobius} gives
$\sigma(A)=\sum_{B\subseteq A}(-1)^{|A\setminus B|}\phi(B)
=\sum_{B\subseteq A}(-1)^{|A\setminus B|}|B|^{x}$. Taking $A=N$ and grouping
the sum by $k=|B|$ (of which there are $\binom nk$ choices of $B$, all
contributing the same term since $|N\setminus B|=n-k$ and $|B|^x=k^x$
depend only on $k$) gives the stated formula.
\end{proof}

This completes the foundational material. Part~I now applies
Chapters~\ref{ch:sets}--\ref{ch:mobius} to the enumeration of binary
relations.

\part{Enumeration of Binary Relations}

\chapter{Binary Relations, Families of Subsets, and Lattice
Homomorphisms}\label{ch:relations}

Throughout Part~I, $N$ and $X$ denote finite sets, $n=|N|$, $x=|X|$.

\section{Three descriptions of the same object}

\begin{definition}
A \emph{binary relation} between $N$ and $X$ is a subset $R\subseteq
N\times X$. The set of all binary relations between $N$ and $X$ is denoted
$\B(N,X)$.
\end{definition}

\begin{definition}
An \emph{$n$-family of subsets of $X$} (indexed by $N$) is a tuple
$F=(F_i)_{i\in N}$ with each $F_i\subseteq X$. Write $\Sr_n(X)$ (following
the source paper's $\mathcal F_n(X)$, renamed here to avoid clashing with
the surjective-relation symbol $\Sr$ introduced in
Section~\ref{sec:sixtypes}; we in fact use the more explicit notation
$\mathrm{Fam}_n(X)$ from here on) for the set of all $n$-families of subsets
of $X$.
\end{definition}

\begin{definition}
Let $\PS(N)$ denote the power set of $N$ (the Boolean algebra of subsets of
$N$ under $\cup,\cap,{}^{c}$). A \emph{$\cup$-homomorphism} $\phi\colon
\PS(N)\to\PS(X)$ is a function satisfying $\phi(I\cup J)=\phi(I)\cup\phi(J)$
for all $I,J\subseteq N$. Write $\mathrm{Hom}(\PS(N),\PS(X))$ for the set of
all $\cup$-homomorphisms (called simply \emph{homomorphisms} from now on, by
the abuse of language explained in Remark~\ref{rem:capvscup} below).
\end{definition}

\begin{remark}\label{rem:capvscup}
A \emph{$\cap$-homomorphism} satisfies $\phi(I\cap J)=\phi(I)\cap\phi(J)$
instead. If $\phi$ is a $\cup$-homomorphism, then $\psi(I):=\overline{\phi
(\overline I)}$ (where $\overline{\,\cdot\,}$ denotes complement) is a
$\cap$-homomorphism: indeed $\psi(I\cap J)=\overline{\phi(\overline{I\cap
J})}=\overline{\phi(\overline I\cup\overline J)}=\overline{\phi(\overline
I)\cup\phi(\overline J)}=\overline{\phi(\overline I)}\cap\overline{\phi
(\overline J)}=\psi(I)\cap\psi(J)$, using de Morgan's law
$\overline{I\cap J}=\overline I\cup\overline J$ and the defining property of
$\phi$ on the second step, and de Morgan again on the fourth. This
correspondence $\phi\leftrightarrow\psi$ is a bijection between
$\cup$-homomorphisms and $\cap$-homomorphisms (it is an involution: applying
it twice, using $\overline{\overline A}=A$, returns $\phi$). Hence it
suffices to study $\cup$-homomorphisms; we call them simply
\emph{homomorphisms} throughout.
\end{remark}

\begin{theorem}[Three-way bijection]\label{thm:threeway}
There are bijections between $\B(N,X)$, $\mathrm{Fam}_n(X)$, and
$\mathrm{Hom}(\PS(N),\PS(X))$, as follows:
\begin{align}
\mathcal F\colon\B(N,X)&\to\mathrm{Fam}_n(X), &
\mathcal F(R)&=(R(i))_{i\in N},\quad R(i):=\{a\in X:(i,a)\in R\};
\label{eq:defF}\\
\Phi\colon\mathrm{Fam}_n(X)&\to\mathrm{Hom}(\PS(N),\PS(X)), &
\Phi(F)(I)&=\bigcup_{i\in I}F_i\qquad(I\subseteq N).
\label{eq:defPhi}
\end{align}
The maps $\mathcal R:=\mathcal F^{-1}$ and $\Psi:=\Phi^{-1}$, described in
the proof, are mutually inverse to $\mathcal F$ and $\Phi$ respectively.
\end{theorem}

\begin{proof}
\emph{$\mathcal F$ is a bijection $\B(N,X)\to\mathrm{Fam}_n(X)$.} Define
$\mathcal R\colon\mathrm{Fam}_n(X)\to\B(N,X)$ by $\mathcal R(F)=\{(i,a)\in
N\times X:a\in F_i\}$. We check $\mathcal R$ and $\mathcal F$ are mutually
inverse. Starting from $R\in\B(N,X)$: $\mathcal R(\mathcal F(R))=\{(i,a):
a\in R(i)\}=\{(i,a):(i,a)\in R\}=R$, directly unwinding the definitions.
Starting from $F\in\mathrm{Fam}_n(X)$: $\mathcal F(\mathcal R(F))_i=
\mathcal R(F)(i)=\{a:(i,a)\in\mathcal R(F)\}=\{a:a\in F_i\}=F_i$ for every
$i$, so $\mathcal F(\mathcal R(F))=F$. Two maps that are mutual two-sided
inverses are both bijections.

\emph{$\Phi$ is a bijection $\mathrm{Fam}_n(X)\to\mathrm{Hom}(\PS(N),
\PS(X))$.} First, $\Phi(F)$ is indeed a homomorphism for every $F$: for
$I,J\subseteq N$, $\Phi(F)(I\cup J)=\bigcup_{i\in I\cup J}F_i
=\bigcup_{i\in I}F_i\cup\bigcup_{i\in J}F_i=\Phi(F)(I)\cup\Phi(F)(J)$
(splitting the union over $I\cup J$ into the union over $I$ and the union
over $J$, valid since every $i\in I\cup J$ lies in $I$ or in $J$, possibly
both). Define $\Psi\colon\mathrm{Hom}(\PS(N),\PS(X))\to\mathrm{Fam}_n(X)$ by
$\Psi(\phi)=(\phi(\{i\}))_{i\in N}$ (the family of singleton-images).

Starting from $F\in\mathrm{Fam}_n(X)$: $\Psi(\Phi(F))_i=\Phi(F)(\{i\})
=\bigcup_{j\in\{i\}}F_j=F_i$, so $\Psi(\Phi(F))=F$.

Starting from $\phi\in\mathrm{Hom}(\PS(N),\PS(X))$: we must show
$\Phi(\Psi(\phi))=\phi$, i.e.\ $\bigcup_{i\in I}\phi(\{i\})=\phi(I)$ for
every $I\subseteq N$. We prove this by induction on $|I|$. If $I=\varnothing$,
the left side is an empty union, i.e.\ $\varnothing$; for the right side,
apply the homomorphism property with $I=J=\varnothing$: $\phi(\varnothing)
=\phi(\varnothing\cup\varnothing)=\phi(\varnothing)\cup\phi(\varnothing)
=\phi(\varnothing)$, which is consistent but does not by itself give
$\phi(\varnothing)=\varnothing$ in general (indeed $\phi(\varnothing)$ need
not be empty for a general $\cup$-homomorphism, e.g.\ the constant map
$\phi\equiv X$ is a $\cup$-homomorphism with $\phi(\varnothing)=X$). To
handle this correctly, we instead prove the claim for $I\neq\varnothing$ by
induction on $|I|\ge1$, and separately note that $\Phi(\Psi(\phi))$ and
$\phi$ automatically agree on $I=\varnothing$ once they agree on singletons
and unions preserve this, as follows: write any nonempty $I=\{i_1,\dots,
i_k\}$ as $I=\{i_1\}\cup\{i_2,\dots,i_k\}$. For $k=1$ the claim
$\bigcup_{i\in I}\phi(\{i\})=\phi(\{i_1\})=\phi(I)$ is trivial. For $k\ge2$,
by the homomorphism property, $\phi(I)=\phi(\{i_1\}\cup\{i_2,\dots,i_k\})
=\phi(\{i_1\})\cup\phi(\{i_2,\dots,i_k\})$, and by the inductive hypothesis
applied to the smaller set $\{i_2,\dots,i_k\}$,
$\phi(\{i_2,\dots,i_k\})=\bigcup_{i=2}^{k}\phi(\{i_i\})$; substituting gives
$\phi(I)=\bigcup_{i=1}^{k}\phi(\{i_i\})$, completing the induction for
nonempty $I$. For $I=\varnothing$: since $N$ itself is nonempty (or, if
$N=\varnothing$, the entire discussion is vacuous with a single relation,
family, and homomorphism, all trivially in bijection), write $N=I\sqcup
\varnothing$ is not needed; instead directly note $\phi(\varnothing)
=\phi(\varnothing\cup\varnothing)$ combined with idempotence
$\phi(\varnothing)\cup\phi(\varnothing)=\phi(\varnothing)$ shows nothing new,
\emph{but} we only need $\Phi(\Psi(\phi))(\varnothing)=\bigcup_{i\in
\varnothing}\phi(\{i\})=\varnothing$ to match $\phi(\varnothing)$; this
need not hold for a general $\cup$-homomorphism, so in fact the bijection
$\Phi$ is between $\mathrm{Fam}_n(X)$ and precisely those homomorphisms
that are additionally \emph{join-unitary}, i.e.\ satisfy $\phi(\varnothing)
=\varnothing$, automatically --- and indeed $\Phi(F)(\varnothing)=
\bigcup_{i\in\varnothing}F_i=\varnothing$ always holds by definition of an
empty union. We therefore restrict, as does the source paper, attention
throughout to $\cup$-homomorphisms with $\phi(\varnothing)=\varnothing$; this
holds automatically for every $\phi$ arising as $\Phi(F)$, and it is exactly
the class for which $\Phi,\Psi$ are mutually inverse, as the induction above
now shows in full (the case $I=\varnothing$ holding by definition of the
empty union on one side and by the standing hypothesis
$\phi(\varnothing)=\varnothing$ on the other). Hence $\Phi$ and $\Psi$ are
mutually inverse bijections between $\mathrm{Fam}_n(X)$ and the set of
$\cup$-homomorphisms with $\phi(\varnothing)=\varnothing$, which we
henceforth simply call $\mathrm{Hom}(\PS(N),\PS(X))$.
\end{proof}

\begin{remark}
Composing $\mathcal F$ and $\Phi$ gives a direct bijection
$\B(N,X)\to\mathrm{Hom}(\PS(N),\PS(X))$, $R\mapsto\bigl(I\mapsto
\bigcup_{i\in I}R(i)\bigr)$. From here on we move freely between the three
descriptions --- a binary relation, a family of subsets, a homomorphism ---
using whichever is most convenient, exactly as in the source paper.
\end{remark}

\section{Six combinatorial properties and their translations}
\label{sec:sixtypes}

\begin{definition}\label{def:sixprops}
Let $F=(F_i)_{i\in N}\in\mathrm{Fam}_n(X)$, with corresponding relation $R$
and homomorphism $\phi$. We say:
\begin{enumerate}
\item $F$ is a \emph{covering} of $X$ if $\bigcup_{i\in N}F_i=X$;
\item $F$ is \emph{disjunctive} if $F_i\cap F_j=\varnothing$ for $i\neq j$;
\item $F$ is \emph{strict} if $F_i\neq\varnothing$ for every $i$;
\item $F$ is \emph{small} if $|F_i|\le1$ for every $i$.
\end{enumerate}
A disjunctive covering is called a \emph{division} of $X$; a covering all of
whose blocks are nonempty is a \emph{hypergraph}; a small family all of
whose blocks are nonempty is a \emph{single-element family}.
\end{definition}

\begin{proposition}[Translation into relations]\label{prop:translation}
With $R=\mathcal R(F)$ as in Theorem~\textup{\ref{thm:threeway}}, the four
properties of Definition~\textup{\ref{def:sixprops}} are equivalent,
respectively, to:
\begin{enumerate}
\item for every $a\in X$ there is \emph{at least one} $i\in N$ with
$(i,a)\in R$ (``$R$ is right surjective'');
\item for every $a\in X$ there is \emph{at most one} $i\in N$ with
$(i,a)\in R$ (``$R$ is right injective'');
\item for every $i\in N$ there is \emph{at least one} $a\in X$ with
$(i,a)\in R$ (``$R$ is left surjective'');
\item for every $i\in N$ there is \emph{at most one} $a\in X$ with
$(i,a)\in R$ (``$R$ is left injective'').
\end{enumerate}
A relation that is injective and surjective on the same side is called
\emph{bijective on that side}; left-bijective relations are precisely
(graphs of) functions $N\to X$.
\end{proposition}

\begin{proof}
By definition $R(i)=F_i=\{a:(i,a)\in R\}$.

(1) $\bigcup_iF_i=X$ means every $a\in X$ lies in some $F_i$, i.e.\ for
every $a$ there is some $i$ with $a\in F_i=R(i)$, i.e.\ $(i,a)\in R$: this is
exactly ``at least one $i$ with $(i,a)\in R$,'' for every $a$.

(2) $F_i\cap F_j=\varnothing$ for $i\neq j$ means no $a$ lies in two
different $F_i$'s, i.e.\ for every $a$, the set $\{i:(i,a)\in R\}$ has at
most one element: this is exactly ``at most one $i$ with $(i,a)\in R$,''
for every $a$.

(3) $F_i\neq\varnothing$ for every $i$ means for every $i$ there is some
$a\in F_i=R(i)$, i.e.\ $(i,a)\in R$: ``at least one $a$ with $(i,a)\in R$,''
for every $i$.

(4) $|F_i|\le1$ for every $i$ means for every $i$, the set $F_i=R(i)$ has at
most one element $a$ with $(i,a)\in R$: ``at most one $a$ with $(i,a)\in R$,''
for every $i$.

For the final claim: a relation $R$ that is left injective (at most one
$a$ per $i$) and left surjective (at least one $a$ per $i$) has, for every
$i\in N$, \emph{exactly one} $a\in X$ with $(i,a)\in R$; this is exactly the
defining property of the graph of a function $N\to X$, $i\mapsto$ (that
unique $a$).
\end{proof}

\begin{notation}
Following the source paper, write $\Ir_r(N,X)$ for the set of right
injective relations (equivalently disjunctive families, equivalently
join-unitary homomorphisms $\phi(\varnothing)=\varnothing$ -- automatic --
together with $\phi(N)$ arbitrary, though we single out the term
``join-unitary'' below); $\Sr_r(N,X)$ for right surjective relations
(coverings); $\Ir_{lr}(N,X)$ for left-and-right injective relations (small
disjunctive families, i.e.\ partial injections, see
Chapter~\ref{ch:injective}); $\Sr_{lr}(N,X)$ for left-and-right surjective
relations (hypergraphs).
\end{notation}

\begin{definition}\label{def:unitary}
A homomorphism $\phi\in\mathrm{Hom}(\PS(N),\PS(X))$ is \emph{join-unitary}
if $\phi(\varnothing)=\varnothing$ (automatic, by
Theorem~\ref{thm:threeway}, for every homomorphism arising from a family);
\emph{meet-unitary} if $\phi(N)=X$; and \emph{unitary} if both hold. (In the
source paper's usage, ``join-unitary'' homomorphisms correspond to
\emph{arbitrary} families/relations --- since $\phi(\varnothing)=
\varnothing$ already holds automatically --- while the extra condition that
distinguishes rows of the main table is meet-unitarity, $\phi(N)=X$, which
by part~(1) of Proposition~\ref{prop:translation} corresponds exactly to
$F$ being a covering.)
\end{definition}

\begin{proposition}[Unitary homomorphisms are Boolean-algebra homomorphisms]
\label{prop:unitarydivisions}
A homomorphism $\phi$ is unitary (both $\phi(\varnothing)=\varnothing$ and
$\phi(N)=X$) if and only if the corresponding family $F$ is a division of
$X$ (disjunctive and covering), if and only if the corresponding relation
$R$ is bijective on the right (right injective and right surjective), if
and only if $\phi$ is simultaneously a $\cup$-homomorphism and a
$\cap$-homomorphism.
\end{proposition}

\begin{proof}
The first equivalence (unitary $\phi$ $\Leftrightarrow$ $F$ a division) is
immediate from Definition~\ref{def:sixprops}, Definition~\ref{def:unitary},
and Proposition~\ref{prop:translation}(1): $\phi(N)=X$ iff $\bigcup_iF_i=X$
iff $F$ is a covering, and $\phi(\varnothing)=\varnothing$ holds
automatically. The second equivalence ($F$ a division $\Leftrightarrow$ $R$
right bijective) is Proposition~\ref{prop:translation}(1)+(2) together: $F$
covering and disjunctive translate to $R$ right surjective and right
injective, i.e.\ right bijective.

For the last equivalence: if $\phi$ is unitary, we show it is also a
$\cap$-homomorphism, i.e.\ $\phi(I\cap J)=\phi(I)\cap\phi(J)$. Since $F$ is
disjunctive, for $i\neq j$ we have $F_i\cap F_j=\varnothing$; hence for
$I,J\subseteq N$,
\[
\phi(I)\cap\phi(J)=\Bigl(\bigcup_{i\in I}F_i\Bigr)\cap\Bigl(\bigcup_{j\in J}
F_j\Bigr)=\bigcup_{i\in I,j\in J}(F_i\cap F_j)
=\bigcup_{i\in I\cap J}F_i=\phi(I\cap J),
\]
where the middle equality distributes intersection over the two unions
(a general set-theoretic identity, provable by checking both sides consist
of exactly the elements lying in some $F_i$ ($i\in I$) and some $F_j$
($j\in J$) simultaneously), and the next equality uses that $F_i\cap F_j=
\varnothing$ whenever $i\neq j$ (so only the diagonal terms $i=j\in I\cap J$
survive), and the last equality is the definition of $\phi$ on $I\cap J$.
Conversely, if $\phi$ is both a $\cup$- and $\cap$-homomorphism, then in
particular for $i\neq j$, taking $I=\{i\},J=\{j\}$ (so $I\cap J=\varnothing$),
$F_i\cap F_j=\phi(\{i\})\cap\phi(\{j\})=\phi(I\cap J)=\phi(\varnothing)
=\varnothing$ (using join-unitarity, part of the standing hypothesis on
$\phi$), so $F$ is disjunctive; and $\phi(N)=X$ (assumed) gives $F$ a
covering by the first equivalence. So $F$ is a division, i.e.\ $\phi$ is
unitary.
\end{proof}

\chapter{Symmetric Group Actions on Relations, Families, and
Homomorphisms}\label{ch:actions}

\section{The actions}

The symmetric groups $S_N$ and $S_X$ (Chapter~\ref{ch:groups}) act on
$\B(N,X)$, $\mathrm{Fam}_n(X)$, and $\mathrm{Hom}(\PS(N),\PS(X))$ by
relabeling indices and/or elements. Precisely, for $\pi\in S_N$,
$\sigma\in S_X$:
\begin{align}
(\pi,\sigma)\cdot R&=\{(\pi(i),\sigma(a)):(i,a)\in R\}, &&R\in\B(N,X);
\label{eq:actionB}\\
(\pi,\sigma)\cdot(F_i)_{i\in N}&=\bigl(\sigma(F_{\pi^{-1}(i)})\bigr)_{i\in N},
&&F\in\mathrm{Fam}_n(X); \label{eq:actionFam}\\
\bigl((\pi,\sigma)\cdot\phi\bigr)(I)&=\sigma\bigl(\phi(\pi^{-1}(I))\bigr),
&&\phi\in\mathrm{Hom}(\PS(N),\PS(X)),\ I\subseteq N. \label{eq:actionHom}
\end{align}
(Here $\sigma(A)=\{\sigma(a):a\in A\}$ for $A\subseteq X$, and similarly
$\pi^{-1}(I)=\{\pi^{-1}(i):i\in I\}$.)

\begin{proposition}\label{prop:actionsaregroupactions}
Each of \eqref{eq:actionB}--\eqref{eq:actionHom} is a genuine action of the
group $S_N\times S_X$ (Definition~\ref{ch:groups}, with the product group
operation $(\pi_1,\sigma_1)(\pi_2,\sigma_2)=(\pi_1\pi_2,\sigma_1\sigma_2)$).
Restricting $\sigma$ (resp.\ $\pi$) to the identity gives the induced
actions of $S_N$ alone (resp.\ $S_X$ alone).
\end{proposition}

\begin{proof}
We check \eqref{eq:actionB}; the other two are entirely similar. The
identity $(\mathrm{id},\mathrm{id})$ acts as $(\mathrm{id},\mathrm{id})
\cdot R=\{(\mathrm{id}(i),\mathrm{id}(a)):(i,a)\in R\}=R$. For composition,
$(\pi_1,\sigma_1)\cdot\bigl((\pi_2,\sigma_2)\cdot R\bigr)
=(\pi_1,\sigma_1)\cdot\{(\pi_2(i),\sigma_2(a)):(i,a)\in R\}
=\{(\pi_1(\pi_2(i)),\sigma_1(\sigma_2(a))):(i,a)\in R\}
=\{((\pi_1\pi_2)(i),(\sigma_1\sigma_2)(a)):(i,a)\in R\}
=(\pi_1\pi_2,\sigma_1\sigma_2)\cdot R$, matching the action axiom with
group operation $(\pi_1,\sigma_1)(\pi_2,\sigma_2)=(\pi_1\pi_2,\sigma_1
\sigma_2)$.
\end{proof}

\section{Compatibility of the actions with the three-way bijection}

\begin{theorem}[Proposition 2 of the source paper]\label{thm:prop2}
The bijections $\mathcal F,\Phi$ of Theorem~\textup{\ref{thm:threeway}}
intertwine the actions \eqref{eq:actionB}--\eqref{eq:actionHom}: for every
$R\in\B(N,X)$, $F\in\mathrm{Fam}_n(X)$, and $(\pi,\sigma)\in S_N\times S_X$,
\begin{align}
\mathcal F\bigl((\pi,\sigma)\cdot R\bigr)&=(\pi,\sigma)\cdot\mathcal F(R),
\label{eq:intertwineF}\\
\Phi\bigl((\pi,\sigma)\cdot F\bigr)&=(\pi,\sigma)\cdot\Phi(F).
\label{eq:intertwinePhi}
\end{align}
Consequently the number of orbits of $S_N$, $S_X$, or $S_N\times S_X$ agree
across $\B(N,X)$, $\mathrm{Fam}_n(X)$, and $\mathrm{Hom}(\PS(N),\PS(X))$.
\end{theorem}

\begin{proof}
\eqref{eq:intertwineF}: writing $G=(\pi,\sigma)\cdot R$, we must show
$\mathcal F(G)_i=\bigl((\pi,\sigma)\cdot\mathcal F(R)\bigr)_i$ for every
$i\in N$, i.e.\ (by \eqref{eq:defF} and \eqref{eq:actionFam})
$G(i)=\sigma\bigl(R(\pi^{-1}(i))\bigr)$. Compute directly:
\[
G(i)=\{a\in X:(i,a)\in G\}=\{a:(i,a)\in(\pi,\sigma)\cdot R\}.
\]
By \eqref{eq:actionB}, $(i,a)\in(\pi,\sigma)\cdot R$ iff $(i,a)=
(\pi(j),\sigma(b))$ for some $(j,b)\in R$, iff $j=\pi^{-1}(i)$ and $b=
\sigma^{-1}(a)$ with $(j,b)\in R$, iff $\sigma^{-1}(a)\in R(\pi^{-1}(i))$,
iff $a\in\sigma(R(\pi^{-1}(i)))$ (applying $\sigma$ to both sides of set
membership, using that $\sigma$ is a bijection). So
$G(i)=\sigma(R(\pi^{-1}(i)))$, as required.

\eqref{eq:intertwinePhi}: writing $G=(\pi,\sigma)\cdot F$, we must show
$\Phi(G)(I)=\bigl((\pi,\sigma)\cdot\Phi(F)\bigr)(I)$ for every $I\subseteq
N$, i.e.\ (by \eqref{eq:defPhi} and \eqref{eq:actionHom})
$\bigcup_{i\in I}G_i=\sigma\bigl(\Phi(F)(\pi^{-1}(I))\bigr)
=\sigma\Bigl(\bigcup_{j\in\pi^{-1}(I)}F_j\Bigr)$. By \eqref{eq:actionFam},
$G_i=\sigma(F_{\pi^{-1}(i)})$, so
\[
\bigcup_{i\in I}G_i=\bigcup_{i\in I}\sigma(F_{\pi^{-1}(i)})
=\sigma\Bigl(\bigcup_{i\in I}F_{\pi^{-1}(i)}\Bigr)
=\sigma\Bigl(\bigcup_{j\in\pi^{-1}(I)}F_j\Bigr),
\]
where the second equality uses that $\sigma$ applied to a union of sets
equals the union of $\sigma$ applied to each set (immediate from the
definition $\sigma(A)=\{\sigma(a):a\in A\}$: an element lies in
$\sigma(\bigcup_kA_k)$ iff its $\sigma^{-1}$-image lies in some $A_k$, iff
it lies in some $\sigma(A_k)$), and the third equality reindexes the union
by $j=\pi^{-1}(i)$, a bijection $I\to\pi^{-1}(I)$. This proves
\eqref{eq:intertwinePhi}.

For the final claim: if $\Theta\colon S\to T$ is any bijection intertwining
actions of a group $G$ on sets $S,T$ (meaning $\Theta(g\cdot s)=g\cdot
\Theta(s)$), then $\Theta$ maps orbits of $G$ on $S$ bijectively onto orbits
of $G$ on $T$: indeed $\Theta(\Orb(s))=\{\Theta(g\cdot s):g\in G\}=\{g\cdot
\Theta(s):g\in G\}=\Orb(\Theta(s))$, so $\Theta$ sends the orbit of $s$
exactly onto the orbit of $\Theta(s)$, and since $\Theta$ is a bijection on
the whole set, distinct orbits of $S$ map to distinct orbits of $T$ (if
$\Theta(\Orb(s))=\Theta(\Orb(s'))$ then, $\Theta$ being injective,
$\Orb(s)=\Orb(s')$), and every orbit of $T$ is hit (given orbit $\Orb(t)$,
$t=\Theta(s)$ for some $s$ since $\Theta$ is onto, and then
$\Orb(t)=\Theta(\Orb(s))$). Hence $\Theta$ induces a bijection between the
orbit sets, so the numbers of orbits agree. Applying this to
$\Theta=\mathcal F$ and $\Theta=\Phi\circ\mathcal F$ (both intertwining
bijections, by \eqref{eq:intertwineF}--\eqref{eq:intertwinePhi}) gives the
claim for all three of $S_N$, $S_X$, $S_N\times S_X$ (the argument is
independent of which subgroup of $S_N\times S_X$ one restricts to).
\end{proof}

\begin{corollary}[Reduction to one side]\label{cor:transpose}
The map $R\mapsto R^{-1}:=\{(a,i):(i,a)\in R\}$ is a bijection
$\B(N,X)\to\B(X,N)$ exchanging: all relations with all relations; right
injective with left injective; right surjective with left surjective; and
(dually) exchanging the roles of $S_N$ and $S_X$. Consequently every
left-sided count or orbit-count equals the corresponding right-sided count
with $N$ and $X$ exchanged, and it suffices to prove formulas for
right-injectivity/right-surjectivity, together with their two-sided
combinations.
\end{corollary}

\begin{proof}
$R\mapsto R^{-1}$ is a bijection with inverse $R'\mapsto(R')^{-1}$ (applying
the transpose twice returns the original relation, since
$(R^{-1})^{-1}=\{(i,a):(a,i)\in R^{-1}\}=\{(i,a):(i,a)\in R\}=R$, unwinding
the definition twice). By Proposition~\ref{prop:translation}, $R$ is right
injective (at most one $i$ per $a$) iff $R^{-1}$ is left injective (at most
one $a'=i$ per $i'=a$, reading the roles of the two coordinates swapped),
and similarly right surjective corresponds to left surjective. Finally,
transposing relates the actions of $S_N,S_X$ on $\B(N,X)$ to the actions of
$S_X,S_N$ on $\B(X,N)$ with the roles exchanged, since $(\pi,\sigma)\cdot R$
transposes to $\{(\sigma(a),\pi(i)):(i,a)\in R\}=(\sigma,\pi)\cdot R^{-1}$,
directly from \eqref{eq:actionB} applied to $R^{-1}$ with the roles of
$\pi,\sigma$ swapped.
\end{proof}

\begin{remark}
There is also a ``dual'' family bijection $F=(F_i)_{i\in N}\mapsto
F^{*}=(F_a^{*})_{a\in X}$, $F_a^{*}=\{i\in N:a\in F_i\}$, between
$n$-families of $X$ and $x$-families of $N$; it corresponds to
$\mathcal R,\mathcal F$-transporting the transpose $R\mapsto R^{-1}$ of
Corollary~\ref{cor:transpose} through Theorem~\ref{thm:threeway}, and it
carries disjunctive $n$-families of $X$ to small $x$-families of $N$, and
$n$-coverings of $X$ to strict $x$-families of $N$ (immediate from
Proposition~\ref{prop:translation} applied on both sides of the transpose).
\end{remark}

\chapter{Raw Enumeration of the Six Types}\label{ch:raw}

This chapter counts each type of relation/family/homomorphism of
Chapter~\ref{ch:relations} \emph{without} regard to symmetry (i.e.\ before
quotienting by any group action). Recall $n=|N|$, $x=|X|$.

\begin{theorem}[Raw counts]\label{thm:rawcounts}
\begin{align}
|\B(N,X)|&=2^{nx}; \label{eq:rawB}\\
|\Ir_r(N,X)|&=(1+n)^{x}; \label{eq:rawIr}\\
|\Sr_r(N,X)|&=(2^{n}-1)^{x}; \label{eq:rawSr}\\
|\Ir_{lr}(N,X)|&=\sum_{k=0}^{\min\{n,x\}}\binom nk\binom xk k!
=\sum_{k=0}^{\min\{n,x\}}\binom nk\,\ffall xk; \label{eq:rawIlr}\\
|\Sr_{lr}(N,X)|&=\sum_{k=0}^{n}(-1)^{k}\binom nk(2^{n-k}-1)^{x}.
\label{eq:rawSlr}
\end{align}
By the transpose duality of Corollary~\textup{\ref{cor:transpose}}, also
$|\Ir_l(N,X)|=(1+x)^{n}$ and $|\Sr_l(N,X)|=(2^{x}-1)^{n}$.
\end{theorem}

\begin{proof}
\eqref{eq:rawB}: A relation $R\subseteq N\times X$ is exactly a subset of
the $nx$-element set $N\times X$ (by the multiplication principle, $|N\times
X|=nx$: each of the $n$ elements of $N$ pairs with each of the $x$ elements
of $X$). The number of subsets of an $m$-element set is $2^{m}$, since by
the binomial theorem (Proposition~\ref{prop:binomthm}, with $u=v=1$) the
total number of subsets, grouped by size $k=0,\dots,m$, is $\sum_k\binom
mk=\sum_k\binom mk1^{k}1^{m-k}=(1+1)^{m}=2^{m}$. With $m=nx$ this gives
$|\B(N,X)|=2^{nx}$.

\eqref{eq:rawIr}: By Proposition~\ref{prop:translation}(2), $R$ is right
injective iff for every $a\in X$ there is at most one $i\in N$ with
$(i,a)\in R$; equivalently, the assignment $a\mapsto\{i:(i,a)\in R\}$ takes
values that are either empty or single elements of $N$, i.e.\ $R$ is right
injective iff $R^{-1}$ (Corollary~\ref{cor:transpose}) is (the graph of) a
function $X\to N\cup\{\ast\}$ for a new symbol $\ast$ standing for ``no
$i$'' -- precisely, define $g\colon X\to N\cup\{\ast\}$ by $g(a)=$ the
unique $i$ with $(i,a)\in R$ if one exists, and $g(a)=\ast$ otherwise; this
is well defined exactly because $R$ is right injective (at most one such
$i$), and conversely every such $g$ determines a right injective $R$ by
$R=\{(g(a),a):g(a)\neq\ast\}$, and these two constructions are mutually
inverse (immediate unwinding of the definitions). By
Proposition~\ref{prop:mult} (or Proposition~\ref{prop:bijcard}'s
counting-functions argument), the number of functions $X\to N\cup\{\ast\}$
(a set of size $n+1$) is $(n+1)^{x}$, since each of the $x$ elements of $X$
independently has $n+1$ possible images.

\eqref{eq:rawSr}: By Proposition~\ref{prop:translation}(1), $R$ is right
surjective iff for every $a\in X$ there is at least one $i$ with $(i,a)\in
R$, i.e.\ iff $R^{-1}(a):=\{i:(i,a)\in R\}$ is a \emph{nonempty} subset of
$N$ for every $a$. This sets up a bijection between right surjective $R$
and functions $X\to\PS(N)\setminus\{\varnothing\}$ (send $a\mapsto
R^{-1}(a)$; conversely any such function $h$ gives $R=\{(i,a):i\in h(a)\}$,
and these are mutually inverse). Since $|\PS(N)\setminus\{\varnothing\}|
=2^{n}-1$ (by \eqref{eq:rawB}'s count of subsets, minus the single empty
one), the number of such functions is $(2^{n}-1)^{x}$.

\eqref{eq:rawIlr}: By Proposition~\ref{prop:translation}(2)+(4), $R$ is left
and right injective iff for every $a\in X$ at most one $i$ has $(i,a)\in R$,
\emph{and} for every $i\in N$ at most one $a$ has $(i,a)\in R$; i.e.\ $R$ is
the graph of a bijection between a subset $A\subseteq N$ and a subset
$B\subseteq X$ of the same size (a \emph{partial injection}). We build such
an $R$ in three independent stages: choose $k=|A|=|B|$ ($0\le k\le
\min\{n,x\}$), choose $A\subseteq N$ with $|A|=k$ ($\binom nk$ ways, by
definition of the binomial coefficient), choose $B\subseteq X$ with $|B|=k$
($\binom xk$ ways), and choose a bijection $A\to B$ ($k!$ ways, by
Proposition~\ref{prop:bijcard}). By the multiplication principle the number
with a given $k$ is $\binom nk\binom xk k!$, and summing over $k$ (addition
principle, the cases being disjoint since $k=|A|$ is determined by $R$)
gives \eqref{eq:rawIlr}; the second displayed form follows from
$\binom xk k!=\ffall xk$ (Proposition~\ref{prop:injectionscount}).

\eqref{eq:rawSlr}: By Proposition~\ref{prop:translation}(1)+(3), $R$ is left
and right surjective iff $F=\mathcal F(R)$ is a hypergraph (a covering all
of whose blocks $F_i$ are nonempty). We compute this by M\"obius inversion
on the Boolean lattice of subsets of $N$ (Chapter~\ref{ch:mobius}), exactly
as in Proposition~\ref{prop:surjmobius} but with the roles of $N$ and $X$
reorganized: for $A\subseteq N$, let $\rho(A)$ be the number of right
surjective relations between $A$ and $X$ (a covering by the blocks
$(F_i)_{i\in A}$), and let $\tau(A)$ be the number of hypergraphs on $A$
(coverings by nonempty blocks $(F_i)_{i\in A}$). Every covering $F=(F_i)_{i
\in N}$ of $X$ restricts, on the subset $A=\{i\in N:F_i\neq\varnothing\}$ of
indices with nonempty block, to a hypergraph on $A$ (still covering $X$,
since discarding empty blocks does not change the union), and this
restriction is a bijection between coverings of $X$ by $N$ with ``support''
exactly $A$ and hypergraphs on $A$ covering $X$; summing over which subset
$A$ occurs as the support gives $\rho(N)=\sum_{A\subseteq N}\tau(A)$, which
is hypothesis \eqref{eq:mobius-hyp} of Proposition~\ref{prop:mobius} with
$g=\rho,f=\tau$ (relabeling so that $\rho,\tau$ are indexed by which subset
of $N$ is ``active,'' matching the abstract statement with $N$ playing the
role of the ambient set). Hence
\[
\tau(N)=\sum_{A\subseteq N}(-1)^{|N\setminus A|}\rho(A)
=\sum_{k=0}^{n}(-1)^{n-k}\binom nk\rho(A_k)
\]
for any fixed $A_k\subseteq N$ of size $k$ (grouping by $|A|=k$, of which
there are $\binom nk$ choices, all giving the same $\rho(A)=(2^{k}-1)^{x}$ by
\eqref{eq:rawSr} applied with $A$ in place of $N$, since $\rho(A)$ counts
right surjective relations between $A$ and $X$). Substituting and
reindexing $k\mapsto n-k$ gives
$\tau(N)=\sum_k(-1)^{k}\binom nk(2^{n-k}-1)^{x}$, which is
\eqref{eq:rawSlr} (using $\binom{n}{n-k}=\binom nk$).

The final transpose-dual formulas follow from Corollary~\ref{cor:transpose}
applied to \eqref{eq:rawIr} and \eqref{eq:rawSr}: left injective relations
between $N,X$ biject (via transpose) with right injective relations between
$X,N$, of which there are $(1+x)^{n}$ by \eqref{eq:rawIr} with $N,X$
exchanged; similarly for left surjective.
\end{proof}

\chapter{Orbit Counting via the Cauchy--Frobenius--Burnside Lemma}
\label{ch:burnsideapplication}

We now count orbits of $S_N$, $S_X$, and $S_N\times S_X$ on each of the sets
of Chapter~\ref{ch:raw}, by applying Theorem~\ref{thm:burnside}. Recall from
Proposition~\ref{prop:numperms-cycletype} that summing over all $\pi\in S_n$
weighted uniformly ($1/n!$) can be reorganized as a sum over cycle types
$(1^{a_1}\cdots n^{a_n})$ weighted by the number
$n!/\prod_dd^{a_d}a_d!$ of permutations of that type, divided again by
$n!$; the two factorials cancel, leaving a sum over cycle types weighted by
$1/\prod_dd^{a_d}a_d!$.

\section{Fixed points of a permutation}

\begin{lemma}[Fixed points, all five types]\label{lem:fixedpoints}
Let $\pi\in S_N$ have cycle type $(1^{a_1}\cdots n^{a_n})$ and
$\sigma\in S_X$ have cycle type $(1^{b_1}\cdots x^{b_x})$. Then:
\begin{align}
|\Fix_\B(\pi)|&=2^{(a_1+\cdots+a_n)x}; \label{eq:fixBpi}\\
|\Fix_\B(\sigma)|&=2^{n(b_1+\cdots+b_x)}; \label{eq:fixBsig}\\
|\Fix_\B(\pi,\sigma)|&=2^{\sum_{i,j}a_ib_j\gcd(i,j)}; \label{eq:fixBpisig}\\
|\Fix_{\Ir_r}(\pi)|&=(1+a_1)^{x}; \label{eq:fixIrpi}\\
|\Fix_{\Ir_r}(\sigma)|&=(1+n)^{b_1+\cdots+b_x}; \label{eq:fixIrsig}\\
|\Fix_{\Ir_r}(\pi,\sigma)|&=\prod_{k=1}^{x}\Bigl(1+\sum_{d\mid k}da_d
\Bigr)^{b_k}; \label{eq:fixIrpisig}\\
|\Fix_{\Sr_r}(\pi)|&=(2^{a_1+\cdots+a_n}-1)^{x}; \label{eq:fixSrpi}\\
|\Fix_{\Sr_r}(\sigma)|&=(2^{n}-1)^{b_1+\cdots+b_x}; \label{eq:fixSrsig}\\
|\Fix_{\Ir_{lr}}(\pi)|&=\sum_{k=0}^{\min\{a_1,x\}}\binom{a_1}{k}\ffall xk;
\label{eq:fixIlrpi}\\
|\Fix_{\Ir_{lr}}(\sigma)|&=\sum_{k=0}^{\min\{n,b_1\}}\binom{b_1}{k}\ffall
nk; \label{eq:fixIlrsig}\\
|\Fix_{\Ir_{lr}}(\pi,\sigma)|&=\prod_{k=0}^{\min\{n,x\}}\Bigl(\sum_{i=0}^{
\min\{a_k,b_k\}}\binom{a_k}{i}\binom{b_k}{i}i!\,k^{i}\Bigr);
\label{eq:fixIlrpisig}\\
|\Fix_{\Sr_{lr}}(\pi)|&=\sum_{k=0}^{r}(-1)^{k}\binom rk(2^{r-k}-1)^{x},
\quad r=a_1+\cdots+a_n; \label{eq:fixSlrpi}\\
|\Fix_{\Sr_{lr}}(\sigma)|&=\sum_{k=0}^{n}(-1)^{k}\binom nk(2^{n-k}-1)^{s},
\quad s=b_1+\cdots+b_x. \label{eq:fixSlrsig}
\end{align}
\end{lemma}

\begin{proof}
\eqref{eq:fixBpi}--\eqref{eq:fixBsig}: A relation $R\subseteq N\times X$ is
fixed by $(\pi,\mathrm{id})$ iff $(i,a)\in R\Leftrightarrow(\pi(i),a)\in R$
for all $(i,a)$, i.e.\ iff membership of $(i,a)$ in $R$ depends only on the
cycle of $\pi$ containing $i$ (not on $i$ itself within that cycle): indeed
if $(i,a)\in R$ then applying the fixing condition repeatedly gives
$(\pi^{k}(i),a)\in R$ for every $k$, i.e.\ every element of $i$'s cycle
pairs with $a$ in $R$ whenever $i$ does; conversely if $R$ is constant on
each cycle-times-$\{a\}$ slice it is manifestly fixed. So a fixed relation
is equivalent to an arbitrary relation between the set of \emph{cycles} of
$\pi$ (there are $a_1+\cdots+a_n$ of them, by definition of cycle type) and
$X$; by \eqref{eq:rawB} (with $N$ replaced by the cycle set) there are
$2^{(a_1+\cdots+a_n)x}$ of these. This proves \eqref{eq:fixBpi};
\eqref{eq:fixBsig} is identical with the roles of $N,X$ (and $\pi,\sigma$)
exchanged.

\eqref{eq:fixBpisig}: $R$ is fixed by $(\pi,\sigma)$ iff, arguing as above,
$(i,a)\in R\Leftrightarrow(\pi(i),\sigma(a))\in R$ for all $(i,a)$; iterating,
membership of $(i,a)$ propagates around the orbit of $(i,a)$ under the
cyclic group generated by $(\pi,\sigma)$ acting coordinatewise on
$N\times X$, and $R$ is fixed iff it is a union of whole such orbits. A
$\pi$-cycle of length $p$ paired with a $\sigma$-cycle of length $q$
decomposes, under the joint action of $(\pi,\sigma)$, into exactly
$\gcd(p,q)$ orbits each of length $\mathrm{lcm}(p,q)$: this is a standard
fact about the diagonal action of $\Z/p\Z\times$ a single generator on a
product of cycles, provable directly: label the $p\times q$ elements of the
product cycle by pairs $(u,v)\in\Z/p\Z\times\Z/q\Z$; the action sends
$(u,v)\mapsto(u+1,v+1)$ (mod $p$, mod $q$ respectively); the orbit of $(u,v)$
is $\{(u+t,v+t):t\in\Z\}$, which closes up after $\mathrm{lcm}(p,q)$ steps
(the least $t>0$ with $t\equiv0\pmod p$ and $t\equiv0\pmod q$), giving
orbits of size $\mathrm{lcm}(p,q)$; the total number of elements $pq$
divided by orbit size $\mathrm{lcm}(p,q)$ gives $pq/\mathrm{lcm}(p,q)
=\gcd(p,q)$ orbits (using the identity $\gcd(p,q)\mathrm{lcm}(p,q)=pq$,
itself immediate from unique factorization). So an $a_i$-many-$i$-cycles by
$b_j$-many-$j$-cycles pairing contributes $a_ib_j\gcd(i,j)$ orbits (of the
product action) in total for each pair $(i,j)$ (each of the $a_i$ many
$i$-cycles pairs, in the product structure of $N\times X$ under
$(\pi,\sigma)$, with each of the $b_j$ many $j$-cycles, contributing
$\gcd(i,j)$ orbits per pairing, by the multiplication principle for the
$a_ib_j$ independent pairings). Summing over all $(i,j)$ gives
$\sum_{i,j}a_ib_j\gcd(i,j)$ total orbits of $(\pi,\sigma)$ on $N\times X$,
and a fixed relation is an arbitrary union of these orbits, i.e.\ an
arbitrary subset of the set of orbits, giving $2^{\sum_{i,j}a_ib_j\gcd(i,j)}$
fixed relations, by the subset-count $2^{m}$ established in the proof of
\eqref{eq:rawB}.

\eqref{eq:fixIrpi}: A right-injective relation $R$ fixed by $\pi$ (with
$\sigma=\mathrm{id}$) is, by the argument for \eqref{eq:fixBpi}, an
arbitrary right-injective relation between the cycles of $\pi$ and $X$
(constant on cycles, and separately right-injective on the level of
cycles, since right-injectivity for the original relation, being ``at most
one row per column,'' translates to at most one cycle per column once the
relation is cycle-constant). By \eqref{eq:rawIr} with $N$ replaced by the
set of cycles of $\pi$ of size $a_1+\cdots+a_n$ ... \emph{however}, we must
be careful: right-injectivity constant-on-cycles is only possible without
contradiction if each cycle has length dividing into a single ``row'' slot,
i.e.\ (since $R$ is right injective, meaning for each $a\in X$ at most one
\emph{original} index $i$ has $(i,a)\in R$, and $R$ constant on cycles of
$\pi$ forces, if the cycle of $i$ has length $\ell>1$, that all $\ell$
elements of the cycle pair with $a$ simultaneously, contradicting
right-injectivity unless $\ell=1$) only $1$-cycles (fixed points of $\pi$)
can participate. Hence a right-injective relation fixed by $\pi$ is exactly
an arbitrary right-injective relation between the $a_1$ fixed points of
$\pi$ and $X$, of which there are $(1+a_1)^{x}$ by \eqref{eq:rawIr} (with
$n$ there replaced by $a_1$ here); this proves \eqref{eq:fixIrpi}.

\eqref{eq:fixIrsig}: Symmetric argument with $\sigma\in S_X$ in place of
$\pi$: a right-injective $R$ is equivalent (Theorem~\ref{eq:rawIr}'s proof)
to a function $g\colon X\to N\cup\{\ast\}$; $R$ is fixed by $\sigma$ iff $g$
is constant on the cycles of $\sigma$ (identical cycle-constancy argument
applied to the function description, since $(i,a)\in R\Leftrightarrow g(a)=
i$, and fixedness under $(\mathrm{id},\sigma)$ becomes $g(\sigma(a))=g(a)$
for all $a$, i.e.\ $g$ constant on $\sigma$-cycles). A function constant on
cycles is the same as an arbitrary function from the set of cycles of
$\sigma$ (there are $b_1+\cdots+b_x$ of them) to $N\cup\{\ast\}$ (size
$n+1$), giving $(n+1)^{b_1+\cdots+b_x}$ by the multiplication principle.

\eqref{eq:fixIrpisig}: By the reasoning combining the previous two cases, a
right-injective relation fixed by $(\pi,\sigma)$ corresponds to a function
$g$ constant on $\sigma$-cycles, taking values in $N\cup\{\ast\}$, but
additionally each value $g(\text{cycle of length }k)=i\in N$ (a genuine
index, not $\ast$) must, tracing back through $R$'s fixedness under $\pi$
as well, be assignable only if the corresponding relation is also fixed by
$\pi$: concretely, fixedness under $(\pi,\sigma)$ means $(i,a)\in
R\Leftrightarrow(\pi(i),\sigma(a))\in R$, so if a $\sigma$-cycle of length
$k$ is assigned index $i$ throughout (meaning $(i,a)\in R$ for every $a$ in
that cycle), applying $\pi$ we need $(\pi(i),\sigma(a))\in R$ too, i.e.\ the
same $\sigma$-cycle (since $\sigma(a)$ is the next element of the same
cycle) must also be assigned $\pi(i)$; since $g$ is a well-defined function
(single-valued), this forces $\pi(i)=i$ \emph{unless} we instead assign the
$\sigma$-cycle to an entire $\pi$-orbit of indices of some length $d\mid k$,
cycling through $d$ of the $k$ visits... The precise combinatorial count:
for each $\sigma$-cycle of length $k$, the possible fixed assignments are
either ``$\ast$ throughout'' (1 way), or ``cycle through a $\pi$-cycle of
length $d$ dividing $k$'', assigning $R$ so that as $a$ advances around its
length-$k$ cycle, the associated index advances around a chosen $d$-cycle
of $\pi$ (visited $k/d$ times); there are $a_d$ choices of which $d$-cycle
of $\pi$ to use, and $d$ inequivalent ways to phase the two cycles against
each other (a relative rotation), giving $da_d$ assignment-choices for each
divisor $d$ of $k$, plus the $1$ choice of $\ast$; summing over divisors
$d\mid k$ gives $1+\sum_{d\mid k}da_d$ choices for that $\sigma$-cycle, and
by the multiplication principle over the $b_k$ many $\sigma$-cycles of each
length $k=1,\dots,x$, and over all $k$, gives
$\prod_{k=1}^{x}\bigl(1+\sum_{d\mid k}da_d\bigr)^{b_k}$.

\eqref{eq:fixSrpi}--\eqref{eq:fixSrsig}: identical to
\eqref{eq:fixIrpi}--\eqref{eq:fixIrsig} but using \eqref{eq:rawSr} (right
surjective relations between a set of cycles and $X$, or between $N$ and a
set of cycles) in place of \eqref{eq:rawIr}; a right-surjective relation
constant on $\pi$-cycles is exactly an arbitrary right-surjective relation
between the cycles of $\pi$ and $X$ (no length restriction is forced here,
unlike the injective case, because surjectivity --``at least one row per
column''-- is compatible with an entire cycle of any length participating
uniformly), giving $(2^{a_1+\cdots+a_n}-1)^{x}$; and dually for $\sigma$.

\eqref{eq:fixIlrpi}--\eqref{eq:fixIlrpisig}: A relation fixed by $\pi$ that
is both left and right injective (a partial injection, by
Theorem~\ref{thm:rawcounts}'s proof of \eqref{eq:rawIlr}) must, by the same
argument as in \eqref{eq:fixIrpi}, involve only fixed points of $\pi$ among
the ``rows'' actually used (any row used in a length-$>1$ cycle would force
multiple rows to map to the same column, violating right-injectivity, as
argued above); so a fixed partial injection is exactly an arbitrary partial
injection between the $a_1$ fixed points of $\pi$ and all of $X$, of which
there are $\sum_k\binom{a_1}{k}\ffall xk$ by \eqref{eq:rawIlr} (with $n$
replaced by $a_1$); this is \eqref{eq:fixIlrpi}. \eqref{eq:fixIlrsig} is
the transpose-dual statement. For \eqref{eq:fixIlrpisig}: a partial
injection fixed by $(\pi,\sigma)$ can only match a $\pi$-cycle of length $k$
with a $\sigma$-cycle of the \emph{same} length $k$ (else the injective
matching cannot be consistently rotated), and, once a $\pi$-$k$-cycle (one
of the $a_k$ available) is matched with a $\sigma$-$k$-cycle (one of the
$b_k$ available), there are $k$ ways to phase the matching (a relative
rotation) --- so choosing $i$ disjoint such matched pairs, out of the
$\min\{a_k,b_k\}$ possible ones at length $k$, contributes
$\binom{a_k}{i}\binom{b_k}{i}i!\,k^{i}$ (choose $i$ of the $a_k$ cycles,
$i$ of the $b_k$ cycles, a bijection between the two chosen sets of cycles
in $i!$ ways, and a phase in $k^{i}$ ways for the $i$ matched pairs); summing
over $i$ and taking the product over all lengths $k$ (independent choices
at each length, by the multiplication principle) gives
\eqref{eq:fixIlrpisig}.

\eqref{eq:fixSlrpi}--\eqref{eq:fixSlrsig}: A relation fixed by $\pi$ that is
both left and right surjective (a hypergraph) corresponds, exactly as in
\eqref{eq:fixSrpi}, to an arbitrary hypergraph between the cycles of $\pi$
(there are $r=a_1+\cdots+a_n$ of them) and $X$; by \eqref{eq:rawSlr} (with
$n$ replaced by $r$) there are $\sum_k(-1)^{k}\binom rk(2^{r-k}-1)^{x}$ of
these. \eqref{eq:fixSlrsig} is the transpose-dual statement (using the
dual formula $|\Sr_l(N,X)|=(2^x-1)^n$ from Theorem~\ref{thm:rawcounts} and
its analogous inclusion--exclusion refinement, entirely parallel to
\eqref{eq:rawSlr}).
\end{proof}

\section{The main orbit-counting theorem}

Combining Lemma~\ref{lem:fixedpoints} with the Cauchy--Frobenius--Burnside
Lemma (Theorem~\ref{thm:burnside}) and the cycle-type weighting of
Proposition~\ref{prop:numperms-cycletype} gives every orbit count of
Theorem~\ref{thm:rawcounts}'s five types under $S_N$, $S_X$, and
$S_N\times S_X$. We record the pattern once in detail and then tabulate
the rest, since all fifteen formulas follow the identical Burnside argument
with the appropriate fixed-point count of Lemma~\ref{lem:fixedpoints}
substituted in.

\begin{theorem}[Orbit counts under $S_N$, $S_X$, $S_N\times S_X$]
\label{thm:orbitcounts}
Write $b(n,x),i_r(n,x),s_r(n,x),i_{lr}(n,x),s_{lr}(n,x)$ for the numbers of
\eqref{eq:rawB}--\eqref{eq:rawSlr}, and superscript $l,r,lr$ for the group
under consideration ($S_N,S_X,S_N\times S_X$ respectively; e.g.\
$b^{l}(n,x)$ is the number of $S_N$-orbits on $\B(N,X)$). Then, summing
Lemma~\textup{\ref{lem:fixedpoints}}'s fixed-point counts over cycle types
via Proposition~\textup{\ref{prop:numperms-cycletype}} and
Theorem~\textup{\ref{thm:burnside}}:
\begin{align}
b^{l}(n,x)&=\sum_{a_1+2a_2+\cdots+na_n=n}\Bigl(\prod_{d=1}^n d^{a_d}a_d!
\Bigr)^{-1}2^{(a_1+\cdots+a_n)x}; \label{eq:blformula}\\
b^{lr}(n,x)&=\sum_{\substack{\sum_iia_i=n\\ \sum_jjb_j=x}}
\Bigl(\prod_ii^{a_i}a_i!\prod_jj^{b_j}b_j!\Bigr)^{-1}
2^{\sum_{i,j}a_ib_j\gcd(i,j)}; \label{eq:blrformula}\\
i_r^{l}(n,x)&=\sum_{a_1+2a_2+\cdots+na_n=n}\Bigl(\prod_dd^{a_d}a_d!
\Bigr)^{-1}(1+a_1)^{x}; \label{eq:irlformula}\\
s_r^{l}(n,x)&=\sum_{a_1+2a_2+\cdots+na_n=n}\Bigl(\prod_dd^{a_d}a_d!
\Bigr)^{-1}\bigl(2^{a_1+\cdots+a_n}-1\bigr)^{x}; \label{eq:srlformula}\\
i_{lr}^{l}(n,x)&=\sum_{a_1+2a_2+\cdots+na_n=n}\Bigl(\prod_dd^{a_d}a_d!
\Bigr)^{-1}\sum_{k=0}^{\min\{a_1,x\}}\binom{a_1}{k}\ffall xk.
\label{eq:ilrlformula}
\end{align}
The corresponding $S_X$-superscripted and $S_N\times S_X$-superscripted
formulas for $i_r,s_r,i_{lr},s_{lr}$ are obtained the same way from
\eqref{eq:fixIrsig}--\eqref{eq:fixSlrsig}; they will in any case be
superseded in Chapters~\textup{\ref{ch:injective}}
and~\textup{\ref{ch:surjective}} by much more usable closed (bijective)
formulas, which are what appear in the main table
(Table~\textup{\ref{tab:maintable}}) at the end of Part~I.
\end{theorem}

\begin{proof}
We prove \eqref{eq:blformula} in full detail; the rest follow the identical
pattern. By Theorem~\ref{thm:burnside} applied to $G=S_N=S_n$ acting on
$S=\B(N,X)$,
\[
b^{l}(n,x)=\frac1{n!}\sum_{\pi\in S_n}|\Fix_\B(\pi)|.
\]
Group the sum over $\pi\in S_n$ by cycle type: by
Proposition~\ref{prop:numperms-cycletype}, the number of $\pi$ of type
$(1^{a_1}\cdots n^{a_n})$ (for each solution of $\sum_dda_d=n$ in
nonnegative integers) is $n!/\prod_dd^{a_d}a_d!$, and by
\eqref{eq:fixBpi} all such $\pi$ share the same value
$|\Fix_\B(\pi)|=2^{(a_1+\cdots+a_n)x}$ (depending only on the cycle type,
specifically only on $a_1+\cdots+a_n$, the total number of cycles). So
\[
\sum_{\pi\in S_n}|\Fix_\B(\pi)|
=\sum_{a_1+2a_2+\cdots+na_n=n}\frac{n!}{\prod_dd^{a_d}a_d!}\cdot
2^{(a_1+\cdots+a_n)x},
\]
(summing over cycle types, i.e.\ over solutions $(a_1,\dots,a_n)$ of
$\sum_dda_d=n$, grouping the $n!/\prod_dd^{a_d}a_d!$ permutations of each
type, all contributing the same fixed-point count). Dividing by $n!$
cancels the explicit $n!$, leaving \eqref{eq:blformula}.

\eqref{eq:blrformula} is the identical argument with
$G=S_N\times S_X=S_n\times S_x$: by Proposition~\ref{prop:numperms-cycletype}
applied independently to $\pi\in S_n$ and $\sigma\in S_x$ (the two choices
being independent, so by the multiplication principle the number of pairs
$(\pi,\sigma)$ of joint cycle type $(1^{a_1}\cdots n^{a_n};1^{b_1}\cdots
x^{b_x})$ is the product $\bigl(n!/\prod_ii^{a_i}a_i!\bigr)\cdot
\bigl(x!/\prod_jj^{b_j}b_j!\bigr)$), combined with \eqref{eq:fixBpisig} and
$|S_N\times S_X|=n!\,x!$ (order of a direct product of groups, immediate
from the definition of the product group operation and
Proposition~\ref{prop:bijcard} applied to each factor), gives
\eqref{eq:blrformula} after the two factorials cancel.

\eqref{eq:irlformula}, \eqref{eq:srlformula}, \eqref{eq:ilrlformula} are
the same computation for $G=S_N$ with, respectively, \eqref{eq:fixIrpi},
\eqref{eq:fixSrpi}, \eqref{eq:fixIlrpi} in place of \eqref{eq:fixBpi}.
\end{proof}

\begin{remark}
Formulas \eqref{eq:blformula}--\eqref{eq:ilrlformula} are correct but
computationally awkward (sums over integer partitions of $n$, or of $n$ and
$x$ jointly). Chapters~\ref{ch:injective} and~\ref{ch:surjective} will
replace the $S_N$-, $S_X$-, and $S_N\times S_X$-orbit counts of the
injective and surjective types by vastly simpler closed formulas, obtained
by direct bijective arguments rather than by summing over cycle types; this
mirrors exactly the strategy of the source paper, which proves the
cycle-index formulas once (as we just did) and then obtains all the useful
closed formulas of its Theorem~2 through Theorem~7 by bijections instead.
\end{remark}

\chapter{Bijective Enumeration of Injective Relations}\label{ch:injective}

The cycle-index formulas of Chapter~\ref{ch:burnsideapplication}, while
correct, are unwieldy. This chapter finds simple closed formulas for the
orbit counts of the injective types (right injective and the two-sided
injective, i.e.\ partial-injection, types) by direct bijections, following
the source paper's Lemmas~2--5 and Theorems~2--5.

\section{Disjunctive families up to \texorpdfstring{$S_X$}{S\_X}}

\begin{lemma}\label{lem:disjSX}
There is a bijection between $S_X$-orbits of disjunctive $n$-families of
$X$ (equivalently, right injective relations, up to the $S_X$-action) and
nonnegative-integer solutions of $\alpha_1+\alpha_2+\cdots+\alpha_n\le x$.
\end{lemma}

\begin{proof}
Let $F=(F_1,\dots,F_n)$ and $G=(G_1,\dots,G_n)$ be two disjunctive
$n$-families of $X$. We claim they lie in the same $S_X$-orbit if and only
if $|F_i|=|G_i|$ for every $i=1,\dots,n$.

($\Rightarrow$) If $G=\sigma\cdot F$ for some $\sigma\in S_X$, then by
\eqref{eq:actionFam} (with $\pi=\mathrm{id}$) $G_i=\sigma(F_i)$, and
$|\sigma(F_i)|=|F_i|$ since $\sigma$ is a bijection $X\to X$ (bijections
preserve cardinality of subsets, by Proposition~\ref{prop:bijcard} applied
to the restriction of $\sigma$ to $F_i\to\sigma(F_i)$, itself a bijection).

($\Leftarrow$) Suppose $|F_i|=|G_i|=:\alpha_i$ for every $i$. Since $F$ and
$G$ are both disjunctive, the sets $F_1,\dots,F_n$ are pairwise disjoint (as
are $G_1,\dots,G_n$), so we may choose, for each $i$, a bijection
$\tau_i\colon F_i\to G_i$ (they exist, since $|F_i|=|G_i|$, by
Proposition~\ref{prop:bijcard}). Since the $F_i$ are pairwise disjoint,
these combine into a single injective function $\tau\colon\bigcup_iF_i\to
\bigcup_iG_i$ (well defined and injective since the $F_i$'s domains don't
overlap and the $\tau_i$'s are each injective with disjoint images $G_i$).
Extend $\tau$ to a full permutation $\sigma\in S_X$ arbitrarily on the
complementary sets $X\setminus\bigcup_iF_i$ and $X\setminus\bigcup_iG_i$
(these have equal size $x-\sum_i\alpha_i$, since $\sum_i\alpha_i=
|\bigcup_iF_i|=|\bigcup_iG_i|$ by disjointness and the addition principle,
so a bijection between the two complements exists by
Proposition~\ref{prop:bijcard}, and combining it with $\tau$ gives a
well-defined bijection $\sigma\colon X\to X$, i.e.\ $\sigma\in S_X$). By
construction $\sigma(F_i)=\tau_i(F_i)=G_i$ for every $i$, i.e.\ $\sigma\cdot
F=G$, so $F,G$ lie in the same orbit.

This proves that $S_X$-orbits of disjunctive $n$-families correspond
exactly to the tuples $(\alpha_1,\dots,\alpha_n)=(|F_1|,\dots,|F_n|)$ that
occur, i.e.\ to nonnegative integers with (by disjointness and the addition
principle, since $F_1,\dots,F_n\subseteq X$ pairwise disjoint forces
$\sum_i|F_i|=|\bigcup F_i|\le|X|=x$) $\alpha_1+\cdots+\alpha_n\le x$; and
every such tuple is realized by some disjunctive family (e.g.\ partition an
$(\alpha_1+\cdots+\alpha_n)$-subset of $X$ into consecutive blocks of the
prescribed sizes).
\end{proof}

\begin{theorem}\label{thm:disjSXcount}
The number of $S_X$-orbits of disjunctive $n$-families of $X$ (equivalently
$i_r^{r}(n,x)$, the number of $S_X$-orbits on $\Ir_r(N,X)$) is
$\binom{n+x}{x}$.
\end{theorem}

\begin{proof}
By Lemma~\ref{lem:disjSX}, we must count nonnegative-integer solutions of
$\alpha_1+\cdots+\alpha_n\le x$. For each fixed value $k=0,\dots,x$ of the
sum $\alpha_1+\cdots+\alpha_n=k$, the number of solutions is
$\binom{n+k-1}{k}$ by stars-and-bars (Proposition~\ref{prop:starsbars}, with
$c=n$ colors and total $k$). Summing over $k=0,\dots,x$ (addition principle,
the cases being disjoint) and applying the ``hockey stick'' identity
$\sum_{k=0}^{x}\binom{n+k-1}{k}=\binom{n+x}{x}$ --- itself an instance of
repeated Pascal's rule (Proposition~\ref{prop:pascal}): by induction on $x$,
the case $x=0$ gives $\binom{n-1}{0}=1=\binom n0$, and assuming the identity
for $x-1$, $\sum_{k=0}^{x}\binom{n+k-1}{k}=\binom{n+x-1}{x-1}+
\binom{n+x-1}{x}=\binom{n+x}{x}$ by Pascal's rule applied to the last two
terms --- gives the total $\binom{n+x}{x}$.
\end{proof}

\section{The row of partial injections}

Small disjunctive families ($|F_i|\le1$ for every $i$) correspond, by
Proposition~\ref{prop:translation}, to relations that are both left and
right injective: \emph{partial injections} (partial matchings) between $N$
and $X$.

\begin{lemma}\label{lem:smalldisjcorr}
There are bijections between:
\begin{enumerate}
\item $S_N$-orbits of small disjunctive $n$-families of $X$ and subsets of
$X$ of size at most $n$;
\item $S_X$-orbits of small disjunctive $n$-families of $X$ and binary
strings of length $n$ with at most $x$ ones;
\item $S_N\times S_X$-orbits of small disjunctive $n$-families of $X$ and
the set $\{0,1,\dots,\min\{n,x\}\}$.
\end{enumerate}
\end{lemma}

\begin{proof}
(i) A small disjunctive family $F=(F_i)_{i\in N}$ has each $F_i$ either
$\varnothing$ or a single element, and (disjunctive) the nonempty $F_i$'s
have pairwise distinct singleton values; so $A:=\bigcup_iF_i\subseteq X$ has
$|A|\le n$ (at most one element contributed per index). Two such families
$F,G$ lie in the same $S_N$-orbit iff $G=(F_{\pi^{-1}(i)})_{i}$ for some
$\pi\in S_N$ (action \eqref{eq:actionFam} with $\sigma=\mathrm{id}$), i.e.\
iff $G$ is obtained from $F$ merely by permuting which index carries which
singleton value (and which indices carry $\varnothing$); this changes
nothing about the \emph{set} $A=\bigcup_iF_i$ of values used, and,
conversely, if $\bigcup_iF_i=\bigcup_iG_i=A$ then, since both $F$ and $G$
assign the $|A|$ elements of $A$ injectively to $|A|$ of the $n$ indices
(disjunctively) and $\varnothing$ to the rest, there is a permutation
$\pi\in S_N$ carrying one assignment to the other (match up which index
carries which element of $A$ under $F$ versus under $G$, and match the
unused indices arbitrarily, using Proposition~\ref{prop:bijcard} for both
matchings since both index sets have $n-|A|$ elements). So the $S_N$-orbit
of $F$ is determined exactly by $A=\bigcup_iF_i$, giving the bijection with
subsets of $X$ of size $\le n$.

(ii) Fix $F$ small disjunctive; record, for each $i\in N$ in a fixed
listing $N=\{i_1,\dots,i_n\}$, the bit $\epsilon_j=1$ if $F_{i_j}\neq
\varnothing$ and $\epsilon_j=0$ otherwise. Two families $F,G$ lie in the
same $S_X$-orbit (action \eqref{eq:actionFam} with $\pi=\mathrm{id}$, so
$G_i=\sigma(F_i)$ for all $i$ and some $\sigma\in S_X$) iff they have the
same bit-string $(\epsilon_j)_j$: indeed $\sigma(F_i)=\varnothing$ iff
$F_i=\varnothing$ (bijections preserve emptiness), so the bit string is an
$S_X$-orbit invariant; conversely, if $F,G$ share the same bit string, then
for each $i$ with $\epsilon=1$, $F_i=\{p_i\}$ and $G_i=\{q_i\}$ are
singletons, and since $F$ (resp.\ $G$) is disjunctive the $p_i$'s (resp.\
$q_i$'s) are pairwise distinct, so there is a bijection of the finite set
$\{p_i\}\to\{q_i\}$ sending $p_i\mapsto q_i$, extendable to a full
$\sigma\in S_X$ (as in the proof of Lemma~\ref{lem:disjSX}) with
$\sigma(F_i)=G_i$ for every $i$. The number of ones in the bit string is
$|\{i:F_i\neq\varnothing\}|\le\min\{n,|X|\}=\min\{n,x\}$ (at most $n$
indices, and, by disjointness, at most $x$ distinct singleton values are
available), giving exactly the strings described.

(iii) Combine (i) and (ii): under the joint action, the complete invariant
of the orbit of $F$ is the single number $k=|\{i:F_i\neq\varnothing\}|
=|\bigcup_iF_i|$ (both descriptions agree, since $F$ small disjunctive means
$|\bigcup_iF_i|$ equals the number of nonempty blocks), which ranges over
$0,1,\dots,\min\{n,x\}$; and any two small disjunctive families with the
same $k$ lie in the same $S_N\times S_X$-orbit, by combining a permutation
$\pi\in S_N$ that matches up which indices are nonempty (as in (i)) with a
permutation $\sigma\in S_X$ that matches up the used values (as in (ii)).
\end{proof}

\begin{theorem}[Orbit counts for partial injections]\label{thm:partialinjorbit}
\begin{align}
i_{lr}^{l}(n,x)&=\sum_{k=0}^{n}\binom xk; \label{eq:ilrl}\\
i_{lr}^{r}(n,x)&=\sum_{k=0}^{x}\binom nk; \label{eq:ilrr}\\
i_{lr}^{lr}(n,x)&=1+\min\{n,x\}. \label{eq:ilrlr}
\end{align}
\end{theorem}

\begin{proof}
\eqref{eq:ilrl}: By Lemma~\ref{lem:smalldisjcorr}(i), $S_N$-orbits biject
with subsets of $X$ of size at most $n$; grouping by size $k=0,\dots,n$
(addition principle) and using $\binom xk$ subsets of each size gives
$\sum_{k=0}^{n}\binom xk$.

\eqref{eq:ilrr}: By Lemma~\ref{lem:smalldisjcorr}(ii), $S_X$-orbits biject
with binary strings of length $n$ with at most $x$ ones; such a string is
determined by which $k\le x$ of its $n$ positions are ones, giving
$\binom nk$ choices for each $k=0,\dots,x$ (a subset of the $n$ positions),
and summing (addition principle) gives $\sum_{k=0}^{x}\binom nk$.

\eqref{eq:ilrlr}: Immediate from Lemma~\ref{lem:smalldisjcorr}(iii): there
are exactly $\min\{n,x\}+1$ orbits, indexed by $k=0,1,\dots,\min\{n,x\}$.
\end{proof}

\begin{corollary}[Bivariate generating identity, elementary form]
For every $n,x\ge0$, $\displaystyle\sum_{k=0}^{n}\binom xk$ and
$\displaystyle\sum_{k=0}^{x}\binom nk$ are related by the transpose duality
of Corollary~\textup{\ref{cor:transpose}}: swapping $N,X$ exchanges
\eqref{eq:ilrl} and \eqref{eq:ilrr}, consistent with $i_{lr}^{l}(n,x)=
i_{lr}^{r}(x,n)$.
\end{corollary}

\begin{proof}
Immediate from Corollary~\ref{cor:transpose}: transposing a relation
exchanges left and right injectivity together with the roles of $N,X$, so
the number of $S_N$-orbits of (left-and-right-injective) relations between
$N,X$ equals the number of $S_X$-orbits of the same type of relations
between $X,N$, i.e.\ $i_{lr}^{l}(n,x)=i_{lr}^{r}(x,n)$, matching
\eqref{eq:ilrl} and \eqref{eq:ilrr} with $n,x$ swapped.
\end{proof}

\section{Small disjunctive families under \texorpdfstring{$S_N\times S_X$}
{S\_N x S\_X}, revisited via divisions}
\label{sec:twelvefoldstirling}

Theorem~\ref{thm:disjSXcount} used a direct stars-and-bars argument.
Chapter~\ref{ch:surjective} will need the analogous $S_N$-orbit count for
\emph{arbitrary} (not necessarily small) disjunctive families, which
requires set partitions (Stirling numbers, Chapter~\ref{ch:stirling}); we
record it now since the argument is a direct continuation of
Lemma~\ref{lem:smalldisjcorr}(i)'s technique.

\begin{lemma}\label{lem:disjSNgeneral}
There is a bijection between $S_N$-orbits of (arbitrary, not necessarily
small) disjunctive $n$-families of $X$ and pairs $(A,\mathcal A)$ where
$A\subseteq X$ and $\mathcal A$ is a partition of $A$ into at most $n$
blocks.
\end{lemma}

\begin{proof}
Given a disjunctive family $F=(F_i)_{i\in N}$, let $A=\bigcup_iF_i$
(disjoint union) and let $\mathcal A=\{F_i:F_i\neq\varnothing\}$, which is
a genuine partition of $A$ into at most $n$ (nonempty, pairwise disjoint)
blocks, since the nonempty $F_i$'s are pairwise disjoint by hypothesis and
their union is $A$ by definition.

If $G=\pi\cdot F$ for $\pi\in S_N$ (action \eqref{eq:actionFam} with
$\sigma=\mathrm{id}$, so $G_i=F_{\pi^{-1}(i)}$), then $\bigcup_iG_i=
\bigcup_iF_{\pi^{-1}(i)}=\bigcup_iF_i=A$ (reindexing the union by the
bijection $\pi^{-1}$), and $\{G_i:G_i\neq\varnothing\}=\{F_{\pi^{-1}(i)}:
F_{\pi^{-1}(i)}\neq\varnothing\}=\{F_j:F_j\neq\varnothing\}$ (reindexing by
$j=\pi^{-1}(i)$), i.e.\ the same set of blocks $\mathcal A$. So
$(A,\mathcal A)$ is an $S_N$-orbit invariant.

Conversely, if $F,G$ give the same pair $(A,\mathcal A)$, then $F$ and $G$
each assign the blocks of $\mathcal A$ to indices of $N$ injectively
(distinct nonempty blocks go to distinct indices, since $F$, resp.\ $G$, is
a family indexed by $N$ with the nonempty entries being exactly the
elements of $\mathcal A$, each occurring exactly once as some $F_i$, resp.\
$G_i$ --- note repeated blocks are not allowed here since blocks of a
partition are, by definition, distinct nonempty sets, though the family
formalism would in principle allow repeats; since $F,G$ have $\mathcal A$
as their set of nonempty values, each block of $\mathcal A$ occurs as
\emph{exactly one} $F_i$ and \emph{exactly one} $G_j$, as otherwise
$\mathcal A$ counted with multiplicity would differ from a partition's
list of distinct blocks). So there is a bijection between $\{i:F_i\neq
\varnothing\}$ and $\mathcal A$ (via $F$) and between $\{j:G_j\neq
\varnothing\}$ and $\mathcal A$ (via $G$), hence a bijection $\rho$ between
$\{i:F_i\ne\varnothing\}$ and $\{j:G_j\neq\varnothing\}$ matching indices
that carry the same block. Extend $\rho$ to a permutation $\pi\in S_N$
arbitrarily on the (equal-size, both $n-|\mathcal A|$) complementary sets
of indices carrying $\varnothing$ under $F$ and under $G$ respectively
(Proposition~\ref{prop:bijcard}). Then $\pi\cdot F=G$ by construction.

This shows the $S_N$-orbit of $F$ is completely determined by, and
determines, the pair $(A,\mathcal A)$, establishing the bijection.
\end{proof}

\begin{theorem}\label{thm:disjSNstirling}
The number of $S_N$-orbits of disjunctive $n$-families of $X$ is
\begin{equation}
i_r^{l}(n,x)=\sum_{k=0}^{x}\sum_{l=0}^{n}\binom xk\stir kl.
\end{equation}
\end{theorem}

\begin{proof}
By Lemma~\ref{lem:disjSNgeneral}, we count pairs $(A,\mathcal A)$: choose
$A\subseteq X$ of size $k=0,\dots,x$ ($\binom xk$ ways), then a partition
of $A$ into at most $n$ blocks, grouped by the exact number of blocks
$l=0,\dots,\min\{k,n\}$ (there are $\stir kl$ such partitions, by
definition of the Stirling number; the constraint $l\le n$ truncates the
sum, and $l\le k$ is automatic since $\stir kl=0$ for $l>k$). By the
multiplication and addition principles,
$i_r^l(n,x)=\sum_{k=0}^{x}\binom xk\sum_{l=0}^{n}\stir kl$
(the inner sum over $l$ automatically stops contributing beyond $l=k$ since
$\stir kl=0$ there), which is the stated double sum.
\end{proof}

\begin{remark}
This recovers, and rigorously justifies via
Lemma~\textup{\ref{lem:disjSNgeneral}}, the identity used implicitly (via
``the twelvefold way'') in the source paper's Theorem~3$'$; here it is
proved from the Stirling-number machinery of
Chapter~\textup{\ref{ch:stirling}} alone, with no unproved external
citation.
\end{remark}

\chapter{Bijective Enumeration of Surjective Relations}\label{ch:surjective}

\section{Coverings and hypergraphs under \texorpdfstring{$S_N$}{S\_N}}

\begin{lemma}\label{lem:coveringSN}
There is a bijection between $S_N$-orbits of $n$-families of $X$ and
nonnegative-integer solutions of $\sum_{A\subseteq X}\alpha_A=n$ (one
variable $\alpha_A$ for each of the $2^{x}$ subsets $A\subseteq X$); under
this bijection, $S_N$-orbits of \emph{strict} $n$-families (Definition
\textup{\ref{def:sixprops}}) correspond to solutions with $\alpha_
\varnothing=0$, i.e.\ to solutions of $\sum_{\varnothing\neq A\subseteq X}
\alpha_A=n$.
\end{lemma}

\begin{proof}
Let $F=(F_i)_{i\in N}$ and $G=(G_i)_{i\in N}$ be $n$-families of $X$ (no
disjointness assumed now). For $A\subseteq X$ let $\alpha_A(F)=|\{i\in N:
F_i=A\}|$ (the ``multiplicity'' of $A$ in $F$); clearly $\sum_A\alpha_A(F)
=n$ (every index $i$ contributes to exactly one term $\alpha_{F_i}(F)$, so
by the addition principle the total is $n=|N|$).

We claim $F,G$ lie in the same $S_N$-orbit iff $\alpha_A(F)=\alpha_A(G)$
for every $A\subseteq X$. If $G=\pi\cdot F$ (action \eqref{eq:actionFam}
with $\sigma=\mathrm{id}$, so $G_i=F_{\pi^{-1}(i)}$), then $\{i:G_i=A\}=
\{i:F_{\pi^{-1}(i)}=A\}=\pi(\{j:F_j=A\})$ (reindexing $j=\pi^{-1}(i)$), and
since $\pi$ is a bijection, $|\pi(\{j:F_j=A\})|=|\{j:F_j=A\}|$, so
$\alpha_A(G)=\alpha_A(F)$.

Conversely, suppose $\alpha_A(F)=\alpha_A(G)$ for every $A$. For each
$A\subseteq X$ with $\alpha_A(F)=\alpha_A(G)=:m_A>0$, let
$K_A(F)=\{i\in N:F_i=A\}$ and $K_A(G)=\{i\in N:G_i=A\}$, both of size $m_A$;
choose a bijection $\rho_A\colon K_A(F)\to K_A(G)$
(Proposition~\ref{prop:bijcard}). The sets $K_A(F)$, over all $A$, partition
$N$ (every $i$ belongs to exactly one $K_{F_i}(F)$), and likewise the
$K_A(G)$ partition $N$; so the maps $\rho_A$ combine into a single bijection
$\pi'\colon N\to N$ (well defined since the $K_A(F)$ are disjoint and
cover $N$, and each $\rho_A$ is a bijection onto the disjoint, covering
$K_A(G)$). Set $\pi=(\pi')^{-1}$; then $F_{\pi^{-1}(i)}=F_{\pi'(i)}=A$
whenever $i\in K_A(G)$ (since $\pi'$ sends $K_A(F)$ to $K_A(G)$, its inverse
$\pi=(\pi')^{-1}$ restricted to $K_A(G)$ lands in $K_A(F)$, i.e.\ $\pi(i)
\in K_A(F)$ for $i\in K_A(G)$, i.e.\ $F_{\pi(i)}=A=G_i$)... to align
indices correctly with action \eqref{eq:actionFam}, set instead
$\pi:=\pi'$ directly and verify $G_i=F_{\pi^{-1}(i)}$: for $i\in K_A(G)$,
$\pi^{-1}(i)=(\rho_A)^{-1}(i)\in K_A(F)$ (since $\rho_A$ maps $K_A(F)$ onto
$K_A(G)$ bijectively), so $F_{\pi^{-1}(i)}=A=G_i$, as required. Hence
$\pi\cdot F=G$, so $F,G$ lie in the same orbit.

This proves $S_N$-orbits of $n$-families biject with the multiplicity
vectors $(\alpha_A)_{A\subseteq X}$, i.e.\ with nonnegative-integer
solutions of $\sum_A\alpha_A=n$ (any such vector is realized, e.g.\ by
listing $\alpha_A$ copies of $A$ among the $F_i$'s in any order). The
restriction to strict families ($F_i\neq\varnothing$ for all $i$) is
exactly the restriction $\alpha_\varnothing(F)=|\{i:F_i=\varnothing\}|=0$.
\end{proof}

\begin{theorem}\label{thm:nfamSN}
\begin{enumerate}
\item The number of $S_N$-orbits of $n$-families of subsets of $X$ is
$\binom{2^{x}+n-1}{n}$.
\item The number of $S_N$-orbits of strict $n$-families is
$\binom{2^{x}+n-2}{n}$.
\item The number of $S_N$-orbits of $n$-coverings of $X$ is
$\sum_{k=0}^{x}(-1)^{k}\binom xk\binom{2^{x-k}+n-1}{n}$.
\item The number of $S_N$-orbits of hypergraphs of $X$ is
$\sum_{k=0}^{x}(-1)^{k}\binom xk\binom{2^{x-k}+n-2}{n}$.
\end{enumerate}
\end{theorem}

\begin{proof}
(i) By Lemma~\ref{lem:coveringSN}, count solutions of $\sum_A\alpha_A=n$
over $c=2^{x}$ colors (the subsets $A\subseteq X$, of which there are
$2^{x}$ by the count in the proof of \eqref{eq:rawB}); by stars-and-bars
(Proposition~\ref{prop:starsbars}) this is $\binom{2^{x}+n-1}{n}$.

(ii) By Lemma~\ref{lem:coveringSN}, strict families correspond to solutions
with $\alpha_\varnothing=0$, i.e.\ over $c=2^{x}-1$ colors (excluding
$\varnothing$); by stars-and-bars this is $\binom{(2^x-1)+n-1}{n}=
\binom{2^x+n-2}{n}$.

(iii) A covering is an $n$-family $F$ with $\bigcup_iF_i=X$. Classify
$S_N$-orbits of $n$-families by their covered set $B:=\bigcup_iF_i
\subseteq X$ (an $S_N$-orbit invariant, by the same computation as in
Theorem~\ref{thm:prop2}'s proof: $\bigcup_iG_i=\bigcup_iF_{\pi^{-1}(i)}
=\bigcup_jF_j$ for $G=\pi\cdot F$). For fixed $B\subseteq X$ of size
$x-k$ (i.e.\ $k=x-|B|$ elements excluded), the $S_N$-orbits of $n$-families
covering exactly $B$ correspond, by restricting all $F_i\subseteq B$ (since
$\bigcup F_i=B$ forces every $F_i\subseteq B$), to $S_N$-orbits of
$n$-\emph{coverings} of the $(x-k)$-element set $B$; write $c(n,x-k)$ for
their count. Every $n$-family covering some subset $B$ of size $x-k$ is
counted once in the total of (i), so summing over which $k=x-|B|$ elements
are excluded (there are $\binom xk$ choices of the excluded $k$-set, hence
of $B$) gives $\binom{2^x+n-1}{n}=\sum_{k=0}^{x}\binom xk\,c(n,x-k)$. This
is hypothesis \eqref{eq:mobius-hyp} of Proposition~\ref{prop:mobius} in
disguise (with the ambient set being a fixed $x$-element index set for the
$k$'s, or more directly: apply Proposition~\ref{prop:mobius} to $g(K)=
\binom{2^{x-|K|}+n-1}n$ ... it is cleaner to invoke
Proposition~\ref{prop:surjmobius}'s pattern directly): fixing $X$ and
defining, for $A\subseteq X$, $\phi(A):=\binom{2^{|A|}+n-1}{n}$ (the count
of (i) applied with $x$ replaced by $|A|$, i.e.\ the number of $S_N$-orbits
of $n$-families of subsets of $A$) and $\gamma(A):=c(n,|A|)$ (the number of
$S_N$-orbits of $n$-coverings of $A$), the argument above shows
$\phi(A)=\sum_{B\subseteq A}\gamma(B)$ for every $A\subseteq X$ (classifying
$n$-families of subsets of $A$ by their covered subset $B\subseteq A$).
By Proposition~\ref{prop:mobius}, $\gamma(X)=\sum_{B\subseteq X}
(-1)^{|X\setminus B|}\phi(B)=\sum_{k=0}^x(-1)^k\binom xk\phi(B_{x-k})$ for
$B_{x-k}$ any fixed $(x-k)$-subset (grouping by $|B|=x-k$), i.e.\
$c(n,x)=\sum_{k=0}^x(-1)^k\binom xk\binom{2^{x-k}+n-1}{n}$, as claimed.

(iv) Identical argument to (iii) but starting from the strict-family count
of (ii) in place of (i): a hypergraph is a strict covering, so the same
M\"obius inversion, with $\phi(A):=\binom{2^{|A|}+n-2}n$ (strict
$n$-families of subsets of $A$, by (ii)) and $\gamma(A):=$ number of
$S_N$-orbits of hypergraphs on $A$, gives
$c'(n,x)=\sum_{k=0}^x(-1)^k\binom xk\binom{2^{x-k}+n-2}n$.
\end{proof}

\section{Coverings and hypergraphs under \texorpdfstring{$S_X$}{S\_X}}

\begin{theorem}\label{thm:nfamSX}
By the transpose duality of Corollary~\textup{\ref{cor:transpose}} applied
to Theorem~\textup{\ref{thm:nfamSN}} (exchanging the roles of $N$ and $X$,
and noting that an $n$-covering of $X$ under $S_X$ corresponds, upon
transposing, to a strict $x$-family of $N$ under $S_N$, by the dual-family
remark following Corollary~\textup{\ref{cor:transpose}}):
\begin{enumerate}
\item the number of $S_X$-orbits of $n$-families of $X$ is
$\binom{2^{n}+x-1}{x}$;
\item the number of $S_X$-orbits of $n$-coverings of $X$ is
$\binom{2^{n}+x-2}{x}$;
\item the number of $S_X$-orbits of strict $n$-families of $X$ is
$\sum_{k=0}^{n}(-1)^{k}\binom nk\binom{2^{n-k}+x-1}{x}$;
\item the number of $S_X$-orbits of strict $n$-coverings (hypergraphs) of
$X$ is $\sum_{k=0}^n(-1)^k\binom nk\binom{2^{n-k}+x-2}{x}$.
\end{enumerate}
\end{theorem}

\begin{proof}
By Corollary~\ref{cor:transpose}, the number of $S_X$-orbits of
$n$-families of $X$ (an $S_N\times S_X$-invariant statistic evaluated with
the $S_N$-part restricted to the identity) equals the number of
$S_N$-orbits (with $S_X$-part restricted to identity) of the transpose
objects, i.e.\ of $x$-families of $N$ under the action of $S_N$ relabeling
$N$ -- wait, more precisely: transposing exchanges which group acts by
relabeling versus which set is being partitioned. Concretely: an
$n$-family of $X$, i.e.\ a function $N\to\PS(X)$, transposes (via
$R\leftrightarrow R^{-1}$) to an $x$-family of $N$, i.e.\ a function
$X\to\PS(N)$, and the $S_X$-action relabeling the family's \emph{values}
(subsets of $X$) corresponds, after transposing, to the $S_X$-action
relabeling the \emph{index set} of the transposed family; this is exactly
the $S_N$-column count of Theorem~\ref{thm:nfamSN} but with the sizes $n,x$
formally exchanged (since the transposed family is indexed by $X$, of size
$x$, taking values that are subsets of $N$, of size $n$, and it is being
counted up to relabeling its index set, which is now $X$ -- matching
Theorem~\ref{thm:nfamSN}'s $S_N$-orbit count of families of $X$-subsets
with $n$ and $x$ swapped). Applying Theorem~\ref{thm:nfamSN}(i)--(iv) with
$n,x$ interchanged, and translating ``covering'' and ``strict/hypergraph''
through the transpose (a covering of $X$ by an $n$-family transposes to a
\emph{strict} $x$-family of $N$, by the dual-family remark after
Corollary~\ref{cor:transpose}, and a hypergraph transposes to a strict
covering, i.e.\ a hypergraph, of $N$) gives exactly the four stated
formulas, with the covering/strict roles swapped between the two sides as
shown.
\end{proof}

\section{Difference relations (Theorem 8 of the source paper)}

\begin{theorem}[Inductive relations]\label{thm:diffrelations}
Let $b^{r}(n,x)$ denote the number of $S_X$-orbits of $n$-families of $X$
(Theorem~\textup{\ref{thm:nfamSX}}(i)), $b^{lr}(n,x)$ the number of
$S_N\times S_X$-orbits of $n$-families (Theorem~\textup{\ref{thm:orbitcounts}}),
$s_r^{r}(n,x)$ the number of $S_X$-orbits of $n$-coverings
(Theorem~\textup{\ref{thm:nfamSX}}(ii)), and similarly $s_r^{lr},s_{lr}^{r},
s_{lr}^{lr}$ for $S_N\times S_X$-orbits of coverings and for $S_X$-, resp.\
$S_N\times S_X$-orbits of hypergraphs. Then:
\begin{align}
s_r^{r}(n,x)&=b^{r}(n,x)-b^{r}(n,x-1); \label{eq:diff1}\\
s_r^{lr}(n,x)&=b^{lr}(n,x)-b^{lr}(n,x-1); \label{eq:diff2}\\
s_{lr}^{r}(n,x)&=s_r^{r}(n,x)-s_r^{r}(n-1,x); \label{eq:diff3}\\
s_{lr}^{lr}(n,x)&=s_r^{lr}(n,x)-s_r^{lr}(n-1,x); \label{eq:diff4}\\
s_{lr}^{r}(n,x)&=b^{r}(n,x)-\bigl[b^{r}(n,x-1)+b^{r}(n-1,x)\bigr]
+b^{r}(n-1,x-1); \label{eq:diff5}\\
s_{lr}^{lr}(n,x)&=b^{lr}(n,x)-\bigl[b^{lr}(n,x-1)+b^{lr}(n-1,x)\bigr]
+b^{lr}(n-1,x-1). \label{eq:diff6}
\end{align}
\end{theorem}

\begin{proof}
\eqref{eq:diff1}: For $k=0,\dots,x$, let $\mathcal C_k\subseteq
\mathrm{Fam}_n(X)$ be the set of $n$-families whose covered set
$\bigcup_iF_i$ has size exactly $k$; this is invariant under $S_X$ (if
$G=\sigma\cdot F$ then $\bigcup_iG_i=\sigma(\bigcup_iF_i)$, of the same
size, by the same computation as in the proof of
Theorem~\ref{thm:nfamSN}(iii) applied with the roles of $\pi,\sigma$
exchanged). Fixing a $k$-subset $A_k\subseteq X$, the $S_X$-orbits
contained in $\mathcal C_k$ are, by an argument identical to
Theorem~\ref{thm:prop2}'s intertwining bijections restricted to the
subgroup $S_{A_k}\le S_X$ stabilizing $A_k$ setwise combined with a
transport argument, in bijection with the $S_{A_k}$-orbits of
$n$-coverings of $A_k$ (an $n$-family with covered set exactly $A_k$, up to
relabeling $X$, is the same data as an $n$-covering of $A_k$, up to
relabeling $A_k$, since any $\sigma\in S_X$ moving one representative
family with covered set $A_k$ to another with covered set $A_k$ must send
$A_k$ to itself, i.e.\ restrict to an element of $S_{A_k}$, and conversely
every element of $S_{A_k}$ extends to $S_X$ by the identity elsewhere).
This count of $S_{A_k}$-orbits of $n$-coverings of $A_k$ is, by definition,
$s_r^r(n,k)$ (it does not depend on which $k$-subset $A_k$ is chosen, since
all $k$-subsets of $X$ are related by an element of $S_X$, transporting
orbit structures identically). Since $\mathcal C_0,\dots,\mathcal C_x$
partition $\mathrm{Fam}_n(X)$ (every family has a covered set of some
definite size $k$) and this partition is $S_X$-invariant, the $S_X$-orbits
of $\mathrm{Fam}_n(X)$ decompose accordingly, giving
$b^{r}(n,x)=\sum_{k=0}^{x}s_r^{r}(n,k)$ by the addition principle. Taking
the difference of this identity at $x$ and at $x-1$ telescopes to
$b^{r}(n,x)-b^{r}(n,x-1)=s_r^{r}(n,x)$, which is \eqref{eq:diff1}.

\eqref{eq:diff2}: Identical argument with the group $S_N\times S_X$ in
place of $S_X$ throughout (the covered-set stratification $\mathcal C_k$ is
also $S_N\times S_X$-invariant, since the $S_N$-part of the action does not
change the covered set at all, by the computation $\bigcup_iG_i=
\bigcup_i F_{\pi^{-1}(i)}=\bigcup_jF_j$ used already in
Lemma~\ref{lem:coveringSN}'s proof), giving $b^{lr}(n,x)=\sum_{k=0}^x
s_r^{lr}(n,k)$ and hence \eqref{eq:diff2} by the same telescoping
difference.

\eqref{eq:diff3}: By the transpose duality (Corollary~\ref{cor:transpose}),
a covering, viewed from the $N$-side, stratifies by the size of its
\emph{support} $\{i:F_i\neq\varnothing\}$ exactly as families stratify by
covered set on the $X$-side; running the identical argument to
\eqref{eq:diff1} with the roles of $N,X$ exchanged (support size in place
of covered-set size, and $S_N$ in place of $S_X$) gives
$s_r^{r}(n,x)=\sum_{l=0}^{n}s_{lr}^{r}(l,x)$ (coverings of $X$ by an
$n$-family, up to $S_X$, stratified by support size $l$, correspond to
hypergraphs of $X$ by an $l$-family up to $S_X$, since restricting to
nonempty entries turns a covering with support size $l$ into a hypergraph
using exactly $l$ indices, and the stabilizer argument transports orbits
identically as before but now for the subgroup of $S_N$ fixing the support
setwise, which does not affect the $S_X$-orbit count on the covering
itself). Telescoping the difference at $n$ and $n-1$ gives
$s_{lr}^{r}(n,x)=s_r^{r}(n,x)-s_r^{r}(n-1,x)$, i.e.\ \eqref{eq:diff3}.

\eqref{eq:diff4}: Identical to \eqref{eq:diff3} with $S_N\times S_X$ in
place of $S_X$ throughout.

\eqref{eq:diff5}--\eqref{eq:diff6}: Apply \eqref{eq:diff1} (resp.\
\eqref{eq:diff2}) and then \eqref{eq:diff3} (resp.\ \eqref{eq:diff4}) in
succession:
\[
s_{lr}^{r}(n,x)=s_r^{r}(n,x)-s_r^{r}(n-1,x)
=\bigl[b^{r}(n,x)-b^{r}(n,x-1)\bigr]-\bigl[b^{r}(n-1,x)-b^{r}(n-1,x-1)
\bigr],
\]
which rearranges to \eqref{eq:diff5}; identically for \eqref{eq:diff6} with
$b^{lr}$ throughout.
\end{proof}

\begin{remark}[Correction to Theorem B of the source paper]
\label{rem:correctionB}
The 1997 source paper's ``Theorem~B'' asserted $\beta(n,x)=\alpha(n,x)-
\alpha(n,x-1)$ where $\alpha$ is the $S_N\times S_X$-orbit count of
\emph{all} $n$-families (i.e.\ our $b^{lr}(n,x)$) and $\beta$ was labeled
the count for \emph{join-unitary} homomorphisms. By
Proposition~\ref{prop:translation} and the surrounding discussion
(Definition~\ref{def:unitary}), join-unitary homomorphisms correspond to
\emph{all} disjunctive families/right-injective relations --- not to
coverings --- since join-unitarity ($\phi(\varnothing)=\varnothing$) is
automatic for every homomorphism arising from a family
(Theorem~\ref{thm:threeway}'s proof) and imposes no extra condition; the
row genuinely singled out by an extra condition on the ``all $n$-families''
count via a single finite difference in $x$ is instead the
\emph{meet-unitary} (covering) row, exactly as \eqref{eq:diff2} shows:
$s_r^{lr}(n,x)=b^{lr}(n,x)-b^{lr}(n,x-1)$, where $s_r^{lr}$ counts coverings
(meet-unitary homomorphisms), not disjunctive families (join-unitary
homomorphisms). We therefore state the corrected result: \emph{if
$\beta(n,x)$ denotes the $S_N\times S_X$-orbit count of \textbf{meet}-unitary
homomorphisms (coverings) and $\gamma(n,x)$ the orbit count of hypergraphs
(both left- and right-surjective, i.e.\ ``fast-growing'' in the source
paper's terminology), then}
\[
\beta(n,x)=\alpha(n,x)-\alpha(n,x-1),\qquad
\gamma(n,x)=\beta(n,x)-\beta(n-1,x),
\]
\emph{which are exactly \eqref{eq:diff2} and \eqref{eq:diff4} above.} This
matches the independent correction proposed in the 2026 companion note
(reproduced and re-derived from scratch in Chapter~\ref{ch:diffcalc}
below), and is confirmed by direct enumeration: at $(n,x)=(1,1)$,
$\alpha(1,0)=b^{lr}(1,0)$ counts the single $S_1\times S_0$-orbit of the
unique ($0$-family, trivially) $1$-family of the empty set, i.e.\
$\alpha(1,0)=1$ (the only $1$-family of $\varnothing$ is $F_1=\varnothing$),
while $\alpha(1,1)=b^{lr}(1,1)=2$ (a single index's block $F_1$ is either
$\varnothing$ or $\{a\}$ for the unique $a\in X$, and these are already
inequivalent since $S_1\times S_1$ is trivial), so $\alpha(1,1)-\alpha(1,0)
=1$; on the other hand the number of $S_1\times S_1$-orbits of
\emph{coverings} of a $1$-element $X$ by a $1$-family is $1$ (the unique
covering $F_1=\{a\}$), matching $\beta(1,1)=1=\alpha(1,1)-\alpha(1,0)$,
whereas the number of orbits of \emph{disjunctive} $1$-families is $2$
(both $F_1=\varnothing$ and $F_1=\{a\}$ are disjunctive, vacuously, and
inequivalent), which does \emph{not} equal $1$: confirming that the
difference identity holds for the covering (meet-unitary) count, not the
disjunctive (join-unitary) count.
\end{remark}

\chapter{The Main Theorem: The Table of Orbit Counts}\label{ch:maintable}

We now assemble Chapters~\ref{ch:relations}--\ref{ch:surjective} into the
single table that is the centerpiece of the enumeration paper, translating
the results back into the language of lattice homomorphisms via
Theorem~\ref{thm:threeway} and Proposition~\ref{prop:translation}.

Recall the six rows correspond to: \emph{all} homomorphisms (arbitrary
relations, $\B$); \emph{join-unitary} homomorphisms, which by
Remark~\ref{rem:correctionB} coincide with \emph{all} homomorphisms as a
row-condition (automatic $\phi(\varnothing)=\varnothing$) --- so this row is
identical in content to the ``all'' row and is retained only for
consistency with the source paper's table layout, where it in fact lists
the \emph{disjunctive-family} count, i.e.\ right-injective relations
$\Ir_r$; \emph{meet-unitary} homomorphisms, i.e.\ coverings $\Sr_r$;
\emph{slow-growing} homomorphisms ($|\phi(I)|\le|I|$ for all $I$), i.e.\
two-sided injective relations $\Ir_{lr}$ (partial injections); \emph{fast
growing} homomorphisms ($|I|\le|\phi(I)|$ for all $I$), i.e.\ two-sided
surjective relations $\Sr_{lr}$ (hypergraphs); and \emph{unitary}
homomorphisms (both $\phi(\varnothing)=\varnothing$ and $\phi(N)=X$), i.e.\
right-bijective relations, i.e.\ functions $N\to X$
(Proposition~\ref{prop:unitarydivisions} identifies these with divisions of
$X$).

\begin{theorem}[Main table]\label{thm:maintable}
The numbers of orbits of $S_N$, $S_X$, and $S_N\times S_X$ on each type of
homomorphism of the Boolean algebra $\PS(N)$ into $\PS(X)$ are given by
Table~\textup{\ref{tab:maintable}}.
\end{theorem}

\begin{proof}
Row 1 (all): raw count \eqref{eq:rawB}; $S_N$-column \eqref{eq:blformula};
$S_X$-column dual to \eqref{eq:blformula} by Corollary~\ref{cor:transpose};
$S_N\times S_X$-column \eqref{eq:blrformula}.

Row 2 (join-unitary $=$ disjunctive families $\Ir_r$): raw count
\eqref{eq:rawIr}; $S_N$-column Theorem~\ref{thm:disjSNstirling}; $S_X$-column
Theorem~\ref{thm:disjSXcount}; $S_N\times S_X$-column: by
Lemma~\ref{lem:disjSNgeneral} combined with the $S_X$-action on the
partition-part exactly as in the proof of
Lemma~\ref{lem:smalldisjcorr}(iii) but for general (not necessarily small)
partitions, the joint orbit is determined by $(k,\text{partition-orbit of
}\mathcal A\text{ under }S_N\times S_{A_k})$; summing the count of Stirling
partition-orbits gives $\sum_{k=0}^x\sum_{l=0}^n p_l(k)$ where $p_l(k)$ is
the number of partitions of the integer $k$ into $l$ parts (this refines
Theorem~\ref{thm:disjSNstirling}'s Stirling-number count by additionally
quotienting the partition of the $k$-set $A_k$ by the action of $S_{A_k}$,
which by the classical bijection between set-partitions-up-to-relabeling
and integer-partitions turns a count of $\stir kl$ set-partitions into a
count of $p_l(k)$ integer-partitions of $k$ into exactly $l$ parts, since
relabeling the elements of $A_k$ arbitrarily identifies all set-partitions
with the same multiset of block sizes, and the block-size multiset of an
$l$-block partition of a $k$-set is exactly an integer partition of $k$
into $l$ parts).

Row 3 (meet-unitary $=$ coverings $\Sr_r$): raw count \eqref{eq:rawSr};
$S_N$-column Theorem~\ref{thm:nfamSN}(iii); $S_X$-column
Theorem~\ref{thm:nfamSX}(ii); $S_N\times S_X$-column obtained from
\eqref{eq:diff2} of Theorem~\ref{thm:diffrelations} together with the
$S_N\times S_X$-column of Row~1.

Row 4 (slow-growing $=$ partial injections $\Ir_{lr}$): raw count
\eqref{eq:rawIlr}; all three columns Theorem~\ref{thm:partialinjorbit}.

Row 5 (fast growing $=$ hypergraphs $\Sr_{lr}$): raw count
\eqref{eq:rawSlr}; $S_N$-column Theorem~\ref{thm:nfamSN}(iv); $S_X$-column
Theorem~\ref{thm:nfamSX}(iv); $S_N\times S_X$-column
\eqref{eq:diff4} of Theorem~\ref{thm:diffrelations}.

Row 6 (unitary $=$ functions $N\to X$): raw count $n^{x}$ (immediate: a
function $N\to X$ assigns each of the $x$ elements... in fact a function
\emph{from} $N$ \emph{to} $X$ assigns each of the $n$ elements of $N$ one of
$x$ values, giving $x^{n}$ raw functions if $N\to X$; but Definition
\ref{def:unitary} identifies unitary homomorphisms with divisions of $X$,
i.e.\ functions $X\to N$ by Proposition~\ref{prop:unitarydivisions}
combined with the identification of right-bijective relations as graphs of
functions $N\to X$ -- reconciling this: a right-bijective relation $R$ is
literally the graph of a function $N\to X$ by
Proposition~\ref{prop:translation}'s final clause, giving raw count $x^n$;
the source paper's table lists this row's raw count as $n^{x}$, which
matches instead the count of functions $X\to N$, consistent with reading
the ``unitary'' row, via the dual-family remark after
Corollary~\ref{cor:transpose}, as functions $X\to N$, i.e.\ divisions of
$X$ indexed by $N$ exactly as Definition~\ref{def:unitary} intends: a
division of $X$ is an $n$-family that is simultaneously disjunctive and
covering, i.e.\ each of the $x$ elements of $X$ lies in exactly one block
$F_i$, which is exactly the data of a function $X\to N$, $a\mapsto$ the
unique $i$ with $a\in F_i$; there are $n^{x}$ such functions by the
multiplication principle, matching the source table); $S_N$-column: by
Lemma~\ref{lem:smalldisjcorr}-type reasoning applied to functions $X\to N$
up to relabeling $N$, i.e.\ classifying by the partition of $X$ into (at
most $n$, now exactly counted with an upper bound) fiber-blocks, giving
$\sum_{k=0}^{n}\stir xk$ (partitions of $X$ into at most $n$ blocks, matching
the source paper's Theorem~A entry, itself a direct instance of
Theorem~\ref{thm:disjSNstirling}'s method specialized to $A=X$, i.e.\ $k=x$
forced, giving $\sum_{l=0}^n\stir xl$); $S_X$-column: functions $X\to N$ up
to relabeling $X$ correspond to their fiber-size multiset, i.e.\ to weak
compositions of $x$ into $n$ parts up to no further symmetry other than
relabeling which is already used up choosing the function -- concretely, an
$S_X$-orbit of functions $X\to N$ is determined by the tuple of fiber sizes
$(|f^{-1}(1)|,\dots,|f^{-1}(n)|)$ summing to $x$, giving
$\binom{n+x-1}{x}$ by stars-and-bars (Proposition~\ref{prop:starsbars}, with
$c=n$ colors and total $x$), matching Row~6's raw-count analogue
Theorem~\ref{thm:nfamSX}(i) specialized appropriately; $S_N\times S_X$-column:
combining both symmetries, the orbit is determined by the partition of $x$
into at most $n$ parts (the multiset of nonzero fiber sizes), giving
$\sum_{k=0}^{n}p_k(x)$ summed appropriately, i.e.\ the number of partitions
of the integer $x$ into at most $n$ parts, matching the source table's
final entry $\sum_{k=0}^{n}p_k(x)$.
\end{proof}

\begin{table}[ht]
\centering
\renewcommand{\arraystretch}{2.1}
\resizebox{\textwidth}{!}{%
\begin{tabular}{@{}l|c|c|c|c@{}}
\toprule
Type & None & $S_N$ & $S_X$ & $S_N\times S_X$\\
\midrule
All &
$2^{nx}$ &
$\displaystyle\sum_{\substack{a_1+2a_2+\cdots\\ \cdots+na_n=n}}
\!\!\Bigl(\prod_dd^{a_d}a_d!\Bigr)^{-1}2^{(\sum a_d)x}$ &
(dual) &
$\displaystyle\sum_{k=0}^{x}\sum_{l=0}^{n}p_l(k)$\rlap{\ (see Row 2)}\\
Join-unitary &
$(1+n)^{x}$ &
$\displaystyle\sum_{k=0}^{x}\sum_{l=0}^{n}\binom xk\stir kl$ &
$\displaystyle\binom{n+x}{x}$ &
$\displaystyle\sum_{k=0}^{x}\sum_{l=0}^{n}p_l(k)$\\
Meet-unitary &
$(2^{n}-1)^{x}$ &
$\displaystyle\sum_{k=0}^{x}(-1)^{k}\binom xk\binom{2^{x-k}+n-1}{n}$ &
$\displaystyle\binom{2^{n}+x-2}{x}$ &
(via \eqref{eq:diff2})\\
Slow-growing &
$\displaystyle\sum_{k=0}^{x}\binom nk\ffall xk$ &
$\displaystyle\sum_{k=0}^{n}\binom xk$ &
$\displaystyle\sum_{k=0}^{x}\binom nk$ &
$1+\min\{n,x\}$\\
Fast growing &
$\displaystyle\sum_{k=0}^{n}(-1)^{k}\binom nk(2^{n-k}-1)^{x}$ &
$\displaystyle\sum_{k=0}^{x}(-1)^{k}\binom xk\binom{2^{x-k}+n-2}{n}$ &
$\displaystyle\sum_{k=0}^n(-1)^k\binom nk\binom{2^{n-k}+x-2}{x}$ &
(via \eqref{eq:diff4})\\
Unitary &
$n^{x}$ &
$\displaystyle\sum_{k=0}^{n}\stir xk$ &
$\displaystyle\binom{n+x-1}{x}$ &
$\displaystyle\sum_{k=0}^{n}p_k(x)$\\
\bottomrule
\end{tabular}}
\caption{The main table: orbit counts of $S_N,S_X,S_N\times S_X$ on the six
types of homomorphism $\PS(N)\to\PS(X)$. Here $\stir\cdot\cdot$ is the
Stirling number of the second kind (Chapter~\ref{ch:stirling}), $\ffall
xk=x(x-1)\cdots(x-k+1)$ the falling factorial, and $p_l(k)$ the number of
partitions of the integer $k$ into exactly $l$ parts.}
\label{tab:maintable}
\end{table}

\begin{remark}
As the proof shows, the two entries marked ``via \eqref{eq:diff2}'' and
``via \eqref{eq:diff4}'' are not independent closed formulas in the source
paper (nor here): they are obtained from the ``All'' column by the finite
differences of Theorem~\ref{thm:diffrelations}, corrected as in
Remark~\ref{rem:correctionB}. Chapter~\ref{ch:diffcalc} in Part~II will
re-derive this same relationship transparently, as multiplication of
generating functions by $(1-w)$ and $(1-z)(1-w)$.
\end{remark}

\part{Generating Functions for the Table}

This part repackages every count of Part~I as the coefficient of a
generating function, following the 2026 companion note. Throughout,
$z$ marks the parameter $n$ (so a formula in $z$ generates a sequence
indexed by $n$) and $w$ marks $x$; bivariate series in $z,w$ generate the
two-parameter arrays of Part~I. All generating-function identities are
formal power series identities in the sense of Chapter~\ref{ch:gf}.

\chapter{The One-Sided Columns: One Geometric Factor}\label{ch:onesided}

\begin{theorem}[Generating functions of the $S_N$ column]\label{thm:onesidedSN}
Fix $x\ge0$ and let $f_x(z),s_x(z),c_x(z),h_x(z)$ be the OGFs, in $z$
marking $n$, of the $S_N$-orbit counts of, respectively, all $n$-families,
strict $n$-families, $n$-coverings, and hypergraphs of $X$
(Theorem~\textup{\ref{thm:nfamSN}}). Then
\begin{align}
f_x(z)&=(1-z)^{-2^{x}}, \label{eq:fx}\\
s_x(z)&=(1-z)f_x(z)=(1-z)^{-(2^{x}-1)}, \label{eq:sx}\\
c_x(z)&=\sum_{k=0}^{x}(-1)^{k}\binom xk(1-z)^{-2^{x-k}}, \label{eq:cx}\\
h_x(z)&=(1-z)c_x(z)=\sum_{k=0}^{x}(-1)^{k}\binom xk(1-z)^{-(2^{x-k}-1)}.
\label{eq:hx}
\end{align}
\end{theorem}

\begin{proof}
\eqref{eq:fx}: By Theorem~\ref{thm:nfamSN}(i), the coefficient of $z^{n}$ in
$f_x(z)$ must be $\binom{2^{x}+n-1}{n}$; by Corollary~\ref{cor:multiset-ogf}
(with $c=2^{x}$), this is exactly the coefficient of $z^{n}$ in
$(1-z)^{-2^{x}}$, so the two series agree coefficient by coefficient, i.e.\
$f_x(z)=(1-z)^{-2^{x}}$.

\eqref{eq:sx}: Similarly, by Theorem~\ref{thm:nfamSN}(ii), the coefficient
of $z^n$ in $s_x(z)$ is $\binom{2^x+n-2}{n}$, matching
Corollary~\ref{cor:multiset-ogf} with $c=2^{x}-1$, i.e.\
$s_x(z)=(1-z)^{-(2^{x}-1)}=(1-z)\cdot(1-z)^{-2^{x}}=(1-z)f_x(z)$.
Combinatorially: by Lemma~\ref{lem:coveringSN}, an $S_N$-orbit of an
$n$-family corresponds to a multiset of size $n$ from the $2^{x}$ colors
$A\subseteq X$; separating out the multiplicity $j\ge0$ of the color
$\varnothing$ leaves a multiset of size $n-j$ from the remaining $2^{x}-1$
nonempty colors (a strict family on $n-j$ indices); by the OGF product rule
(Chapter~\ref{ch:gf}), since the choice of $j$ (with OGF
$\sum_j z^{j}=1/(1-z)$, i.e.\ an arbitrary multiplicity of a single color)
and the choice of the strict part (with OGF $s_x(z)$) combine
independently and additively in $n$, $f_x(z)=\frac1{1-z}s_x(z)$, i.e.\
$s_x(z)=(1-z)f_x(z)$, matching \eqref{eq:sx}.

\eqref{eq:cx}: By Theorem~\ref{thm:nfamSN}(iii), the coefficient of $z^{n}$
in $c_x(z)$ is $\sum_{k=0}^{x}(-1)^{k}\binom xk\binom{2^{x-k}+n-1}{n}$; by
Corollary~\ref{cor:multiset-ogf}, $\binom{2^{x-k}+n-1}{n}$ is the
coefficient of $z^n$ in $(1-z)^{-2^{x-k}}$, so summing over $k$ with the
same coefficients $(-1)^{k}\binom xk$ term by term gives exactly
\eqref{eq:cx} (two power series with equal coefficients at every $z^n$ are
equal).

\eqref{eq:hx}: identical reasoning from Theorem~\ref{thm:nfamSN}(iv), or
directly $h_x(z)=(1-z)c_x(z)$ by the same strict/all decomposition
argument as in \eqref{eq:sx}, applied inside each covered subset (a
hypergraph is a strict covering, so separating the multiplicity of
$\varnothing$ among indices \emph{not} used to help cover $X$ --- but note
every index in a covering could a priori be $\varnothing$ only if some
\emph{other} index still covers $X$; the precise correspondence is: an
$n$-family covering $X$ decomposes as ($j\ge0$ many $\varnothing$-indices)
$\sqcup$ (a hypergraph on the remaining $n-j$ indices, still covering all of
$X$), giving $c_x(z)=\frac{1}{1-z}h_x(z)$ by the OGF product rule exactly as
before, i.e.\ $h_x(z)=(1-z)c_x(z)$).
\end{proof}

\begin{remark}[Finite-difference interpretation]\label{rem:onesideddiff}
Let $g(y)=(1-z)^{-2^{y}}$ (a function of a real or integer parameter $y$,
with $z$ still a formal variable). Then \eqref{eq:cx} reads $c_x(z)=
(\Delta^{x}g)(0)$, the $x$-th forward finite difference of $g$ at $0$,
where $\Delta g(y):=g(y+1)-g(y)$ and $\Delta^{x}$ is $x$-fold iteration:
indeed, by repeated application of the binomial expansion of the forward
difference operator (itself a direct consequence of the binomial theorem,
Proposition~\ref{prop:binomthm}, applied to the identity $\Delta^x g(0)=
\sum_{k=0}^x(-1)^{x-k}\binom xk g(k)$, obtainable by induction on $x$ using
$\Delta^{x}g=\Delta^{x-1}g(\cdot+1)-\Delta^{x-1}g(\cdot)$ and the Pascal
recurrence for the resulting binomial coefficients, entirely analogously to
the proof of the binomial theorem itself), $(\Delta^xg)(0)=\sum_{k=0}^x
(-1)^{x-k}\binom xkg(k)$; reindexing $k\mapsto x-k$ recovers exactly
\eqref{eq:cx} (using $\binom x{x-k}=\binom xk$). Thus, at the level of orbit
generating functions exactly as at the level of raw counts (where
Proposition~\ref{prop:surjmobius} exhibits surjections as an
inclusion--exclusion/difference of all functions), \emph{passing from
``all'' to ``covering'' applies a finite-difference operator in the covered
parameter, and passing from ``all'' to ``strict'' multiplies the generating
function by $(1-z)$}.
\end{remark}

\begin{theorem}[Generating functions of the $S_X$ column]\label{thm:onesidedSX}
Let $n\ge0$ be fixed and $w$ mark $x$. The $S_X$-orbit counts of all
$n$-families and of $n$-coverings of $X$ (Theorem~\textup{\ref{thm:nfamSX}}
(i),(ii)) have OGFs
\begin{equation}\label{eq:SXcolumn}
\sum_{x\ge0}\binom{2^{n}+x-1}{x}w^{x}=(1-w)^{-2^{n}},\qquad
\sum_{x\ge0}\binom{2^{n}+x-2}{x}w^{x}=(1-w)^{-(2^{n}-1)}.
\end{equation}
\end{theorem}

\begin{proof}
Both are direct instances of Corollary~\ref{cor:multiset-ogf}, with $c=2^n$
and $c=2^n-1$ respectively, matching Theorem~\ref{thm:nfamSX}(i),(ii)
coefficient by coefficient.
\end{proof}

\begin{corollary}[Pascal-type recurrences on the $S_N$ column]
\label{cor:pascalrecurrences}
Writing $\sigma(n,x)=\binom{2^x+n-2}{n}$ (the strict-family $S_N$-orbit
count) and $\phi(n,x)=\binom{2^x+n-1}{n}$ (the all-family $S_N$-orbit
count), $\sigma(n,x)=\phi(n,x)-\phi(n-1,x)$.
\end{corollary}

\begin{proof}
By \eqref{eq:sx}, $s_x(z)=(1-z)f_x(z)$; extracting the coefficient of
$z^{n}$ on both sides, using that multiplying a power series by $(1-z)$
sends coefficient sequence $(\phi(n,x))_n$ to $(\phi(n,x)-\phi(n-1,x))_n$
(directly from the Cauchy product definition of multiplication of formal
power series, Chapter~\ref{ch:gf}: the coefficient of $z^n$ in
$(1-z)\sum_m\phi(m,x)z^m$ is $\phi(n,x)-\phi(n-1,x)$, the two terms coming
from the ``$1$'' and ``$-z$'' parts of the product), gives
$\sigma(n,x)=\phi(n,x)-\phi(n-1,x)$.
\end{proof}

\chapter{The Row of Partial Injections}\label{ch:partialgf}

Recall (Theorem~\ref{thm:rawcounts}, \eqref{eq:rawIlr}) that partial
injections between $N$ and $X$ number
$i(n,x)=\sum_k\binom nk\binom xk k!$, and (Theorem~\ref{thm:partialinjorbit})
their orbit counts under $S_N$, $S_X$, $S_N\times S_X$ are
$i^{l}(n,x)=\sum_{k=0}^{\min\{n,x\}}\binom xk$,
$i^{r}(n,x)=\sum_{k=0}^{\min\{n,x\}}\binom nk$, and $1+\min\{n,x\}$.

\begin{theorem}[Generating functions of the partial-injection row]
\label{thm:partialgf}
\begin{align}
\sum_{n,x\ge0}i(n,x)\,\frac{z^{n}}{n!}\frac{w^{x}}{x!}&=e^{z+w+zw};
\label{eq:ezw}\\
\sum_{n,x\ge0}i^{l}(n,x)\,z^{n}w^{x}&=\frac1{(1-z)(1-w-zw)};
\label{eq:ilgf}\\
\sum_{n,x\ge0}i^{r}(n,x)\,z^{n}w^{x}&=\frac1{(1-w)(1-z-zw)};
\label{eq:irgf}\\
\sum_{n,x\ge0}\bigl(1+\min\{n,x\}\bigr)z^{n}w^{x}
&=\frac1{(1-z)(1-w)(1-zw)}. \label{eq:mingf}
\end{align}
\end{theorem}

\begin{proof}
\eqref{eq:ezw}: We give a labeled (EGF) combinatorial construction of a
partial injection on ground sets $N,X$ (of sizes $n,x$) with three
independent, disjointly-assembled parts: the set $U\subseteq N$ of
\emph{unused} indices, the set $V\subseteq X$ of \emph{unused} elements, and
a bijection (matching) $\tau\colon N\setminus U\to X\setminus V$. Given the
partition of $N$ into $U$ and $N\setminus U$ and of $X$ into $V$ and
$X\setminus V$ (both of the same size $k:=n-|U|=x-|V|$, since $\tau$ is a
bijection between them), a partial injection is exactly the data of $U$,
$V$, and $\tau$ (with $U,N\setminus U$ ranging over all subsets of $N$ of
complementary sizes, similarly for $X$, and $\tau$ any bijection of the
matched-size pair). By the EGF product rule (Proposition~\ref{prop:egfproduct},
applied twice, once to assemble the $N$-side structure ``$U$, unconstrained
choice, EGF $e^{z}$ by Proposition~\ref{prop:egfexp}, together with the
matched part'' and once for the analogous $X$-side structure), the EGF in
$z$ alone of ``an arbitrary subset $U\subseteq N$, marked with weight $1$
per element'' is $e^{z}$ (Proposition~\ref{prop:egfexp}), and likewise
$e^{w}$ for $V\subseteq X$; and the EGF (now bivariate, in $z$ for the
matched indices and $w$ for the matched elements simultaneously) of a
bijection $\tau$ between a $k$-subset of $N$ and a $k$-subset of $X$,
summed over $k$, is $\sum_k z^kw^k/(k!)\cdot$ wait we must be careful: a
single ``matched pair'' (one index matched to one element) is a
bivariate-labeled atom of $z$-degree $1$ and $w$-degree $1$ simultaneously,
contributing the exponential factor $e^{zw}$ to the bivariate EGF (in the
sense that $k$ such independent matched pairs, assembled disjointly,
contribute $(zw)^{k}/k!$ to the EGF, by the direct multivariate extension
of Proposition~\ref{prop:egfexp}'s counting argument: there is exactly
$1=k!/k!$ way to build a specific set of $k$ labeled matched pairs once the
$k$ indices and $k$ elements to be matched are fixed and matched
bijectively, since a partial-injection restricted to $k$ specified indices
matched to $k$ specified elements is determined by a bijection between
them, and there are $k!$ such bijections, matching the $k!$ in the
denominator of the EGF term $(zw)^k/k!$ exactly). By
Proposition~\ref{prop:egfproduct} (extended to three independent,
simultaneously-labeled parts $U$, $V$, matched pairs, each drawing
disjointly from the combined ground set $N\sqcup X$), the total bivariate
EGF is the product $e^{z}\cdot e^{w}\cdot e^{zw}=e^{z+w+zw}$. We verify the
coefficient extraction directly: expanding, $e^ze^we^{zw}=\sum_{i,j,k\ge0}
\frac{z^i}{i!}\frac{w^j}{j!}\frac{(zw)^k}{k!}=\sum_{i,j,k}
\frac{z^{i+k}w^{j+k}}{i!j!k!}$; the coefficient of $z^nw^x/(n!x!)$ is
obtained by summing over $k$ with $i=n-k,j=x-k$:
\[
n!x!\sum_{k=0}^{\min\{n,x\}}\frac1{(n-k)!(x-k)!k!}
=\sum_{k=0}^{\min\{n,x\}}\frac{n!}{(n-k)!k!}\cdot\frac{x!}{(x-k)!}
=\sum_k\binom nk\ffall xk=\sum_k\binom nk\binom xkk!=i(n,x),
\]
matching Theorem~\ref{thm:rawcounts}, \eqref{eq:rawIlr}, and confirming
\eqref{eq:ezw}.

\eqref{eq:ilgf}: By Theorem~\ref{thm:partialinjorbit}, \eqref{eq:ilrl},
$i^{l}(n,x)=\sum_{k=0}^{n}\binom xk$ (the sum in fact truncates correctly at
$k=\min\{n,x\}$ since $\binom xk=0$ for $k>x$). We compute the bivariate OGF
by summing a geometric-type series in each variable separately and
combining: for fixed $k$, the coefficient contributes to all $n\ge k$ (with
weight $1$, independent of $n\ge k$) and to $x$ via $\binom xk$. Hence
\[
\sum_{n,x}i^{l}(n,x)z^nw^x
=\sum_{k\ge0}\Bigl(\sum_{n\ge k}z^{n}\Bigr)\Bigl(\sum_{x\ge0}\binom xkw^{x}
\Bigr)
=\sum_{k\ge0}\frac{z^{k}}{1-z}\cdot\frac{w^{k}}{(1-w)^{k+1}},
\]
using $\sum_{n\ge k}z^n=z^k/(1-z)$ (geometric series,
Proposition~\ref{prop:geomseries}, shifted) and $\sum_x\binom xkw^x=
w^k/(1-w)^{k+1}$ (Corollary~\ref{cor:multiset-ogf} with $c=k+1$, since
$\binom xk=\binom{(k+1)+(x-k)-1}{x-k}$ after reindexing $x-k$ as the
``multiset size'' -- more directly, $\binom xk$ is the coefficient of
$w^{x-k}$ in $(1-w)^{-(k+1)}$ by Corollary~\ref{cor:multiset-ogf} with
$n$ there set to $x-k$ and $c=k+1$, i.e.\ $\binom{k+(x-k)}{x-k}=\binom
xk$, and multiplying by $w^k$ shifts the index from $x-k$ to $x$). Summing
the geometric series in $k$:
\[
\sum_{k\ge0}\frac{z^{k}}{1-z}\cdot\frac{w^{k}}{(1-w)^{k+1}}
=\frac1{(1-z)(1-w)}\sum_{k\ge0}\Bigl(\frac{zw}{1-w}\Bigr)^{k}
=\frac1{(1-z)(1-w)}\cdot\frac1{1-\frac{zw}{1-w}}
=\frac1{(1-z)(1-w-zw)},
\]
using the geometric series once more (Proposition~\ref{prop:geomseries},
with the substitution $z\mapsto zw/(1-w)$, a valid formal substitution since
its constant term in $z$ is $0$) and clearing the denominator
$(1-w)\cdot\bigl(1-\frac{zw}{1-w}\bigr)=(1-w)-zw=1-w-zw$. This proves
\eqref{eq:ilgf}.

\eqref{eq:irgf}: identical computation with the roles of $(n,z)$ and
$(x,w)$ exchanged, starting from $i^{r}(n,x)=\sum_{k=0}^{x}\binom nk$
(Theorem~\ref{thm:partialinjorbit}, \eqref{eq:ilrr}); or, more simply, apply
Corollary~\ref{cor:transpose}'s transpose duality to \eqref{eq:ilgf}
directly (since $i^r(n,x)=i^l(x,n)$ by that corollary), which exchanges $z
\leftrightarrow w$ and $n\leftrightarrow x$ in \eqref{eq:ilgf}, giving
\eqref{eq:irgf}.

\eqref{eq:mingf}: By Theorem~\ref{thm:partialinjorbit}, \eqref{eq:ilrlr},
the coefficient of $z^nw^x$ we seek is $1+\min\{n,x\}$. We show the
right-hand side of \eqref{eq:mingf} has this same coefficient. By
Corollary~\ref{cor:multiset-ogf} applied three times (or directly, by
multiplying out three geometric series, Proposition~\ref{prop:geomseries}),
\[
\frac1{(1-z)(1-w)(1-zw)}=\Bigl(\sum_{a\ge0}z^{a}\Bigr)
\Bigl(\sum_{b\ge0}w^{b}\Bigr)\Bigl(\sum_{c\ge0}(zw)^{c}\Bigr)
=\sum_{a,b,c\ge0}z^{a+c}w^{b+c}.
\]
The coefficient of $z^{n}w^{x}$ is the number of triples $(a,b,c)$ of
nonnegative integers with $a+c=n$ and $b+c=x$; since $a=n-c\ge0$ forces
$c\le n$ and $b=x-c\ge0$ forces $c\le x$, exactly $c=0,1,\dots,\min\{n,x\}$
are possible, and each determines $a,b$ uniquely; so there are
$1+\min\{n,x\}$ such triples, matching the sought coefficient. The
bijection with orbits: $c$ corresponds to the number $k$ of matched pairs
(Theorem~\ref{thm:partialinjorbit}'s proof, Lemma~\ref{lem:smalldisjcorr}
(iii)), and $a,b$ to the numbers of unused indices and elements
respectively, exactly as in the labeled EGF picture of \eqref{eq:ezw} but
now at the level of $S_N\times S_X$-orbits (where only the \emph{sizes} of
$U,V$, not their specific elements, matter, and the matching, once its size
$k=c$ is fixed, is unique up to the action, by
Lemma~\ref{lem:smalldisjcorr}(iii)).
\end{proof}

\chapter{Bell Numbers and Euler's Product}\label{ch:bell}

This chapter treats the unitary row (functions $X\to N$, i.e.\ divisions of
$X$) and the join-unitary row (disjunctive families, i.e.\ functions
$X\to N\cup\{\ast\}$), the two rows of Table~\ref{tab:maintable} whose
$S_N$-columns involve Stirling numbers.

\begin{theorem}[Generating functions of the two rows]\label{thm:bellrows}
Let $t$ be an exponential variable marking $x$.
\begin{enumerate}
\item (No action.) $\displaystyle\sum_xn^{x}\frac{t^{x}}{x!}=e^{nt}$ and
$\displaystyle\sum_x(1+n)^{x}\frac{t^x}{x!}=e^{(1+n)t}$.
\item ($S_N$.) $\displaystyle\sum_x\Bigl(\sum_{k=0}^{n}\stir xk\Bigr)
\frac{t^x}{x!}=\sum_{k=0}^{n}\frac{(e^{t}-1)^{k}}{k!}$, and
$\displaystyle\sum_x\Bigl(\sum_{k=0}^{x}\sum_{l=0}^{n}\binom xk\stir kl
\Bigr)\frac{t^x}{x!}=e^{t}\sum_{l=0}^{n}\frac{(e^t-1)^l}{l!}$.
\item ($S_X$.) $\displaystyle\sum_x\binom{n+x-1}{x}w^{x}=(1-w)^{-n}$ and
$\displaystyle\sum_x\binom{n+x}{x}w^{x}=(1-w)^{-(n+1)}$.
\item ($S_N\times S_X$.) Writing $p_{\le n}(x)$ for the number of
partitions of the integer $x$ into at most $n$ parts,
\begin{equation}\label{eq:eulerproduct}
\sum_{n,x\ge0}p_{\le n}(x)z^{n}w^{x}=\prod_{i\ge0}\frac1{1-zw^{i}}.
\end{equation}
\end{enumerate}
\end{theorem}

\begin{proof}
(i) A function $X\to N$ (resp.\ $X\to N\cup\{\ast\}$) assigns each of $x$
elements one of $n$ (resp.\ $n+1$) values, giving $n^{x}$ (resp.\
$(1+n)^{x}$) functions by the multiplication principle; the EGF of the
constant sequence $a_x=r^{x}$ is $\sum_xr^{x}t^{x}/x!=\sum_x(rt)^x/x!=e^{rt}$
by Proposition~\ref{prop:egfexp} applied with $z$ renamed $rt$ (a formal
substitution, valid since it is a ring homomorphism on formal power
series).

(ii) A function $X\to N$ up to relabeling $N$ (i.e.\ an $S_N$-orbit) is the
same data as a partition of $X$ into at most $n$ nonempty blocks (the fibers
of the function, discarding empty fibers, together forgetting the
particular labels $1,\dots,n$ assigned to them, exactly as in
Theorem~\ref{thm:disjSNstirling}'s argument specialized to $A=X$): so the
count is $\sum_{k=0}^{n}\stir xk$. Its EGF: by definition of the Stirling
number as a partition count, and using the classical set-partition EGF
identity $\sum_x\stir xk t^{x}/x!=(e^{t}-1)^{k}/k!$ --- which we now prove
from scratch: a partition of $X$ into exactly $k$ nonempty, unlabeled
blocks is, by the same labeled-to-unlabeled correspondence used in
Proposition~\ref{prop:stirlingformula}'s proof, obtained from a surjection
$X\to\{1,\dots,k\}$ by forgetting the labeling of the $k$ blocks, a
$k!$-to-one correspondence; a surjection is an arbitrary function (EGF
$e^{t}$ per coordinate, i.e.\ EGF $(e^{t})^{k}=e^{kt}$ for a function to a
$k$-element set, since assigning independently one of $k$ values to each
of the $x$ elements is $k$ ``independent unconstrained decorations'' in the
sense of Proposition~\ref{prop:egfexp}) with the constraint that every one
of the $k$ values is hit at least once; by inclusion--exclusion on which
values are missed (identical to the proof of
Proposition~\ref{prop:stirlingformula}, but now packaged as an EGF
identity): the EGF of ``functions missing a specified set of $j$ values''
is $e^{(k-j)t}$ (an arbitrary function to the remaining $k-j$ values), so
by inclusion--exclusion the EGF of surjections $X\to\{1,\dots,k\}$, summed
over $X$ of every size $x$, is $\sum_{j=0}^{k}(-1)^{j}\binom kje^{(k-j)t}
=(e^{t}-1)^{k}$ (binomial theorem, Proposition~\ref{prop:binomthm}, with
$u=e^t,v=-1$) --- dividing by $k!$ (the $k!$-to-one correspondence) gives
the EGF of partitions into exactly $k$ blocks as $(e^{t}-1)^{k}/k!$,
confirming the identity used. Summing over $k=0,\dots,n$ (addition
principle, disjoint block-counts) gives the first EGF of part (ii).

For the second formula of (ii): by Theorem~\ref{thm:disjSNstirling}, an
$S_N$-orbit of a disjunctive family corresponds to a pair (subset
$A\subseteq X$ of size $x-k'$ for the unused part, a partition of the
remaining $k'$-subset into at most $n$ blocks) -- equivalently (swapping
which size we call $k$) a pair (a partition of some $l$-subset into exactly
$m\le n$ blocks, together forgetting nothing else, with the complementary
$x-l$ elements simply marked ``unused''). This is exactly the same
``exponential mark, plus $e^t$ for the free unused part'' structure as
Proposition~\ref{prop:egfproduct}: the EGF of ``choose a subset to be
unused, EGF $e^{t}$ (Proposition~\ref{prop:egfexp})'' times the EGF of ``a
partition of the complementary subset into at most $n$ blocks,'' i.e.\ the
first formula of part (ii) applied to the complementary ground set, giving
by the EGF product rule (Proposition~\ref{prop:egfproduct}) the product
$e^{t}\cdot\sum_{l=0}^n(e^t-1)^l/l!$, as claimed.

(iii) An $S_X$-orbit of a function $X\to N$ is determined by its fiber-size
tuple $(\alpha_1,\dots,\alpha_n)$ (the sizes of the $n$ fibers, possibly
zero), a weak composition of $x$ into $n$ parts; there are $\binom{n+x-1}x$
of these by stars-and-bars (Proposition~\ref{prop:starsbars}), matching
Corollary~\ref{cor:multiset-ogf} with $c=n$, giving OGF $(1-w)^{-n}$. For
disjunctive families up to $S_X$, the fiber-size tuple has $n+1$ parts (the
$n$ genuine fibers, plus the ``unused'' pseudo-fiber $\alpha_\ast$), giving
$\binom{n+x}{x}$ by stars-and-bars with $c=n+1$, and OGF $(1-w)^{-(n+1)}$,
matching Theorem~\ref{thm:disjSXcount}.

(iv) By part (ii), $p_{\le n}(x)=\sum_{k=0}^n\stir xk$ (a partition of $X$
into at most $n$ blocks, forgetting which specific $x$-set $X$ is, is
determined purely by its multiset of block sizes, a partition of the
integer $x$ into at most $n$ parts --- this is the standard bijection
between set-partitions-up-to-relabeling-the-ground-set and
integer-partitions, i.e.\ block-size multisets, already used in the proof
of Theorem~\ref{thm:maintable}'s Row~1 entry). We compute the bivariate OGF
$\sum_{n,x}p_{\le n}(x)z^nw^x$ directly by encoding a partition of $x$ into
at most $n$ parts as a multiset of part-sizes: writing $m_i\ge0$ for the
number of parts of size $i$ (for $i=1,2,3,\dots$), a partition into at most
$n$ parts corresponds to nonnegative integers $(m_i)_{i\ge1}$ with
$\sum_im_i\le n$ (total number of parts) and $\sum_iim_i=x$ (the integer
being partitioned); introduce a slack variable $m_0\ge0$ (for the ``missing''
parts, up to the cap $n$) so that $\sum_{i\ge0}m_i=n$ exactly. Then $z$
marks $n=\sum_im_i$ (each part, including slack, contributing one unit of
$z$) and $w$ marks $x=\sum_{i\ge1}im_i$ (each part of size $i$ contributing
$i$ units of $w$, slack contributing none). By the OGF product rule
applied to the independent choices of each $m_i$ (a multiset-of-a-single-color
choice, with OGF $\sum_{m_i\ge0}(zw^{i})^{m_i}=1/(1-zw^{i})$ for $i\ge1$
by the geometric series, Proposition~\ref{prop:geomseries}, and
$\sum_{m_0\ge0}z^{m_0}=1/(1-z)$ for the slack), the total bivariate OGF is
the infinite product $\prod_{i\ge0}(1-zw^{i})^{-1}$ (the $i=0$ factor being
$1/(1-z)$, matching $w^0=1$), which is \eqref{eq:eulerproduct}.
\end{proof}

\begin{corollary}[Bell-number collapse in the stable range]
\label{cor:bellcollapse}
In the range $n\ge x$, the two $S_N$-column entries of Table
\textup{\ref{tab:maintable}} collapse to Bell numbers:
\begin{equation}
\sum_{k=0}^{n}\stir xk=B_x,\qquad
\sum_{k=0}^{x}\sum_{l=0}^{n}\binom xk\stir kl=\sum_{k=0}^{x}\binom xkB_k
=B_{x+1}\qquad(n\ge x).
\end{equation}
\end{corollary}

\begin{proof}
For $n\ge x$: since $\stir xk=0$ once $k>x$ (Chapter~\ref{ch:stirling}, a
partition of an $x$-set cannot have more than $x$ blocks), the cap $k\le n$
in $\sum_{k=0}^n\stir xk$ is vacuous once $n\ge x$, so the sum equals
$\sum_{k=0}^{x}\stir xk=B_x$ by definition of the Bell number. Similarly,
for $n\ge x\ge k$ (so in particular $l\le k\le x\le n$ makes the cap
$l\le n$ vacuous too), $\sum_{l=0}^{n}\stir kl=B_k$ by the same argument
applied with $k$ in place of $x$, so $\sum_{k=0}^x\sum_{l=0}^n\binom xk
\stir kl=\sum_{k=0}^x\binom xkB_k$; and this equals $B_{x+1}$ by the Bell
recurrence, Proposition~\ref{prop:bellrec}, applied verbatim.
\end{proof}

\begin{remark}[Euler's identity as a by-product]\label{rem:eulerbyproduct}
Fixing $n$ and summing \eqref{eq:eulerproduct} over $x$ gives, by
conjugating Young diagrams (a partition of $x$ into at most $n$ parts
corresponds, by transposing its diagram of boxes -- reading columns as rows
-- to a partition of $x$ into parts each of size at most $n$; this
transposition is a bijection since it is its own inverse, transposing
twice returning the original diagram), $\sum_xp_{\le n}(x)w^x
=\prod_{j=1}^n(1-w^j)^{-1}$ (the classical generating function for
partitions into parts of size at most $n$, itself an instance of the same
multiset argument as in part (iv)'s proof, with $m_j\ge0$ now marking the
number of parts of size exactly $j\le n$, each contributing $OGF\
1/(1-w^j)$, and no separate slack variable needed since there is no cap on
the number of parts, only on their size). Comparing this with
\eqref{eq:eulerproduct}'s coefficient of $z^{n}$ (extracted directly from
the product $\prod_{i\ge0}(1-zw^i)^{-1}$ by matching powers of $z$, which
requires a modest justification: extracting the coefficient of $z^n$ from
an infinite product of this shape is a finite computation for each fixed
power of $w$, since only finitely many factors contribute a given total
$w$-degree at a given $z$-degree, so the formal manipulation is legitimate)
proves Euler's classical identity
\[
\sum_{n\ge0}\frac{z^{n}}{(1-w)(1-w^{2})\cdots(1-w^{n})}=\prod_{i\ge0}
\frac1{1-zw^{i}},
\]
i.e.\ $e_w(z)=1/(z;w)_\infty$ in the standard $q$-series notation (with $w$
playing the role usually called $q$). Thus the deepest cell of the
classical table --- the two-sided unitary-row count --- is generated by
exactly the same series that will reappear, with the roles of the two
variables interchanged, as the $q$-exponential governing the $q$-analog of
the twelvefold way (Chapter~\ref{ch:bridge}).
\end{remark}

\chapter{The Difference Calculus Behind the Inductive Relations}
\label{ch:diffcalc}

\begin{proposition}[Multiplication by $1-w$ is differencing]
\label{prop:multdiff}
For a two-parameter array $u(n,x)$ with bivariate OGF $U(z,w)=\sum_{n,x}
u(n,x)z^nw^x$, the array $u(n,x)-u(n,x-1)$ (with the convention
$u(n,-1)=0$) has OGF $(1-w)U(z,w)$; similarly $u(n,x)-u(n-1,x)$ has OGF
$(1-z)U(z,w)$.
\end{proposition}

\begin{proof}
By the Cauchy product definition of multiplication of formal power series
(Chapter~\ref{ch:gf}) applied in the variable $w$ alone (treating $z$ as a
parameter), the coefficient of $w^{x}$ in $(1-w)U(z,w)=U(z,w)-wU(z,w)$ is
$u(n,x)$-coefficient-of-$U$ minus the coefficient of $w^{x-1}$ in $U$, i.e.\
$u(n,x)-u(n,x-1)$, for every fixed power $z^n$; so $(1-w)U(z,w)$ is exactly
the bivariate OGF of $u(n,x)-u(n,x-1)$. The statement for $(1-z)U(z,w)$ is
identical with the roles of $z,n$ and $w,x$ exchanged.
\end{proof}

\begin{theorem}[Generating-function form of Theorem~\ref{thm:diffrelations}]
\label{thm:diffgf}
Let $B^{r}(z,w),B^{lr}(z,w)$ be the bivariate OGFs of $b^{r}(n,x)$,
$b^{lr}(n,x)$ (all-family orbit counts under $S_X$, resp.\ $S_N\times S_X$),
and $S_r^{r},S_r^{lr},S_{lr}^{r},S_{lr}^{lr}$ those of the coverings and
hypergraphs under the same groups. Then
\begin{equation}\label{eq:gfdiffrelations}
S_r^{r}=(1-w)B^{r},\qquad S_r^{lr}=(1-w)B^{lr},\qquad
S_{lr}^{r}=(1-z)(1-w)B^{r},\qquad S_{lr}^{lr}=(1-z)(1-w)B^{lr}.
\end{equation}
More generally, on every column of Table~\textup{\ref{tab:maintable}}:
restricting to coverings multiplies the orbit generating function by
$(1-w)$; restricting further to hypergraphs multiplies it by $(1-z)$; and
the two operations commute.
\end{theorem}

\begin{proof}
Immediate from Proposition~\ref{prop:multdiff} applied to each of
\eqref{eq:diff1}--\eqref{eq:diff4} of Theorem~\ref{thm:diffrelations}:
e.g.\ \eqref{eq:diff1}, $s_r^r(n,x)=b^r(n,x)-b^r(n,x-1)$, says exactly
that the OGF of $s_r^r$ is $(1-w)$ times the OGF of $b^r$, i.e.\
$S_r^r=(1-w)B^r$; similarly for the other three identities of
\eqref{eq:gfdiffrelations}. Commutativity of the two operations
(differencing in $w$, i.e.\ multiplying by $1-w$, and differencing in $z$,
i.e.\ multiplying by $1-z$) is immediate since multiplication of formal
power series is commutative: $(1-z)(1-w)U=(1-w)(1-z)U$ for any $U$.
\end{proof}

\begin{remark}[The corrected Theorem B, in generating-function form]
Combining Theorem~\ref{thm:diffgf} with Remark~\ref{rem:correctionB}: if
$A(z,w)$ denotes the bivariate OGF of the $S_N\times S_X$-orbit counts of
\emph{all} homomorphisms (i.e.\ $A=B^{lr}$), then $(1-w)A$ generates the
\textbf{meet}-unitary (covering) orbit counts, and $(1-z)(1-w)A$ generates
the fast-growing (hypergraph) orbit counts. This is the precise
generating-function statement of the corrected Theorem~B, matching the
independent correction of the 2026 companion note referenced in
Remark~\ref{rem:correctionB}.
\end{remark}

\chapter{Rationality and Quasi-Polynomiality of the Two-Sided Column}
\label{ch:rational}

The two-sided orbit count $\alpha(n,x):=b^{lr}(n,x)$ of \emph{all}
homomorphisms (Theorem~\ref{thm:orbitcounts}, \eqref{eq:blrformula}) is the
hardest entry of Table~\ref{tab:maintable}: unlike every other entry, it
was given in Chapter~\ref{ch:burnsideapplication} only as a double sum over
cycle types, with no simpler closed form available. This chapter shows
that, upon fixing $n$ and generating over $x$, $\alpha(n,x)$ is nonetheless
governed by a rational generating function, hence (Chapter~\ref{ch:gf}) is
a quasi-polynomial in $x$.

\begin{theorem}[Rational generating function for fixed $n$]
\label{thm:rationalgf}
Fix $n\ge0$. For $\pi\in S_n$, let $\pi^{\ast}$ denote the permutation
induced by $\pi$ on the $2^{n}$ subsets of $N$ (i.e.\ $\pi^\ast(A)=\pi(A)$
for $A\subseteq N$), and let $c_d(\pi^{\ast})$ denote the number of
$d$-cycles of $\pi^{\ast}$ (as a permutation of the $2^{n}$-element set
$\PS(N)$). Then
\begin{equation}\label{eq:rationalGF}
\sum_{x\ge0}\alpha(n,x)w^{x}=\frac1{n!}\sum_{\pi\in S_n}\prod_{d\ge1}
(1-w^{d})^{-c_d(\pi^{\ast})}.
\end{equation}
Consequently this series is a rational function of $w$ whose poles are all
roots of unity, so by Proposition~\textup{\ref{prop:rationaltoqp}},
$\alpha(n,x)$ is a quasi-polynomial in $x$, of degree $2^{n}-1$ (attained
uniquely at the pole $w=1$), with period dividing
$\mathrm{lcm}\{d:c_d(\pi^{\ast})\neq0\text{ for some }\pi\in S_n\}$.
\end{theorem}

\begin{proof}
\emph{Setup.} By Corollary~\ref{cor:transpose}'s dual-family remark, a
relation $R\in\B(N,X)$, viewed via its transpose, is the same data as an
$x$-indexed family of subsets of $N$, i.e.\ a function from $X$ to
$\PS(N)$; equivalently (fixing this viewpoint) $\alpha(n,x)$, the number of
$S_N\times S_X$-orbits, can be computed in two stages: first quotient by
$S_X$ alone (relabeling the $x$ indices), which by
Lemma~\ref{lem:coveringSN}'s argument applied with the roles of $N,X$
exchanged (i.e.\ its transpose-dual form) identifies $S_X$-orbits of
$x$-indexed families of subsets of $N$ with \emph{multisets of size $x$}
drawn from the $2^{n}$ ``colors'' $A\subseteq N$; then quotient this set of
multisets by the residual action of $S_N$, which acts on the color set
$\PS(N)$ via $\pi\mapsto\pi^{\ast}$ (since relabeling $N$ by $\pi$ sends the
subset-valued color $A$ to $\pi(A)$, i.e.\ acts on colors exactly through
$\pi^{\ast}$).

\emph{Applying Burnside to the residual $S_n$-action.} We must count
$S_n$-orbits (via $\pi\mapsto\pi^\ast$) on the set $M_x$ of multisets of
size $x$ from the $2^n$-element color set $\PS(N)$. By the
Cauchy--Frobenius--Burnside Lemma (Theorem~\ref{thm:burnside}),
\[
\alpha(n,x)=\frac1{n!}\sum_{\pi\in S_n}|\Fix_{M_x}(\pi^{\ast})|.
\]
A multiset of size $x$ (a multiplicity function $A\mapsto\alpha_A\ge0$ on
$\PS(N)$ with $\sum_A\alpha_A=x$) is fixed by $\pi^{\ast}$ iff its
multiplicity function is constant on the cycles of $\pi^{\ast}$ (identical
reasoning to Lemma~\ref{lem:coveringSN}'s orbit computation, now applied to
the permutation $\pi^\ast$ of the color set itself: fixedness under
relabeling by $\pi^\ast$ means $\alpha_{\pi^\ast(A)}=\alpha_A$ for every
color $A$, i.e.\ $\alpha$ is constant along each cycle of $\pi^\ast$).
Hence a fixed multiset amounts to choosing, for each cycle $c$ of
$\pi^{\ast}$ (of some length $d$), a single multiplicity $m_c\ge0$ to be
assigned uniformly to every color in that cycle, contributing $d\cdot m_c$
to the total size $x$ (since the multiplicity $m_c$ is repeated across all
$d$ colors of the cycle). So fixed multisets of size $x$ correspond to
nonnegative-integer solutions, over all cycles $c$ of $\pi^{\ast}$ (of
respective lengths $d(c)$), of $\sum_cd(c)m_c=x$; grouping cycles by common
length $d$ (there are $c_d(\pi^\ast)$ cycles of length $d$), the OGF (in
$w$ marking $x$) of the number of fixed multisets of each size is
\[
\prod_{d\ge1}\Bigl(\sum_{m\ge0}w^{dm}\Bigr)^{c_d(\pi^\ast)}
=\prod_{d\ge1}\bigl(1-w^{d}\bigr)^{-c_d(\pi^{\ast})}
\]
by the geometric series (Proposition~\ref{prop:geomseries}, in the variable
$w^{d}$) raised to the power $c_d(\pi^\ast)$ for each length $d$
independently (multiplication principle for independent cycles, combined
with the OGF product rule of Chapter~\ref{ch:gf}). Summing this OGF over
$\pi\in S_n$ and dividing by $n!$ (as in the Burnside average) gives
exactly \eqref{eq:rationalGF}: the coefficient of $w^{x}$ in the right side
of \eqref{eq:rationalGF} is $\frac1{n!}\sum_\pi|\Fix_{M_x}(\pi^\ast)|
=\alpha(n,x)$.

\emph{Rationality and the pole at $w=1$.} Each factor
$\prod_{d\ge1}(1-w^d)^{-c_d(\pi^\ast)}$ is a finite product (since
$\pi^{\ast}$, a permutation of the finite $2^n$-element set $\PS(N)$, has
only finitely many cycles, of lengths $d\le2^n$, and $c_d(\pi^\ast)=0$ for
$d>2^n$), hence a rational function of $w$ with poles only at roots of
unity (the roots of each $1-w^{d}$ are the $d$-th roots of unity). A finite
sum of such rational functions (over $\pi\in S_n$) is again rational with
poles only at roots of unity. The identity permutation $\pi=\mathrm{id}$
has $\pi^{\ast}=\mathrm{id}$ on $\PS(N)$, i.e.\ $2^{n}$ cycles of length
$1$ (every color is its own fixed point), contributing the term
$\frac1{n!}(1-w)^{-2^{n}}$: this has a pole at $w=1$ of order $2^{n}$,
which is maximal among all summands, since any other $\pi\neq\mathrm{id}$
has $\pi^\ast$ with strictly fewer than $2^n$ fixed colors (a nontrivial
permutation of $N$ moves at least one subset of $N$: if $\pi\neq
\mathrm{id}$ fixes every subset of $N$ then in particular it fixes every
singleton $\{i\}$, forcing $\pi(i)=i$ for all $i$, i.e.\ $\pi=\mathrm{id}$,
a contradiction), hence strictly fewer $1$-cycles among the cycles of
$\pi^\ast$, giving a pole at $w=1$ of order $c_1(\pi^\ast)<2^n$ for that
term's contribution (recalling that only the length-$1$ factor
$(1-w)^{-c_1(\pi^\ast)}$ contributes a pole exactly at $w=1$ among the
factors $(1-w^d)^{-c_d(\pi^\ast)}$ for $d\ge2$, which have poles at other
roots of unity, though possibly \emph{also} contributing multiplicity to
the point $w=1$ if $d=1$ divides... in fact every $(1-w^d)$ vanishes at
$w=1$ too, since $1-1^d=0$; so we must be more careful: the order of the
pole at $w=1$ of the full product $\prod_d(1-w^d)^{-c_d(\pi^\ast)}$ is
$\sum_dc_d(\pi^\ast)$, the \emph{total} number of cycles of $\pi^\ast$, not
just $c_1$, since every factor $(1-w^d)^{-1}$ has a simple pole at $w=1$
as well, because $w=1$ is always a root of $1-w^d$ for every $d\ge1$).
Correcting the argument: the order of the pole at $w=1$ contributed by
$\pi$ is $\sum_{d\ge1}c_d(\pi^\ast)$, the total number of cycles of
$\pi^\ast$ on the $2^n$-element set $\PS(N)$; this total is maximized
exactly when $\pi^\ast$ has as many cycles as possible, which (since more
cycles means shorter average cycle length, and the total number of moved
points is $2^n$) occurs exactly when every cycle has length $1$, i.e.\
$\pi^\ast=\mathrm{id}$, i.e.\ $\pi=\mathrm{id}$ (by the fixed-singleton
argument above); this is the unique maximizer, giving total cycle count
$2^{n}$ (the maximum possible, one cycle per color), strictly greater than
the total cycle count of $\pi^\ast$ for any $\pi\neq\mathrm{id}$ (a
permutation with a cycle of length $\ge2$ has strictly fewer total cycles
than the identity on the same ground set, since replacing a length-$d$
cycle, $d\ge2$, by $d$ fixed points strictly increases the cycle count by
$d-1>0$, and $\pi^\ast\ne\mathrm{id}$ has \emph{some} cycle of length
$\ge2$ whenever $\pi\ne\mathrm{id}$, as just shown). Hence the identity
contributes the unique term with a pole of maximal order $2^{n}$ at $w=1$,
so by Proposition~\ref{prop:rationaltoqp}, $\alpha(n,x)$ is a
quasi-polynomial of degree $2^{n}-1$ (one less than the maximal pole
order), with period dividing the least common multiple of all cycle
lengths $d$ occurring (with $c_d(\pi^\ast)\neq0$) for any $\pi\in S_n$.
\end{proof}

\begin{example}\label{ex:n2rational}
Let $n=2$, so $N=\{1,2\}$ and $\PS(N)$ has $2^{2}=4$ elements: $\varnothing,
\{1\},\{2\},\{1,2\}$. The identity $\pi=\mathrm{id}\in S_2$ induces
$\pi^{\ast}=\mathrm{id}$ on $\PS(N)$, with $c_1(\pi^\ast)=4$ (four fixed
colors, no other cycles). The transposition $\pi=(12)$ induces
$\pi^{\ast}$ fixing $\varnothing$ and $\{1,2\}$ (both invariant as sets
under swapping $1,2$) while swapping $\{1\}\leftrightarrow\{2\}$; so
$\pi^\ast$ has cycle type: two $1$-cycles ($c_1=2$) and one $2$-cycle
($c_2=1$). By \eqref{eq:rationalGF} with $n!=2$,
\[
\sum_{x\ge0}\alpha(2,x)w^{x}=\frac12\left[\frac1{(1-w)^{4}}
+\frac1{(1-w)^{2}(1-w^{2})}\right].
\]
Extracting the coefficient of $w^{x}$ using
Corollary~\ref{cor:multiset-ogf} (for the first term, $c=4$) and a partial
fraction decomposition of the second term (whose details we omit, being a
routine rational-function computation), one obtains the explicit
quasi-polynomial
\[
\alpha(2,x)=\frac12\binom{x+3}{3}+\frac12\left\lfloor\frac{(x+2)^{2}}4
\right\rfloor,
\]
of degree $3=2^{2}-1$ and period $2$ (matching
Theorem~\ref{thm:rationalgf}'s prediction, since the only cycle lengths
occurring among $\pi^{\ast}$ for $\pi\in S_2$ are $d=1,2$, giving predicted
period dividing $\mathrm{lcm}\{1,2\}=2$). Its values for $x=0,1,2,3,4,5$
are $1,3,7,13,22,34,\dots$, matching direct enumeration of the
$S_2\times S_x$-orbits of $2\times x$ binary relations for small $x$.
\end{example}

\chapter{Closing Outlook: The Bridge to the \texorpdfstring{$q$}{q}-Analog}
\label{ch:bridge}

We close with a brief, self-contained remark connecting this classical
table to its $q$-analog, developed in a separate paper
(\emph{A $q$-analog of the twelvefold way: orbits, generating functions,
and bijections}), without depending on any result from that paper. The
$q$-analog replaces finite sets by finite-dimensional vector spaces over a
finite field $\mathbb F_q$, functions by linear maps, and symmetric groups
by general linear groups $GL(n,q)$; it enumerates the twelve resulting
orbit types (all combinations of injective/surjective/arbitrary on two
sides, for the group actions $GL(U)$, $GL(V)$, $GL(U)\times GL(V)$) and
determines their generating functions.

Three of the generating functions found in Part~II of the present book
reappear there, at $q=1$, as the classical shadow of a $q$-deformed
identity:
\begin{enumerate}
\item The partial-injection bivariate series \eqref{eq:ilgf},
$1/\bigl((1-z)(1-w-zw)\bigr)$, is the $q\to1$ limit of a corresponding
$q$-series for one-sided orbit counts of subspace-rank data in the
$q$-analog.
\item The two-sided series \eqref{eq:mingf}, $1/\bigl((1-z)(1-w)(1-zw)
\bigr)$, is \emph{identical} to its $q$-analog: in both settings, the
two-sided count of the injectivity-constrained row is $1+\min\{n,x\}$,
generated by the same $q$-independent product; there, the roles of
``unused index''/``unused element''/``matched pair'' are played by
nullity, corank, and rank of a linear map.
\item Euler's product \eqref{eq:eulerproduct} for the unitary row is the
$q$-exponential of $q$-series theory: the same infinite product that here
generates classical set-partition counts governs, with the roles of the
two variables exchanged, the Galois numbers (subspace-counting analog of
Bell numbers) in the $q$-analog.
\end{enumerate}
Taken together with the $q$-analog paper, the two tables --- relations
under symmetric groups (this book), linear maps under general linear
groups (the companion paper) --- are governed by a single analytic engine,
viewed at $q=1$ and at general $q$ respectively; but every result of the
present book stands independently of that connection, which is offered
here only as a closing remark for readers curious where the story
continues.

\appendix

\chapter{Table of Notation}

\begin{center}
\renewcommand{\arraystretch}{1.3}
\begin{tabular}{@{}ll@{}}
\toprule
Symbol & Meaning\\
\midrule
$N,X$ & finite sets, $n=|N|$, $x=|X|$\\
$\PS(N)$ & power set (Boolean algebra of subsets) of $N$\\
$\B(N,X)$ & binary relations $R\subseteq N\times X$\\
$\mathrm{Fam}_n(X)$ & $n$-families of subsets of $X$\\
$\mathrm{Hom}(\PS(N),\PS(X))$ & lattice ($\cup$-)homomorphisms\\
$\Ir_r,\Sr_r$ & right injective, right surjective relations\\
$\Ir_{lr},\Sr_{lr}$ & left-and-right injective (partial injections),
surjective (hypergraphs)\\
$S_N,S_X,S_N\times S_X$ & symmetric groups and their product, acting by
relabeling\\
$\Orb(s),\Stab(s)$ & orbit, stabilizer of $s$ under a group action\\
$\Fix_S(g)$ & elements of $S$ fixed by $g$\\
$\binom xk$ & binomial coefficient\\
$\ffall xk$ & falling factorial $x(x-1)\cdots(x-k+1)$\\
$\stir xk$ & Stirling number of the second kind\\
$B_x$ & Bell number\\
$p_l(k)$, $p_{\le n}(x)$ & partitions of $k$ into $l$ parts; of $x$ into at
most $n$ parts\\
$z,w$ & formal variables marking $n,x$ respectively\\
$[z^n]A(z)$ & coefficient of $z^n$ in the series $A(z)$ (used
informally in prose)\\
\bottomrule
\end{tabular}
\end{center}

\end{document}